\documentclass[a4paper,10pt]{amsart}

\usepackage{amssymb,amsfonts,amsmath,amsthm}
\usepackage{hyperref}
\usepackage[margin=1in]{geometry}
\usepackage{cancel}
\usepackage{mathtools}
\usepackage{shuffle}
\usepackage{tikz}
\usepackage{mathrsfs}
\usepackage{tikz-cd}
\usepackage{verbatim}
\usepackage{xcolor}
\usepackage[all]{xy}
\usepackage{xparse}

\usepackage[normalem]{ulem}

\ExplSyntaxOn
\NewDocumentCommand{\cyc}{ O{\,} m }
{
    (
    \alec_cycle:nn { #1 } { #2 }
    )
}

\seq_new:N \l_alec_cycle_seq
\cs_new_protected:Npn \alec_cycle:nn #1 #2
{
    \seq_set_split:Nnn \l_alec_cycle_seq { , } { #2 }
    \seq_use:Nn \l_alec_cycle_seq { #1 }
}
\ExplSyntaxOff

\DeclareMathOperator{\Lic}{Li}
\DeclareMathOperator{\GL}{GL}
\DeclareMathOperator{\SL}{SL}

\DeclareMathOperator{\Alt}{Alt}

\DeclareMathOperator{\CR}{cr}
\DeclareMathOperator{\R}{r}
\DeclareMathOperator{\sgn}{sgn}

\newcommand{\GR}{\mathrm{Gr}}
\newcommand{\wI}{\widetilde{I}}
\newcommand{\Li}{\mathcal{L}}
\newcommand{\Sl}{\mathcal{S}}
\newcommand{\II}{I}
\newcommand{\Pid}{D}

\newcommand{\symsix}{\text{\bf Sym}_{36}}
\newcommand{\fiveterm}{\text{\bf V}}
\newcommand{\quadrat}{\mathcal{Q}}

\newcommand{\Cyc}{\text{Cyc}}
\newcommand{\FE}{\text{FE}}
\newcommand{\CA}{\mathcal{A}}

\newcommand{\modsh}{\mathrel{\overset{\shuffle}{=}}}
\newcommand{\modalt}{\mathrel{\overset{\mathrm{Alt}_8}{=}}}
\newcommand{\modaltt}{\mathrel{\overset{\mathrm{Alt}_{10}}{=}}}
\newcommand{\modS}{\mathrel{\overset{\mathcal{S}}{=}}}

\newcommand{\Conf}[2]{{\rm Conf}_{#1}(#2)}
\newcommand{\detv}[1]{\Delta(#1)}

\newcommand{\Mod}[1]{\ (\text{\normalfont mod #1})}

\newcommand{\Aa}{\mathcal{A}}

\newcommand{\ZZ}{\mathbb{Z}}
\newcommand{\QQ}{\mathbb{Q}}
\newcommand{\RR}{\mathbb{R}}
\newcommand{\CC}{\mathbb{C}}

\newcommand{\ol}[1]{\overline{#1}}

\newcommand{\sm}{\smallsetminus}

\renewcommand{\phi}{\varphi}

\def\arXiv#1{arXiv:\href{http://arXiv.org/abs/#1}{\texttt{#1}}}

\newtheorem{theorem}{Theorem}
\newtheorem{lemma}[theorem]{Lemma}
\newtheorem{corollary}[theorem]{Corollary}
\newtheorem{conjecture}[theorem]{Conjecture}
\newtheorem{proposition}[theorem]{Proposition}

\theoremstyle{definition}
\newtheorem{definition}[theorem]{Definition}
\newtheorem{remark}[theorem]{Remark}

\allowdisplaybreaks

\makeatletter
\let\@@pmod\pmod
\DeclareRobustCommand{\pmod}{\@ifstar\@pmods\@@pmod}
\def\@pmods#1{\mkern4mu({\operator@font mod}\mkern 6mu#1)}
\makeatother

\title{Explicit formulas for Grassmannian polylogarithms}
\date{\today}
\author[Charlton]{Steven Charlton}
\address{Fachbereich Mathematik (AZ), Universit\"at Hamburg, Bundesstra\textup{\ss}e 55, 20146 Hamburg, Germany}
\email{steven.charlton@uni-hamburg.de}
\author[Gangl]{Herbert Gangl}
\address{Department of Mathematical Sciences, Durham University, Durham DH1 3LE, United Kingdom}
\email{herbert.gangl@durham.ac.uk}

\author[Radchenko]{Danylo Radchenko}
\address{ETH Z{\"u}rich, Mathematics Department, R\"amistrasse 101, 8092 Z\"urich, Switzerland}
\email{danradchenko@gmail.com}


\begin{document}
\begin{abstract}
    We give new explicit formulas for Grassmannian and Aomoto polylogarithms
    in terms of iterated integrals, for arbitrary weight.
    We also explicitly reduce the Grassmannian polylogarithm in weight~$4$
    and in weight~$5$ each to depth~$2$. Furthermore, using this reduction
    in weight~$4$ we obtain an explicit, albeit complicated, form of
    the so-called $4$-ratio, which gives an expression for the Borel class
    in continuous cohomology of $\GL_4(\CC)$ in terms of $\Lic_4$.
\end{abstract}
\maketitle

\section{Introduction}
The classical polylogarithm $\Lic_m$ is an analytic function defined
by the power series
    \[\Lic_m(z)=\sum_{n=1}^{\infty}\frac{z^n}{n^m}\,,\quad|z|<1\,.\]
For $m\ge1$ it extends to a multivalued analytic function on $\CC\sm\{0,1\}$
as can be seen from the differential equation
$\frac{d}{dz}\Lic_m(z) = \frac{1}{z}\Lic_{m-1}(z)$
together with $\Lic_0(z)=\frac{z}{1-z}$. Polylogarithms
appear in many diverse areas of mathematics, from
hyperbolic geometry to number theory, algebraic geometry and algebraic $K$-theory.

An important open problem in the area, and one of our principal
motivations for this paper, is Zagier's Polylogarithm Conjecture about
the connection between classical polylogarithms and special
values of Dedekind zeta functions at positive integers.
Let us briefly recall one of its formulations.
Let~$F$ be a number field of discriminant~$D_F$ with $r_1$
real embeddings and $r_2$ conjugate pairs of complex
embeddings. Recall that the Dedekind zeta function of~$F$
is defined by $\zeta_F(s)=\sum_{\mathfrak{a}}N(\mathfrak{a})^{-s}$, for \( \mathrm{Re}(s) > 1 \),
where the sum is taken over all non-zero ideals $\mathfrak{a}$ in
the ring of integers~$\mathcal{O}_F$, and \( N(\mathfrak{a}) \) denotes the norm of the ideal~\( \mathfrak{a} \).  The sum is absolutely convergent for~$\mathrm{Re}(s)>1$ and extends to a meromorphic function on \( \CC\) with a simple pole at \( s = 1 \).  For $m\ge 2$ we define an integer $d_m = d_m(F)$ by the formula
	\begin{align*}
	d_m=\begin{cases}
	r_2,     &\mbox{if }m\mbox{ is even}\,,\\
	r_1+r_2, &\mbox{if }m\mbox{ is odd} \,.
	\end{cases}
	\end{align*}
(More conceptually, $d_m(F)$ is the order of
vanishing of $\zeta_F(s)$ at $s=1-m$.)  Let us also define a single-valued version of $\Lic_m$
due to Zagier~\cite{Za-conj}:
	\[\Li_m(z)=\mathrm{Re}_m\Big(\sum_{j=0}^{m-1}
	\frac{2^jB_j}{j!}\Lic_{m-j}(z)\log^j|z|\Big)\,,
	\quad m\ge2\,,\]
where $\mathrm{Re}_m(z)$ denotes the real part of~$z$ if~$m$ is odd
and the imaginary part of~$z$ if~$m$ is even, and \( B_j \) denotes the \( j \)-th Bernoulli number.
For~$m=2$ the function~$\Li_2$ is better known as
the Bloch-Wigner dilogarithm~\cite{Bl1}.
The function~$\Li_m$ is real-analytic on $\CC\sm\{0,1\}$
and continuous on $\mathbb{P}^1(\CC)$.
For convenience we extend~$\Li_m$ to a function on $\ZZ[\CC]$
(formal linear combinations of elements in $\CC$) by linearity.
\begin{conjecture}[Zagier] \label{conj:zagier}
    Let $\{\sigma_{j}\}_{j=1,\dots,n}$ be the set of all complex
    embeddings of a number field~$F$, where $n=[F:\QQ]=r_1+2r_2$,
    labeled in such a way that $\sigma_{j}=\overline{\sigma_{r_1+r_2+j}}$,
    $j=1,\dots,r_2$.
    Then there exist elements
    $y_1,\dots,y_{d_m}\in \mathbb{Z}[F^{\times}]$ such that
    \begin{equation*}
    \zeta_F(m)\sim_{\mathbb{Q}^{\times}}|D_F|^{1/2}\cdot
    \pi^{md_{m+1}}\cdot\det\left(
    \Li_m\left(\sigma_i(y_j)\right)_{1\leq i,j\leq d_m}
    \right).
    \end{equation*}
\end{conjecture}
In fact, the full statement of Zagier's Conjecture also
gives a precise recipe for the choice of the
elements $y_1,\dots,y_{d_m}$: one has to take $y_i$ to be
elements of the
so-called $m$-th Bloch group~$\mathcal{B}_m(F)$, a certain subquotient of \( \ZZ[F^\times] \).
For a precise definition we refer to~\cite{Za-conj}.  Conjecturally \( \mathcal{B}_m(F) \otimes \QQ \) has dimension \( d_m \) and
is (canonically) isomorphic to \( K_{2m-1}(F) \otimes \QQ \), see \cite{BD,dJ}.  Here \( K_n(F) \) is the \( n \)-th algebraic \( K \)-group of \( F \).

Zagier's Conjecture generalizes the regulator
part of the analytic class number formula
	\[ {\rm Res}_{s=1} \zeta_F(s) =
	\frac{2^{r_1}(2\pi)^{r_2} \cdot h_F\cdot {\rm Reg}_F}
    {w_F\cdot \sqrt{|D_F|}}\,,\]
where~$h_F$ is the class number, ${\rm Reg}_F$ is the classical regulator, and~$w_F$
is the number of roots of unity of~$F$.
For~$m=2$ Conjecture~\ref{conj:zagier} follows from the
results of Bloch~\cite{Bl2} and Suslin~\cite{Su} as well as Beilinson (as laid out in~\cite{Burgos}).
In a slightly weaker form it was also proved by Zagier
in~\cite{Za-hyp}. For~$m=3$ it was proved by Goncharov
in~\cite{Go1}, where, in particular, he also outlined
a general approach towards Zagier's Conjecture for $m>3$.
Recently the conjecture was also settled
in the case~$m=4$ by Goncharov and Rudenko~\cite{Go-Ru}.
The conjecture remains open for $m\ge5$.

Goncharov's strategy for proving Conjecture~\ref{conj:zagier}
relies on a theorem of A.~Borel, which we
briefly recall. In~\cite{Bo1} Borel has defined a regulator
map~$r_m^{B}\colon K_{2m-1}(\CC)\to \RR(m-1)$,
where $\RR(k)\coloneqq (2\pi i)^{k}\RR$, and proved that,
	if $\Sigma_F = {\rm Hom}(F,\CC)$ and
	$\psi$ is defined by the composition
	\begin{center}
	\begin{tikzcd}
		K_{2m-1}(F) \ar[r]
		& \bigoplus_{\sigma\in\Sigma_F}
	 K_{2m-1}(\CC) \ar[r, "{\oplus} r_m^{B}"]
		& \ZZ^{\Sigma_F}\otimes \RR(m-1)\,,
	\end{tikzcd}
	\end{center}
	then~$\psi$ is injective modulo torsion, the image
	of~$\psi$ defines a lattice $\Lambda^m_F$
	in $(\ZZ^{\Sigma_F}\otimes \RR(m-1))^+$ (the superscript $+$ denoting invariants under complex conjugation),
	and its covolume \( \mathrm{covol}(\Lambda_F^m) \) is related to \( \zeta_F(m) \) via
	\[\zeta_F(m) \sim_{\QQ^{\times}}
	\sqrt{|D_F|}\,
	\pi^{md_{m+1}}\,{\rm covol}(\Lambda_F^m)\,.\]
(The stronger version of Zagier's conjecture predicts that the image of \( \mathcal{B}_m(F) \) under \( \Li_m \), evaluated on the suitable complex embeddings, is also a lattice in \( \ZZ^{\Sigma_F}\otimes \RR(m-1) \), and that the two lattices should be commensurable.)

The Borel regulator can be represented by the so-called Borel class~\cite{Bo1} in continuous cohomology~$b_m^{(N)}\in H_{{\mathrm{cts}}}^{2m-1}(\GL_N(\CC),\RR(m-1))$,
for $N\ge m$.  An argument in Goncharov's paper \cite[\S2.2]{Go1}
(see also~\cite{Bo2}) establishes that to prove Zagier's conjecture for \( \zeta_F(m) \),
it is enough to give a formula for this Borel class as a linear combination of $\Li_m$'s.
For $m=2$ such a formula was given by
Bloch~\cite{Bl2} using the Bloch-Wigner dilogarithm~$\Li_2$,
and for $m=3$ Goncharov gave an ingenious formula for the
Borel class using~$\Li_3$.

For~$m\ge4$ Goncharov
has shown in~\cite{Go-arak} that the Borel class~$b_m^{(m)}$
can be expressed in terms of a certain function~$\Li_m^{G}$,
the single-valued \emph{Grassmannian polylogarithm}, defined on
the space of $m$-planes in $\CC^{2m}$.
%(he also gave a construction, using the Grassmannian polylogarithm, for 
%all~$b_m^{(N)}$, $N\ge m$).
However, the function~$\Li_m^{G}$ cannot be expressed in terms of
only~$\Li_m$ for $m\ge4$.

In their proof of Conjecture~\ref{conj:zagier}
for $m=4$ Goncharov and Rudenko have overcome this difficulty
by giving a formula for the Borel regulator using
the (multi-valued) Grassmannian polylogarithm~$\GR_4$
from~\cite{Go-grass} (see Section~\ref{sec:grpolylog} below),
and showing the existence of an~$\Li_4$-expression
for a small modification of~$\GR_4$ that represents
the same cohomology class.
More precisely, to prove the existence of the \( \Li_4 \) expression they established part of the conjectural structure of the motivic Lie coalgebra in weight $4$.
 Their proof does not seem to give any practical way
of producing an explicit~$\Li_4$-formula for~$b^{(4)}_4$, though.

Motivated by Goncharov's original work \cite{Go1} in conjunction with \cite{Go-Ru}, 
with a view towards Zagier's Conjecture in weights \( 5 \) and higher, we investigate 
the Grassmannian polylogarithm \( \GR_m \) via explicit formulas in terms of 
classical iterated integrals. This gives rise to numerous explicit formulas which we 
now state. In order to formulate our results it will be convenient to introduce some 
more notation.

\smallskip
\noindent{\textbf{Notation.}}
First, we introduce new coordinates $\rho_i=\rho_i^{(2m-1,2m)}$ on the moduli 
space of configurations of $2m$ points $v_1,\dots,v_{2m}$ in an $m$-dimensional 
vector space, modulo the action of $\GL_m$. Explicitly, 
$\rho_i$ is the ratio
$\frac{\detv{i,i+1,\dots,i+m-2,2m-1}}{\detv{i,i+1,\dots,i+m-2,2m}}$,
where $\detv{i_1,\dots,i_m}$ denotes the determinant of the $m{\times}m$-matrix
with columns $v_{i_1},\dots,v_{i_m}$.

Second, following~\cite{Go-Ru}, given $2n$ points $x_1,\dots,x_{2n}\in\mathbb{P}^1$ we will denote by $(i_1i_2\dots i_{2n})_x$ the cyclic ratio
\[(i_1i_2\dots i_{2n})_x \coloneqq 
(-1)^n\,\frac{{}_{\phantom{n}}x_{i_{1}}-x_{i_{2}}}{x_{i_{2n}}-x_{i_1}}\, 
    \frac{x_{i_3}-x_{i_4}}{x_{i_2}-x_{i_3}}\cdots 
    \frac{x_{i_{2n-1}}-x_{i_{2n\phantom{-1}}}}{x_{i_{2n-2}}-x_{i_{2n-1}}}\,.\] 

Finally, it will also be 
convenient to introduce various symmetrization operators: 
by $\Alt_m f(x_1,\dots,x_m)$ we denote the skew-symmetrization of $f$ over all 
permutations of the $x_i$; by $\Alt_{m,m} f(x_1,\dots,x_m,x_{m+1},\dots,x_{2m})$
the skew-symmetrization of $f$ taken over all permutations of indices in 
$\mathfrak{S}_{\{1,\dots,m\}}\times\mathfrak{S}_{\{m+1,\dots,2m\}}$; and by
$\Cyc_{2m}^{\varepsilon}\,f(x_1,\dots,x_{2m})$ we denote the cyclic linear 
combination 
$\sum_{i\Mod{2m}}\varepsilon^{i}f(x_{i+1},\dots,x_{i+2m})$ with indices written 
modulo~$2m$.

\smallskip 
%\noindent{\textbf{Main results.}}\nolinebreak[4]\nopagebreak[4]
\paragraph{{\bf Main results.}}\nobreak
\begin{enumerate} 
%) and $\Alt_{2m}$ denotes the skew-symmetrization over all permutations of the $2m$ 
%vectors.
\item[1)] We give a new and concise formula for
Grassmannian polylogarithms $\GR_m=\GR_m(v_1,\dots,v_{2m})$ in arbitrary weight $m$ in terms of iterated integrals
up to depth $m$ (Theorem~\ref{thm:i211}) as
    \begin{equation*} 
    -{2m\choose 2}\GR_m \;=\; m!^2
    \Alt_{2m}\II(0;0,\rho_1,\rho_2,\dots,\rho_{m-1};\rho_m)\,,
    \end{equation*}
    where $\II$ is defined in Section~\ref{sec:prelim} below, and equality is to be
    understood on the level of symbols (see Section~\ref{sec:modprod}).
\item[2)] 
In weight~4 we give an expression for $\GR_4$ using only depth $\le2$ iterated 
integrals (Theorem~\ref{thm:gr4toi31_2}). Specifically, modulo products the identity 
can be written as 
	\[
	\frac{7}{72}\,\GR_4 \;=\;  \Alt_8 \,
	C_4(\rho_1, \rho_2, \rho_3, \rho_4, \infty, 0)\,,
	\]
where, for general weight~$n$, $C_n$ is the depth~$2$ function of $6$ points in 
$\mathbb{P}^1$, defined by
    \[
    C_n(x_1,\dots,x_6) = \Cyc_6^{(-1)^{n-1}}\Big( 
    I_{n-1,1} 
    \big((1234)_x,(4561)_x\big) - (-1)^n \tfrac{2(n-1)}{3}\Lic_{n}((123456)_x)\Big)\,.
    \]
    (For the definition of $\II_{n,1}$ see~\eqref{eq:Indef} below.)
\begin{comment}
	\[
	\frac{7}{144}\,\GR_4 \;\modsh\;  \Alt_8 \Big[
	\II_{3,1}\Big(\frac{\rho_{1,2}\rho_{3,4}}
	{\rho_{3,2}\rho_{1,4}},
	\frac{\rho_1}{\rho_{1,4}}\Big)
	+2\II_{3,1}\Big(\frac{\rho_{1,2}}
	{\rho_1},
	\frac{\rho_{3,2}}{\rho_{3,4}}\Big)
	+6\Lic_{4}\Big(\frac{\rho_1\rho_{3,2}}
	{\rho_{1,2}\rho_{3,4}}\Big)\Big]\,,
	\]
\end{comment}

\item[3)] Building on 2) we obtain an explicit, albeit
complicated, formula for the elusive 
quadruple ratio by expressing a non-zero rational multiple of the Borel
class~$b_4^{(4)}$, given by a slight modification of $\GR_4$, in terms of~$\Li_4$
(Theorem~\ref{thm:gr4toli4} and Corollary~\ref{cor:borel}).
\item[4)]  In weight~5 we find a similar expression to the one in 2), also in terms 
of depth 2 iterated integrals (Theorem~\ref{thm:gr5toi41} ), namely, modulo products 
we have
	\[
	\frac{1}{320}\,\GR_5 \;=\;  \Alt_{10} \Big[
    C_5(\infty,\rho_1,\rho_2,\rho_3,\rho_4,\rho_5)
    -C_5(\infty, 0, \rho_2,\rho_3,\rho_4,\rho_5)
	\Big]\,.
	\]
\item[5)] Building on 4) we express in Conjecture~\ref{conj:gr5del}
a non-zero rational multiple of the Borel
class~$b_5^{(5)}$, given by a slight modification of $\GR_5$, as a linear
combination of expressions of the form $\II_{4,1}^+(\FE_2,\cdot)$ and 
$\II_{4,1}^+(\cdot, \FE_3)$, where
$\FE_k$ can be any functional equation for $\Li_k$ ($k=2,3$) and $\II_{4,1}^{+}$ is a 
symmetrized version of $\II_{4,1}$, defined in~\eqref{eq:i41plusdef} below.

This constitutes a substantial stepping stone towards Zagier's Polylogarithm Conjecture in weight~5
as, assuming the conjectural structure of the motivic Lie coalgebra in weight~5, each 
such expression can be written in terms of $\Li_5$ only.
\item[6)] Finally, following a suggestion of Daniil Rudenko, we obtain via 1) a 
formula for the Aomoto polylogarithm, defined on pairs of simplices and subject to 
scissors congruence relations
and further symmetries (see~\cite{Go-grass} and Section~\ref{sec:grpolylog} below), 
in terms of iterated integrals as well, again for arbitrary weight.
Moreover, we find a concise form for it in  Theorem~\ref{thm:aomoto},
yet again using the $\rho$-coordinates introduced above, as
 \begin{equation*}
	-m^2 \,\CA_{m-1}(v_1,\dots,v_m;v_{m+1},\dots,v_{2m}) \;=\; 
	 \Alt_{m,m} 
	 \II(0;\rho_{m+1}^{(1,2m)},\rho_{m}^{(1,2m)},\dots,\rho_{3}^{(1,2m)};\rho_{2}^{(1,2m)})\,.
   \end{equation*}
\end{enumerate}

\smallskip
\noindent{\textbf{Acknowledgements.}}
This work was initiated during the Trimester Program
\emph{Periods in Number Theory, Algebraic Geometry and Physics}
at the Hausdorff Research Institute for Mathematics in Bonn.
We are grateful to this institution, as well as to the
Max Planck Institute for Mathematics in Bonn, for their
hospitality, support and excellent working conditions.  SC was partially supported by DFG Eigene Stelle grant CH 2561/1-1, for Projektnummer 442093436.
Many thanks to Daniil Rudenko who suggested, following the original 
preprint, that our analysis should also allow us to write the Aomoto
polylogarithm as an iterated integral.

\medskip
\section{Preliminaries}
\label{sec:prelim}
We briefly recall some of the motivic framework of multiple
polylogarithms and iterated integrals from Goncharov's paper~\cite{Go-galois},
in particular their Hopf algebra structure and $\otimes$-symbols
of iterated integrals.

\subsection{Iterated integrals}
Recall the definition of the iterated integral function
    \[ \II(x_0; x_1, \ldots, x_N; x_{N+1}) =
    \int_{x_0 < t_1 < \cdots < t_N < x_{N+1}}
    \frac{dt_1}{t_1 - x_1}
    \wedge \frac{dt_2}{t_2 - x_2}
    \wedge \cdots \wedge
    \frac{dt_N}{t_N - x_N} \,.\]
As per standard notation, we put
    \begin{equation} \label{eq:Indef}
    \II_{n_1,\dots,n_d}(x_1,\dots,x_d)=\II(0;x_1,\{0\}^{n_1-1},\dots,x_d,\{0\}^{n_d-1};1)\,,
    \end{equation}
and $\{a\}^n$ is~$a$ repeated~$n$ times.
These functions are related to the multiple polylogarithms
    \[ \Lic_{n_1,\dots,n_d}(z_1,\dots,z_d)
    = \sum_{0<k_1<\dots<k_d}\frac{z_1^{k_1}\cdots z_d^{k_d}}
    {k_1^{n_1}\cdots k_d^{n_d}}\]
by the formula
	\begin{equation*}
	\II_{n_1,\dots,n_d}(0;
	(a_1\dots a_d)^{-1},
	(a_2\dots a_d)^{-1},\dots,
	a_d^{-1};1)
	= (-1)^d\Lic_{n_1,\dots,n_d}(a_1,a_2,\dots,a_d)\,.
	\end{equation*}

\subsection{Motivic iterated integrals}
In \cite{Go-galois}, the iterated integrals $\II(x_0; \ldots; x_{N+1})$,
\( x_i \in \overline{\QQ} \), are upgraded to framed mixed Tate motives, to
define motivic iterated integrals \( I^{\mathscr{M}}(x_0; \ldots; x_{N+1})
\) living in a graded (by the weight~$N$) connected Hopf algebra
\( \mathcal{A}_\bullet = \mathcal{A}_\bullet(\ol{\QQ}) \).
The Hopf algebra $\mathcal{A}_{\bullet}$ is the ring of regular functions on
the unipotent part of the motivic Galois group.
(Note that in Brown's motivic framework,~\cite{Br}, the motivic iterated
integrals $\II^{\mathscr{M}}$ are denoted by~$\II^{\mathfrak{u}}$.)
The coproduct \( \Delta \) on this Hopf
algebra is computed via Theorem 1.2 in \cite{Go-galois} as
    \begin{equation} \label{eq:coprod}
    \begin{split}
    & \Delta I^{\mathscr{M}}(x_0; x_1,\ldots,x_N; x_{N+1})\\
    & = \sum_{0 = i_0 < i_1 < \cdots < i_{k} < i_{k+1} = N+1}
    I^{\mathscr{M}}(x_0; x_{i_1}, \ldots, x_{i_k}; x_{N+1}) \otimes
    \prod_{p=0}^{k} I^{\mathscr{M}}(x_{i_p}; x_{i_{p}+1}, \ldots, x_{i_{p+1}-1};
    x_{i_{p+1}}) \,.
    \end{split}
    \end{equation}
Here $I^{\mathscr{M}}(a;b;c)$ is regularized as
    \[
    I^{\mathscr{M}}(a;b;c) = \begin{cases}
    \log^{\mathscr{M}}(1) = 0 & \text{if $ a = b $ and $ b = c $\,,} \\
    \log^{\mathscr{M}}(\tfrac{1}{b - a})  & \text{if $ a \neq b $ and $ b = c $\,,} \\
    \log^{\mathscr{M}}(b - c)  & \text{if $ a = b $ and $ b \neq c $\,,} \\
    \log^{\mathscr{M}}(\tfrac{b - c}{b - a})  & \text{otherwise\,.}
    \end{cases}
    \]

Similarly, if $x_i\in\CC$ we follow Goncharov and consider
$\II^{\mathcal{C}}(x_0;\dots;x_{m+1})$ as framed Hodge-Tate structures
(see~\cite[\S3.2 (ii), p.~232]{Go-galois}), where the coproduct for the
corresponding Hopf algebra of framed objects is given by the same
formula~\eqref{eq:coprod}, with~${}^{\mathscr{M}}$ replaced
by~${}^{\mathcal{C}}$ (see~\cite[Thm.~3.4]{Go-galois}).
Since our results ultimately only use~\eqref{eq:coprod},
they apply in any of these two setups.  We therefore adopt the following convention.\smallskip
\par\noindent {\em Convention.} We will omit the superscripts from
the notation, and simply
write $\II(x_0;x_1,\dots,x_m;x_{m+1})$.

\subsection{\texorpdfstring{The $\otimes$-symbol modulo products and the Lie coalgebra}{The tensor-symbol modulo products and the Lie coalgebra}}
\label{sec:modprod}
Recall from \cite[\S4.4]{Go-galois} the ``$\otimes^m$-invariant'',
or~\emph{symbol}, of an iterated integral.
The symbol $\Sl(\II^{\mathscr{M}})\in\mathcal{A}_1^{\otimes N}$
is an algebraic invariant of~$\II^{\mathscr{M}}$ that respects functional
equations. More precisely, $\Sl\colon (\Aa_{\bullet},\shuffle, \Delta) \to (T(\Aa_1),
\shuffle, \Delta_{\mathrm{dec}})$ is a map of graded Hopf algebras,
where $T(\Aa_1)$ is the tensor algebra of $\Aa_1$, and $\Delta_{\mathrm{dec}}$
is the deconcatenation coproduct on tensors. It is obtained by iterating the reduced
coproduct $\Delta'$ exactly $N-1$ times. A corollary of the definition is the
following recursive formula for the symbol that we will use repeatedly
	\begin{align*}
	\Sl I(x_0; \ldots; x_{N+1})
	= \sum_{j=1}^N \Sl I(x_0; x_1, \dots, \widehat{x_j}, \dots, x_N; x_{N+1})
	 \otimes I(x_{j-1}; x_j; x_{j+1})\,.
	\end{align*}

Recall also the projectors~$\Pi_\bullet$ from~\cite[\S5.5]{Du-Ga-Rh} which
annihilate the symbols of products.
The projector $D_N = N\Pi_N$ acts on length~$N$ tensors as follows.
	\[
	\Pid_N(a_1\otimes\dots\otimes a_N)
	=\Pid_{N-1}(a_1\otimes\dots\otimes a_{N-1})
	\otimes a_N
	-\Pid_{N-1}(a_2\otimes\dots\otimes a_{N})
	\otimes a_1
	\,.\]
We prefer $D_N$ to $\Pi_N$ in order to avoid unnecessary scaling factors, at
the expense of that operator no longer being idempotent. Note that the map
$D_N$ is the classical Dynkin operator in the theory of Hopf algebras (see, e.g.,~\cite[Section~4]{EFGBP}).

We call the composition
$\Pid_N \circ \Sl \eqqcolon \Sl^\shuffle\colon\Aa_{\bullet}\to T(\Aa_1)$ the
\emph{mod-products symbol}. As usual with symbols, we will drop the
$\log^{\mathscr{M}}$ from
the notation and write tensors multiplicatively. The notation might suggest
that the symbol entries are elements $x\in \ol{\QQ}$, when they are
really elements of $\ol{\QQ}^{\times} \otimes_{\ZZ}\QQ \cong \Aa_1$.
In particular, on the level of the symbol 2-torsion vanishes,
because in the Hopf algebra $\Aa_{\bullet}$ (as $(2\pi i)^{\mathscr{M}}$ is
zero) one has the exact equality of motivic logarithms $\log^{\mathscr{M}}(x) = \log^{\mathscr{M}}(-x)$. We can therefore ignore signs in the tensor entries, and freely interchange between $\otimes (-x)$ and $\otimes x$.
To emphasize that certain identities hold only on the level of symbol
or on the level of mod-products we shall write $f\modS g$ and $f\modsh g$ to
denote $\Sl(f)=\Sl(g)$ and $\Sl^{\shuffle}(f)=\Sl^{\shuffle}(g)$, respectively.

To give an example, $\Sl^{\shuffle}$
for classical polylogarithms is given
by~$\Sl^{\shuffle}\Lic_m(x) = (x\wedge(1-x))\otimes x^{\otimes (m-2)}$, where we
write $a\wedge b = a\otimes b - b\otimes a$.
An important property of the single-valued
polylogarithms~$\Li_m$ is that if the $f_j$ are rational
functions and~$\Sl^{\shuffle}\sum \nu_j \Lic_m(f_j(x)) = 0$,
then $\sum_{j}\nu_j\Li_m(f_j(x))$ is constant
(see~\cite[Prop.~1,~p.~411]{Za-conj}).

Finally, recall that the coproduct in a Hopf algebra induces a
cobracket $\delta = \Delta - \Delta^{\mathrm{op}}$ (with $\Delta^{\mathrm{op}}$
the opposite coproduct) on the Lie coalgebra of irreducibles
    \(
    \mathcal{L}_\bullet \coloneqq \Aa_{>0} / \Aa_{>0}^2
    \).
The $2$-part of this cobracket in weight~$m$, i.e. the composition of
$\delta$ with projection to~$\bigoplus_{k=2}^{m-2} \mathcal{L}_{k} \wedge \mathcal{L}_{m-k}$, can be seen to annihilate all classical polylogarithms,
and conjecturally this is the only obstruction, see Conjecture 1.20 and
Section 1.6 in~\cite{Go-ams}. We use the vanishing of the $2$-part
of $\delta$ as a guiding principle for possible depth reduction of
the weight~$5$ Grassmannian polylogarithm in Section~\ref{sec:wt5}.

\medskip
\section{Grassmannian and Aomoto polylogarithms}
\label{sec:grpolylog}
There are several different constructions of
``Grassmannian polylogarithms'' in the literature: there is a real-valued
Grassmannian logarithm of Gelfand and MacPherson~\cite{GM},
Grassmannian $m$-logarithms constructed by Hanamura
and MacPherson~\cite{HM1},~\cite{HM2},
and Goncharov's construction of real-analytic
single-valued Grassmannian polylogarithms~\cite{Go-arak}.
The subject of our investigations is Goncharov's complex analytic multi-valued
Grassmannian polylogarithm, defined in~\cite{Go-grass}.
We refer to it as \emph{the} Grassmannian polylogarithm throughout this paper.

For $m,n\ge 1$, let $\Conf{n}{m}$ be the space of all $n$-tuples of vectors
$(v_1,\dots,v_n)$ in general position in~$\CC^m$ modulo the diagonal
action of $\GL_m(\CC)$. Let us denote by $\detv{i_1,\dots,i_m}$ the
determinant of the $m\times m$ matrix with columns $v_{i_1},\dots, v_{i_m}$
(for better readability we will often omit the commas and simply
write~$\detv{i_1\dots i_m}$). The functions $\detv{i_1\dots i_m}$ are
invariant under the action of~$\SL_m(\CC)$ and the ring of regular
functions $\mathcal{O}(\Conf{n}{m})$ is generated by all possible
ratios of determinants $\detv{i_1\dots i_m}/\detv{j_1\dots j_m}$.

First, following~\cite{Go-grass}, we recall the definition and basic properties of Aomoto polylogarithms. The Aomoto $n$-logarithm is a complex analytic
function defined on admissible pairs of $n$-simplices in $\mathbb{P}^{n}(\CC)$. Here a simplex in $\mathbb{P}^{n}(\CC)$ is simply a collection $(l_0,\dots,l_n)$ of $(n+1)$ hyperplanes. We call a simplex $L=(l_0,\dots,l_n)$ non-degenerate if its hyperplanes are in general position. A pair of simplices
$L=(l_0,\dots,l_n)$ and $M=(m_0,\dots,m_n)$ is called admissible, if $L$ and $M$ are non-degenerate and if $L$ and $M$ share no faces of the same dimension. To any non-degenerate simplex $L$ one associates an $n$-form $\omega_L$ in $\mathbb{P}^{n}(\CC)$ having simple poles at the hyperplanes $l_i$,
namely, $\omega_L = d\log(z_1/z_0)\wedge \dots \wedge d\log(z_n/z_0)$,
where $z_i=0$ denotes the homogeneous equation of the hyperplane $l_i$. Next, to any 
non-degenerate simplex $M=(m_0,\dots,m_n)$ one associates a topological $n$-cycle 
$\Delta_M$ representing the generator of the rank one (relative) homology group 
$H_n(\mathbb{P}^n(\CC), m_0\cup \dots\cup m_n)$. The Aomoto $n$-logarithm is then 
defined as
	\[\mathcal{A}_n(L;M) = \int_{\Delta_M}\omega_L\,.\]
For our purposes it is convenient to dualize and instead view $\mathcal{A}_n$
as a function of pairs of $(n+1)$-tuples of points in $\mathbb{P}^n(\CC)$, which we 
will also denote by $l_i$ and $m_j$.
To give the simplest example, for $n=1$
    \[\mathcal{A}_1(l_0,l_1;m_0,m_1) = \log \CR(l_0,l_1,m_0,m_1)\,,\]
where $\CR(a,b,c,d)=\frac{(a-c)(b-d)}{(a-d)(b-c)}$ is the classical cross-ratio.
The function $\mathcal{A}_n$ is skew-symmetric in both sets of points,
projectively invariant, and most importantly satisfies the additivity relation
	\[\sum_{i=0}^{n+1}(-1)^i\mathcal{A}_n((l_0,\dots,\widehat{l_i},\dots,l_{n+1});(m_0,\dots,m_n))
	 = 0\]
for any configuration $(l_0,\dots,l_{n+1})$ in $\mathbb{P}^n(\CC)$. It also satisfies 
analogous additivity relation in the second variable~$M$ as well as dualized 
versions of additivity~\cite[\S2.1]{Go-grass}. Abstracting away these properties one 
considers scissors congruence groups $A_n(F)$, abelian groups
generated by the elements $\langle l_0,\dots,l_n;m_0,\dots,m_n\rangle_{A_n}$ 
with $l_i,m_j\in\mathbb{P}^n(F)$, subject to the above-mentioned skew-symmetry, projective invariance, and additivity properties. The graded sum $\bigoplus_{n\ge0}A_n(F)$ also has a graded coassociative coalgebra structure,
of which we will only need the component $\nu_{n-1,1}$ of the coproduct 
(see~\cite[eq.~(7)]{Go-grass}) that for generic simplices can be expressed in the 
following form
	\[\nu_{n-1,1}(\langle L;M\rangle_{A_n})
	= -\sum_{i,j=0}^{n}(-1)^{i+j}\langle l_i| l_0,\dots,\widehat{l_i},\dots,l_n;m_0,\dots,\widehat{m_j},\dots,m_n\rangle_{A_n}\otimes
	\Delta(l_i,m_0,\dots,\widehat{m_j},\dots,m_n)\,.\]
As per general setup explained in~\S\ref{sec:modprod}, iterating the $(n-1,1)$
part of the coproduct using the above formula we obtain the tensor symbol of 
the Aomoto polylogarithms.
\begin{proposition}\label{prop:aomoto_symbol}
The Aomoto polylogarithm $\mathcal{A}_n(v_1,\dots,v_{n+1};v_{n+2},\dots,v_{2n+2})$ has $\otimes$-symbol
	\[(-1)^n\Alt_{n+1,n+1}\Delta(2,\dots,n+2)\otimes \Delta(3,\dots,n+3)\otimes
	\dots\otimes \Delta(n+1,\dots,2n+1)\,.\]
\end{proposition}

%For the rest of this section we fix the weight to be \( m \).  
In~\cite{Go-grass} Goncharov has defined the Grassmannian
$m$-logarithm $\GR_m(v_1,\dots,v_{2m})$ as a multivalued analytic function
on $\Conf{2m}{m}$ by requiring that
$\GR_m(v_1,\dots,v_{2m})=\Alt_{2m}F(v_1,\dots,v_{2m})$,
where, as in the introduction, $\Alt_{n}$ denotes the skew-symmetrization operator
	\[\Alt_n f(x_1,\dots,x_{n}) =
	\sum_{\sigma\in\mathfrak{S}_n}\sgn(\sigma)
	f(x_{\sigma(1)},\dots,x_{\sigma(n)})\]
%(where symmetrization variables are usually understood from the context),
and the function~$F$ is a primitive of the following $1$-form
    \begin{equation} \label{eq:gdef}
    \Alt_{2m}\; \Big(\mathcal{A}_{m-1}(v_1,\dots,v_m;v_{m+1},\dots,v_{2m})
    \,d\log\detv{m+1,\dots,2m}\Big)\,,
    \end{equation}
where $\mathcal{A}_{m-1}$ is the Aomoto polylogarithm
(see~\cite[\S1.1]{Go-grass}).
Goncharov has proved (loc.~cit.) that $\GR_m$ is well-defined, i.e., that the
$1$-form on the right-hand side of~\eqref{eq:gdef} is indeed closed,
and that it is projectively invariant,
i.e., $\GR_m(\lambda_1v_1,\dots,\lambda_{2m}v_{2m})=\GR_m(v_1,\dots,v_{2m})$
for all $\lambda_1,\dots,\lambda_{2m}\in\CC^{\times}$.
In particular,~$\GR_m$ is a well-defined function on the space
of configurations of~$2m$ points in~$\mathbb{P}^{m-1}(\CC)$.
Note that $\GR_m(v_1,\dots,v_{2m})$ is manifestly skew-symmetric under
the permutations of $v_1,\dots,v_{2m}$. The key property
of the Grassmannian polylogarithm is that it satisfies the
following functional equations.
\begin{proposition}[\cite{Go-grass}] \label{prop:grassfeq}
    \begin{enumerate}
    	\item[(i)] For any generic configuration of $(2m+1)$ vectors
    $v_0,\dots,v_{2m}$ in~$\CC^m$ we have
    \begin{equation} \label{eq:grassfeq1}
    \sum_{i=0}^{2m}(-1)^i\GR_m(v_0,\dots,\widehat{v_i},\dots,v_{2m}) =
    \mathrm{const}\,.
    \end{equation}
    \item[(ii)] For any generic configuration of $(2m+1)$ vectors
    $w_0,\dots,w_{2m}$ in $\CC^{m+1}$ we have
    \begin{equation}\label{eq:grassfeq2}
    \sum_{i=0}^{2m}(-1)^i\GR_m(\pi_i(w_0),\dots,\widehat{w_i},\dots,\pi_i(w_{2m}))
     = \mathrm{const}\,,
    \end{equation}
    where $\pi_i$ denotes the canonical projection from $\CC^{m+1}$
    to $\CC^{m+1}/\langle w_i\rangle$.
    \end{enumerate}
\end{proposition}
These identities follow from the following expression for
the symbol of $\GR_m$ (\cite[Thm.~4.2]{Go-grass})
    \begin{equation} \label{grsymbol}
    \Sl(\GR_m)
    = 2(-1)^m(m!)^2
    \Alt_{2m} \detv{1,\dots,m} \otimes \detv{2,\dots,m+1} \otimes
    \dots \otimes \detv{m,\dots,2m-1}\,.
    \end{equation}

Our first result is an explicit formula that expresses $\GR_m(v_1,\dots,v_{2m})$ in
terms of the classical iterated integrals. To state this and other formulas,
for a set $I = \{i_1,\dots,i_{m-1}\}\subset \{1,\dots,2m\}$ we define
    \[\rho_{I}^{(k,l)}
    \;\coloneqq\; \frac{\detv{i_1,i_2,\dots,i_{m-1},k}}
    {\detv{i_1,i_2,\dots,i_{m-1},l}}\,,\qquad
    I\cap \{k,l\} = \varnothing\,.\]
Note that $\rho_{I}^{(k,l)}$ is a rational function that is
symmetric and projectively invariant in $v_{i_1},\dots,v_{i_{m-1}}$.
Furthermore, let us set $\rho_i^{(k,l)}:=\rho_{\{i,\dots,i+m-2\}}^{(k,l)}$
and $\rho_i \coloneqq \rho_i^{(2m-1,2m)}$.
\begin{theorem}\label{thm:i211}
    For $m\ge 2$ we have the following identity on the level of symbols
    \begin{equation} \label{eq:i211}
    -\frac{(2m-1)}{m!(m-1)!}\GR_m \;\modS\;
    \Alt_{2m}\II(0;0,\rho_1,\rho_2,\dots,\rho_{m-1};\rho_m)\,.
    \end{equation}
\end{theorem}
The proof will be given in Section~\ref{sec:proof_i211}.
As a corollary from this Theorem and Proposition~\ref{prop:grassfeq}
we obtain explicit geometric functional equations for the iterated
integral~$\II(0;0,x_1,\dots,x_{m-1};x_m)$.

\begin{remark} The proof can also be adapted to show the same formula
    with only the change of the lower bound of the integrals from 0 to $\infty$
    \begin{equation} \label{eq:i211inf}
    -\frac{(2m-1)}{m!(m-1)!}\GR_m \;\modS\;
    \Alt_{2m}\II(\infty;0,\rho_1,\rho_2,\dots,\rho_{m-1};\rho_m)\,.
    \end{equation}
    Here the iterated integrals $\II(\infty;a_1,\dots,a_m;a_{m+1})$
    can be shuffle-regularized and written as an explicit combination of
    iterated integrals evaluated at finite points. These integrals
    have better symmetry properties, for instance, they are
    invariant (up to sign and modulo products)
    under dihedral permutations of the variables $a_1,\dots,a_{m+1}$.
    These iterated integrals starting at~$\infty$ are
    also related to the ``motivic correlators''
    $\mathrm{Cor}_{\infty}(a_1,\dots,a_{m+1})$ from~\cite{Go-hodge}
    and~\cite{Go-Ru}
    (the main difference is that $\II(\infty;a_1,\dots,a_m;a_{m+1})$
    lives in the motivic Hopf algebra, while $\mathrm{Cor}_{\infty}(a_1,\dots,a_{m+1})$
    lives in the motivic Lie coalgebra, a quotient of the latter).
\end{remark}

\begin{remark}
    The geometric meaning of~$\rho_{i_1,\dots,i_{m-1}}^{(k,l)}$ is as follows.
    If we pick projective coordinates on the line~$\ell$ passing
    through~$P_{k}$ and~$P_{l}$ (where $P_j\in\mathbb{P}^{m-1}$
    corresponds to $v_j\in\CC^{m}$)
    in such a way that $P_{k}$ has coordinate $0$, and $P_{l}$
    has coordinate $\infty$, then $\rho_{i}$ is the
    coordinate for the intersection of $\ell$ with the
    hyperplane passing through the points $P_{i_1},\dots,P_{i_{m-1}}$.
    Note that~$\rho$'s depends on the choice of the vectors
    $v_{k},v_{l}$ that represent $P_{k},P_{l}\in\mathbb{P}^{m-1}$,
    but any ratio of two $\rho$'s is well-defined and projectively invariant.
\end{remark}

\begin{theorem}\label{thm:aomoto}
    For $m\ge 2$ we have the following identity
    \begin{equation} \label{eq:aomoto}
    \mathcal{A}_{m-1}(v_1,\dots,v_m;v_{m+1},\dots,v_{2m}) \;\modS\;
    \frac{(-1)^{m-1}}{m^2}\Alt_{m,m}\II(0;\rho_{m+1}^{(1,2m)},\rho_{m}^{(1,2m)},\dots,
    \rho_{3}^{(1,2m)};\rho_{2}^{(1,2m)})\,,
    \end{equation}
    where $\Alt_{m,m}$ denotes the skew-symmetrization over all index permutations 
    in $\mathfrak{S}_{\{1,\dots,m\}}\times\mathfrak{S}_{\{m+1,\dots,2m\}}$.
\end{theorem}
Let us also remark that the identities~\eqref{eq:i211} and~\eqref{eq:aomoto} also 
make sense as equalities between multivalued analytic functions, if one chooses their 
branches appropriately (compare with the remark after Theorem~3.9 in~\cite{Go2}).

\begin{remark}
\begin{enumerate}
\item There is a slightly better formula for the Aomoto polylogarithm in which we 
replace $\Alt_{m,m}$ by $\Alt_{m-1,m-1}$ (fixing the two labels 1 and $2m$---note the different
superscripts as opposed to the formula for $\GR_m$). This has 
the further benefit of cancelling the factor $m^2$. 
\item Furthermore, generalizing~\eqref{eq:i211} and~\eqref{eq:aomoto} one can also 
consider the higher weight functions
    \[\Alt_{m-1,m-1}
    \II(0;{\bf 0,\dots,0},\rho_{m+1}^{(1,2m)},\rho_{m}^{(1,2m)},\dots,\rho_{3}^{(1,2m)};\rho_{2}^{(1,2m)})\]
obtained by adding a string of zeros on the left (highlighted in boldface above) to the iterated integral (this construction is 
inspired by the construction of the ${\rm QLi}$ functions in~\cite{Ru}). Then 
following the initial steps in the proof of Theorem~\ref{thm:aomoto} gives a 
recursive formula for the $(n-1,1)$-part of the coproduct of these functions, that as 
a corollary shows that their symbol is a linear combination of tensor products of 
Pl\"ucker coordinates.
\end{enumerate}
\end{remark}

%{\color{red} (Also seems Herbert is  correct, that
%	\[
%		(-1)^{m-1} \Alt_{m-1,m-1} \cdots	
%	\]
%
%suffices)
%}
\medskip
\section{The Grassmannian polylogarithm in weight 4}
In this section we formulate our main results for the
weight~$4$ Grassmannian polylogarithm. First we give
an explicit formula for $\GR_4$ in terms of $\II_{3,1}$ and
$\Lic_4$ (Theorem~\ref{thm:gr4toi31_2}).
It is known that $\GR_4$ is not expressible in terms of $\Lic_4$  alone,
the only obstacle to doing so being the non-vanishing of the 2-part of its cobracket (see Section \ref{sec:modprod}).
We reproduce a `coboundary correction' for it---the $\Alt_8$ term on the left hand side of Theorem
\ref{thm:gr4toli4}---with matching cobracket.
This `coboundary correction'  is just a version of Goncharov's $\delta_{2,2}$ in  \cite{Go2}.
Finally we give an explicit expression for the
difference, i.e.~of $\GR_4$ minus this coboundary correction,
in terms of $\Lic_4$ only (Theorem~\ref{thm:gr4toli4}).
The resulting $\Lic_4$ expression is our version of the elusive \emph{quadruple ratio}.
(Here by a `coboundary'
we mean a linear combination of functions on configurations
of 8 points in $\mathbb{P}^3$ where each individual term depends on
at most 7 of these points. The reason for this terminology is that
such a `coboundary' lies in the image of the coboundary operator~$d$ of a
suitable cochain complex. Note that any such `coboundary' will trivially
vanish when we alternate over 9 points.)
While the existence of such formulas follows from the results
of Goncharov and Rudenko in~\cite{Go-Ru}, their proof does
not seem to give any practical approach to obtaining them.

\subsection{\texorpdfstring{Explicit formula for $\GR_4$ in terms of $\II_{3,1}$ and $\Lic_{4}$}{Explicit formula for GR\textunderscore{}4 in terms of I\textunderscore{}\{3,1\} and Li\textunderscore{}4}}
Theorem~\ref{thm:i211} already gives us an explicit formula for
$\GR_4$ in terms of iterated integrals, to which one
could apply the known reduction formulas in weight~$4$
(see~\cite{Dan},~\cite{Dan2},~\cite{Ga},~\cite{Ch}) to obtain an
explicit expression in terms of~$\II_{3,1}$ and~$\Lic_4$
(recall that $\II_{3,1}(x,y) = \II(0;x,0,0,y;1)$).
This reduction, however, produces a somewhat complicated
expression. Instead we will give a direct formula for~$\GR_4$
in terms of~$\II_{3,1}$ and~$\Lic_4$ that is much shorter.

We are working with the configuration space $\Conf{8}{4}$
and, as in the more general situation in Theorem~\ref{thm:i211},
we set
    \[\rho_{i}\coloneqq \rho_{\{i,i+1,i+2\}}^{(7,8)} =
    \frac{\detv{i,i+1,i+2,7}}{\detv{i,i+1,i+2,8}}\,,\qquad 1\le i \le 4\,.\]
We will also use the following notation for projected cross-ratios
	\[\CR(ab| cdef) \coloneqq
	\frac{\detv{abce}\detv{abdf}}{\detv{abcf}\detv{abde}}\,.\]
Geometrically $\CR(ab| cdef)$ is simply the cross-ratio
of the images of $v_{c},v_{d},v_{e},v_{f}$ in $\mathbb{P}(\CC^4/\langle 
v_a,v_b\rangle)$. 
\begin{theorem}\label{thm:gr4toi31_2}
    We have the following identity modulo products
    \begin{equation}\label{eq:gr4i31_2}
    \begin{split}
    \frac{7}{144}\,\GR_4 \;\modsh\; \Alt_8 \Big[
    \II_{3,1}\Big(\frac{\rho_{1,2}\rho_{3,4}}
    {\rho_{3,2}\rho_{1,4}},
    \frac{\rho_1}{\rho_{1,4}}\Big)
    +2&\II_{3,1}\Big(\frac{\rho_{1,2}}
    {\rho_1},
    \frac{\rho_{3,2}}{\rho_{3,4}}\Big)
    +6\Lic_{4}\Big(\frac{\rho_1\rho_{3,2}}
    {\rho_{1,2}\rho_{3,4}}\Big)\Big]\,,
    \end{split}
    \end{equation}
    where we denote $\rho_{i,j}=\rho_i-\rho_j$.
\end{theorem}
The proof will be given in Section~\ref{sec:pf_i31}.
\begin{remark}\label{rem:gr4toi31}
Note that the combination inside the square brackets on the right-hand side
is $\Alt_8$-equivalent to the function 
$C_4(\rho_1,\rho_2,\rho_3,\rho_4,\infty,0)$ defined in the introduction.
The function $C_4$ essentially coincides with the map 
$L_4^{1}$,~\cite[eq.~(168)]{Go-Ru}, used
by Goncharov and Rudenko to construct a map from the chain complex of the Stasheff polytope to the polylogarithmic complex. This suggests a
connection between Grassmannian polylogarithms~$\GR_{2k}$ and the
cluster polylogarithm map in weight~$2k$ from~\cite{Go-Ru} (the latter
being a conjectural object for~$k\ge 3$). The formula for $\GR_5$
from Theorem~\ref{thm:gr5toi41} and Remark~\ref{rem:gr5toi41} below
suggests that something analogous might also be the case for Grassmannian
polylogarithms in odd weights.
\end{remark}

\smallskip
\subsection{\texorpdfstring{Explicit formula for $\GR_4$ minus coboundary in terms of $\Lic_4$}{Explicit formula for GR\textunderscore{}4 minus coboundary in terms of Li\textunderscore{}4}}
It is known that~$\GR_4$ cannot be written purely in terms
of~$\Lic_4$ (this is explained in~Section~1 in~\cite{Go1}),
so one cannot hope to completely remove~$\II_{3,1}$ from~\eqref{eq:gr4i31_2}.
Nevertheless, for the application to Zagier's Conjecture it is
important to have an expression in terms of~$\Lic_4$ for some
function that represents the same cohomology class (as before,
we interpret this naively: it has to be a function of the form
$\GR_4(v_1,\dots,v_8)-\Alt_8 f(v_1,\dots,v_7)$ for some~$f$).
In Theorem~\ref{thm:gr4toli4} we exhibit an explicit expression
of exactly this type. First we recall some results
about~$\II_{3,1}(x,y)$.
\begin{theorem}[\cite{Ga}, Prop.~22]\label{thm:i31syms}
    \begin{enumerate}
    	\item[(i)] Modulo products
    $\II_{3,1}(x,y)+\II_{3,1}(1-x,y)$ is equal to
    \begin{align*}
     \frac{1}{2}&\Lic_4\Big(\frac{x(1-y)}{y(1-x)}\Big)
    -\frac{1}{2}\Lic_4\Big(\frac{xy}{(1-x)(1-y)}\Big)
    +\frac{1}{2}\Lic_4\Big(\frac{y(1-y)}{x(1-x)}\Big)
    +\Lic_4\Big(\frac{y}{y-1}\Big)\\
    -&\Lic_4\Big(\frac{1-y}{x}\Big)
    -3\Lic_4\Big(\frac{y}{x}\Big)
    -3\Lic_4\Big(\frac{y}{1-x}\Big)
    -\Lic_4\Big(\frac{1-y}{1-x}\Big)\,.
    \end{align*}
    \item[(ii)] (Zagier) Modulo products the combination
    $\II_{3,1}(x,y)+\II_{3,1}(1/x,y)$ is equal to
    \begin{align*}
    -\frac{1}{2}&\Lic_4\Big(\frac{x(1-y)^2}{y(1-x)^2}\Big)
    +\frac{3}{2}\Lic_4\Big(\frac{1}{xy}\Big)
    -\frac{3}{2}\Lic_4\Big(\frac{y}{x}\Big)
    +2\Lic_4\Big(\frac{1-y}{1-x}\Big)\\
    +2&\Lic_4\Big(\frac{x(1-y)}{x-1}\Big)
    +2\Lic_4\Big(\frac{y-1}{y(1-x)}\Big)
    +2\Lic_4\Big(\frac{x(1-y)}{y(1-x)}\Big)
    -2\Lic_4\Big(\frac{1}{1-x}\Big)\\
    -2&\Lic_4\Big(\frac{x}{x-1}\Big)
    +2\Lic_4\Big(\frac{1}{1-y}\Big)
    -\Lic_4\Big(\frac{1}{y}\Big)
    +2\Lic_4\Big(\frac{y}{y-1}\Big) \,.
    \end{align*}
    \item[(iii)] We have
    \[\II_{3,1}(x,y)+\II_{3,1}(y,x) \modsh 0 \,.\]

\end{enumerate}
\end{theorem}
Combining these identities we obtain the following.
\begin{definition}
The function~$\wI_{3,1}(x,y)$ is defined by
    \[\wI_{3,1}(x,y) \coloneqq
    \frac{1}{36}\sum_{\sigma,\pi\in\mathfrak{S}_3}
    \sgn(\sigma)\sgn(\pi)\II_{3,1}(x^{\sigma},y^{\pi})
    \,,\]
where~$x^{\sigma}$ denotes elements of
$\{x,\frac{1}{x},1-x,1-\frac{1}{x},\frac{1}{1-x},\frac{x}{x-1}\}$ (the anharmonic group)
corresponding to $\sigma\in\mathfrak{S}_3$ (under some isomorphism).
\end{definition}
\begin{proposition} \label{prop:i31tilde}
    There is an explicit combination of $\Lic_4$ terms,
    denoted by $\text{\bf Sym}_{36}(x,y)$,
    such that
    \[\II_{3,1}(x,y)-\wI_{3,1}(x,y)\;\modsh\;
    \text{\bf Sym}_{36}(x,y) \,.\]
    Moreover, the function $\wI_{3,1}(x,y)$ satisfies
    \begin{align*}
    \wI_{3,1}(x^{\sigma},y^{\pi}) \;&\modsh\;
    \sgn(\sigma)\sgn(\pi)\wI_{3,1}(x,y) \,,\\
    \wI_{3,1}(y,x) \;&\modsh\;
    -\wI_{3,1}(x,y) \,.
    \end{align*}
\end{proposition}
For the sake of completeness we give an expression for $\symsix(x,y)$
in terms of~$\Lic_4$ in Appendix~\ref{app:i31}. The key property
of~$\II_{3,1}$ that we need is the following theorem that is the
main result of~\cite{Ga}.
\begin{theorem}[Gangl, \cite{Ga}] \label{thm:i31fiveterm}
    There is an explicit collection of
    rational functions $f_j\in\QQ(x,y,z)$
    and numbers~$c_j\in\QQ$ such that
    \[\wI_{3,1}
    \Big(z,[x]+[y]+\Big[\frac{1-x}{1-xy}\Big]+[1-xy]
    +\Big[\frac{1-y}{1-xy}\Big]\Big)
    \;\modsh\; \sum_{j=1}^{N} c_j\Lic_4(f_j(x,y,z)) \,.\]
\end{theorem}
We denote the left-hand side of the above expression
by $\fiveterm(z;x,y)$; by this theorem it is equal to an
explicit combination of~$\Lic_4$ terms. For the sake of
completeness we also reproduce the expression for $\fiveterm(z;x,y)$
in terms of~$\Lic_4$ in Appendix~\ref{app:i31}.

Note that due to the $6$-fold symmetries of~$\wI_{3,1}$ we have
that $\wI_{3,1}(x,\CR(ab| cdef))$ is symmetric in~$\{a,b\}$,
and skew-symmetric in~$\{c,d,e,f\}$ (recall that $\CR(ab| cdef)$
is the projected cross-ratio).
Finally, we extend~$\fiveterm$ by linearity to formal
linear combinations of arguments
    \[\fiveterm(a;\,\sum_j\nu_j[x_j;y_j]) \coloneqq
    \sum_j\nu_j\fiveterm(a;\,x_j,y_j)\,,\qquad\nu_{j}\in\QQ\,.\]
For simplicity we will write $[ab|cdef;ij|klmn]$ instead
of $[\CR(ab| cdef);\CR(ij| klmn)]$.
We are now ready to state our main result regarding the weight~$4$
Grassmannian polylogarithm.
\begin{theorem} \label{thm:gr4toli4}
    We have
    \begin{align}\label{eq:gr4toli4}
    \begin{split}
    \frac{7}{144}&\GR_4 + 2\Alt_8 \wI_{3,1}(\CR(34|2567),
    \CR(67|1345))\\
    \modsh \Alt_8 \Big[
    -&\fiveterm\big(\tfrac{\rho_4}{\rho_1};\,
    [\tfrac{\rho _{4,2}}{\rho_{4,1}};\tfrac{\rho_{4,1}}{\rho_{4,3}}]
    -            [43|2685;48|7653]
    +\tfrac{1}{4} [43|1256;43|1268]
    -\tfrac{1}{12}[43|1256;42|1365]\big)\\
    +&\fiveterm(\tfrac{\rho_2}{\rho_1};\,
    -            [43|2685;48|7653]
    +            [48|7235;48|7263]
    +\tfrac{1}{2}[46|5238;43|2568]) \\
    &+\symsix(\tfrac{\rho _{1,2} \rho _{3,4}}{\rho _{1,4} \rho _{3,2}},
    \tfrac{\rho_1}{\rho _{1,4}})
    +2\symsix(\tfrac{\rho _{1,2}}{\rho _1},
    \tfrac{\rho _{3,2}}{\rho_{3,4}})
    +6 \Lic_4(\tfrac{\rho _1 \rho _{3,2}}{\rho _{1,2} \rho _{3,4}})
    \Big]\,.
    \end{split}
    \end{align}
\end{theorem}
The proof is based on the identity from Theorem~\ref{thm:gr4toi31_2},
and will be given in Section~\ref{sec:pf_gr4toli4}.

\subsection{Corollaries}
Let us denote the formal linear combination of the arguments
of $\Lic_4$ in~\eqref{eq:gr4toli4} by~$\quadrat(v_1,\dots,v_8)$
(one can think of $\mathcal{Q}$ as a function from $\Conf{8}{4}$
to $\ZZ[\mathbb{P}^1(\CC)]$). This is an explicit form of the
map~$f_8(4)$ from~\cite{Ga2}, it is also related to the map~$r_8^*(4)$
from~\cite{Go-Ru} (more precisely, it is~$r_8^*(4)$ together with the
correction term used in the proof of Theorem~1.17 in~\cite{Go-Ru}).

Recall that the symbol of the left-hand side of~\eqref{eq:gr4toli4}
is equal, in view of~\eqref{grsymbol}, to $\Alt_{8}f(v_1,\dots,v_7)$
for some $f$.  So by symmetrizing over $9$ points we immediately obtain the following 
`cohomological' functional equation for~$\Li_4$.
\begin{corollary}
    The tetralogarithm function satisfies the following $9$-term relation
	\[\sum_{j=0}^{8}(-1)^j
	\Li_4(\quadrat(v_0,\dots,\widehat{v_j},\dots,v_{8})) = 0\,,\]
    where $v_0,\dots,v_8\in\CC^4$ are vectors in general position.
\end{corollary}
This functional equation is an analogue of the $5$-term
relation for the dilogarithm and of Goncharov's $840$-term
relation for the trilogarithm \cite{Go-ams}.

Moreover, if $v\in\CC^4$ is any non-zero vector,
then the function $\phi\colon\GL_4(\CC)^8\to\RR$ defined by
	\[\phi(g_1,\dots,g_8) = \Li_4(\quadrat(g_1v,\dots,g_8v))\]
defines a measurable $7$-cocycle for $\GL_4(\CC)$,
and, as explained for example in~\cite[\S1.4]{Go1} and in~\cite[\S9.7]{Bo2}, also
defines a continuous cohomology class
in~$H_{{\mathrm{cts}}}^{7}(\GL_4(\CC),\RR)$ (that is independent of~$v$).
The following result follows from Theorem~\ref{thm:gr4toli4}
together with the proofs of Th.~1.17 and Th.~1.2 in the work of
Goncharov and Rudenko~\cite[p.~72-73]{Go-Ru}.
\begin{corollary} \label{cor:borel}
	The cohomology class $[i\phi]$
	in $H_{{\mathrm{cts}}}^{7}(\GL_4(\CC),\RR(3))$
	is a non-zero rational multiple of the Borel
	class $b_4^{(4)}\in H_{{\mathrm{cts}}}^{7}(\GL_4(\CC),\RR(3))$
	(here, as before, $\RR(m)\coloneqq(2\pi i)^m\RR$).
\end{corollary}
Due to the validity of the Rank Conjecture for number fields~\cite{BoYa},
this result suffices to compute the
Borel regulator for $K_7(F)$ in terms of \( \Li_4 \), where $F$ is a number field
(for details see~\cite[\S9.4]{Bo2}).
\begin{remark}
    The number of different
    $\mathfrak{S}_8$-orbits of arguments in $\quadrat$ is at most $7N+2M+1$,
    where~$N$ is the number of terms in~$\fiveterm$
    and~$M$ is the number of terms in~$\symsix$. If we use the
    expressions for~$\fiveterm$ and~$\symsix$ given in
    Appendix~\ref{app:i31} below, then we get that~$\quadrat$
    is a sum of~$7\cdot2340 +2\cdot26+1=16433$ orbits.
    A slightly better version of $\fiveterm$ has $246$ terms, so we obtain
    ~$7\cdot246 +2\cdot26+1=1775$ orbits. However, some of
    the $\mathfrak{S}_8$-orbits in the resulting expression turn out to
    coincide, and some of them cancel out after skew-symmetrization
    over~$\mathfrak{S}_8$. A more careful (computer assisted) analysis
    of~$\quadrat$ using techniques similar to those used in~\cite{R}
    gives an expression with~$368$ orbits of $\Lic_4$ arguments.
\end{remark}

\medskip
\section{The Grassmannian polylogarithm in weight 5}
\label{sec:wt5}
Here we are now working with the configuration space $\Conf{10}{5}$
and, as before, we define
    \[\rho_{i}\coloneqq \rho_{\{i,i+1,i+2,i+3\}}^{(9,0)} =
    \frac{\detv{i,i+1,i+2,i+3,9}}{\detv{i,i+1,i+2,i+3,0}}\,,\qquad 1\le i \le 5\,,\]
where for notational reasons we use the digit $0$ for the $10$-th point.
Again we also use the following notation for projected cross-ratios
    \[\CR(abc| defg) \coloneqq
    \frac{\detv{abc df}\detv{abc eg}}{\detv{abc dg}\detv{abc ef}}\,.\]
Finally, we also use the following notation for the projected triple-ratio
(note that Goncharov's triple-ratio is the \( \Alt_6 \)-skew-symmetrization of this, and the single term is denoted by \( \R_3' \) in \cite{Go-ams})
    \[
	\R_3(ab| cde, fgh) \coloneqq
	\frac{\detv{abcdf} \detv{abceh}\detv{abdeg}}{
		\detv{abcdg}\detv{abcef}\detv{abdeh}} \,.
    \]
By analogy with Theorem~\ref{thm:gr4toi31_2} we have the following
expression for $\GR_5$ in terms of~$\II_{4,1}$ and~$\Lic_5$.
\begin{theorem}\label{thm:gr5toi41}
    We have the following identity modulo products
    \begin{equation}\label{eq:gr5i41}
    \begin{split}
    \frac{1}{640}\,\GR_5 \;\modsh\;
    {\Alt}_{10} \Big[
    &\II_{4,1}\Big(\frac{\rho_{2,3}}{\rho_{2,1}},
                  \frac{\rho_{4,3}}{\rho_{4,5}}\Big)
    +2\II_{4,1}\Big(\frac{\rho_{4,5}}{\rho_{4,1}},
                     \frac{\rho_{1,2}\rho_{3,4}}{\rho_{3,2}\rho_{1,4}}\Big)
    +8\Lic_{5}\Big(\frac{\rho_{1,2}\rho_{3,4}}
                        {\rho_{2,3}\rho_{4,5}}\Big) \\
    +&\II_{4,1}\Big(\frac{\rho_{1,2}\rho_{3,4}}
                        {\rho_{3,2}\rho_{1,4}},
    \frac{\rho_1}{\rho_{1,4}}\Big)
    +2\II_{4,1}\Big(\frac{\rho_{1,2}}{\rho_1},
                     \frac{\rho_{3,2}}{\rho_{3,4}}\Big)
    +8\Lic_{5}\Big(\frac{\rho_1\rho_{3,2}}
                        {\rho_{1,2}\rho_{3,4}}\Big)\Big]\,,
    \end{split}
    \end{equation}
    where we denote $\rho_{i,j}=\rho_i-\rho_j$.
\end{theorem}

\begin{remark}\label{rem:gr5toi41}
	If we define $C_{5}(x_1,\dots,x_6)$ as in the introduction,
    then the first row of \eqref{eq:gr5i41} is \( \Alt_{10}\) equivalent
    to \( C_{5}(\infty,\rho_1,\ldots,\rho_5) \) and the second row
    is \( \Alt_{10} \) equivalent to \( -C_5(\infty,0,\rho_2,\ldots,\rho_5) \).
     In particular, we see that the second line is a specialization of the first.
\end{remark}
This identity can be proved similarly to the proof of
Theorem~\ref{thm:gr4toi31_2}, as it is done in Section~\ref{sec:pf_i31}.
We have checked it symbolically
using a computer. Naively one could write out all $10!\cdot 6$ terms on the
right, and compute the symbol, but since both sides of the identity
are given by $\Alt_{10}$ of a small number of terms, it is much more manageable
to expand the symbol of a single term in the square brackets in terms of
irreducible~$\SL_5$-invariant polynomials, and then compute its canonical form
(say, lexicographically minimal representatives)
modulo the action of~$\mathfrak{S}_{10}$ for each of the tensors.

As for \( \GR_4 \), the non-vanishing of the 2-part of the cobracket
of \( \GR_5 \) shows that one cannot express it in terms of depth~$1$
functions alone. In analogy to the weight~4 case we first look for a `coboundary
correction' with  matching cobracket, which in this case amounts to a linear
combination of functions on configurations of 10 points in $\mathbb{P}^4$, where each
individual term depends on at most 9 of these points.

In order to be able to write a concise formula, we introduce
the following shorthand for the invariant and the anti-invariant under
the swap of the left and right triples inside the triple ratio
$$ \R_3^\pm(abc,de\hskip -1pt f) := [\R_3(abc,de\hskip -1pt f)] \pm [\R_3(de\hskip -1pt f,abc)]\,.$$
(Note that $\;1-\R_3(abc,de\hskip -1pt f)\;$
and $\;1-\R_3(de\hskip -1pt f,abc)\;$ share a nontrivial irreducible
factor.) We also adopt the notation $ \R_3^\pm(ab|cde,fgh)$ in line
with the standard notation for the projected version of the triple
ratio.

\begin{proposition}\label{prop:gr5del}
The following combination has a vanishing $2$-part of the cobracket
	\begin{equation} \label{eq:gr5del}
	\begin{split}
	\frac{1}{640}\,\GR_5 +
	\frac{5}{3} \Alt_{10} \Big[ & 2 \II_{4,1}(\CR(136|2459), \R_3^+(12|345,678))
	 + \II_{4,1}(\CR(346|1279), \R_3^-(12|345,678)) \Big] \,.
	\end{split}
	\end{equation}
\end{proposition}

According to the conjectured structure of the motivic Lie coalgebra
(of a field) in weight 5 as predicted by
Goncharov (see e.g.~\cite{Go1}, \cite{Go-ams}), this combination should
actually lie in depth~1, and we expect the following relationship
with the Borel class $b_5^{(5)}$ in weight~5.

\begin{conjecture} \label{conj:gr5del}
    \begin{enumerate}
    \item[(i)] There exists a formal linear combination $\quadrat_5(v_1,\dots,v_{10})$ of rational
    functions on $\Conf{10}{5}$
    such that \[\Lic_5(\quadrat_5(v_1,\dots,v_{10}))\]  equals \eqref{eq:gr5del}, modulo products.

    \item[(ii)] Assuming~(i) holds, the function $\phi\colon\GL_5(\CC)^{10}\to\RR$ defined by
    \[\phi(g_1,\dots,g_{10}) = \Li_5(\quadrat_5(g_1v,\dots,g_{10}v))\]
    is a bounded measurable $9$-cocycle for $\GL_5(\CC)$,
    whose corresponding continuous cohomology class is a
    non-zero rational multiple of the Borel class
    $b_5^{(5)}\in H_{{\mathrm{cts}}}^{9}(\GL_5(\CC),\RR)$.
    \end{enumerate}
\end{conjecture}

Explicitly, in view of~\eqref{eq:gr5i41} the combination
\begin{equation}\label{eqn:gr5pluscob}
\begin{split}
	\Alt_{10} \Big[ &\II_{4,1}\Big(\frac{\rho_{2,3}}{\rho_{2,1}},
	\frac{\rho_{4,3}}{\rho_{4,5}}\Big)
	+2\II_{4,1}\Big(\frac{\rho_{4,5}}{\rho_{4,1}},
	\frac{\rho_{1,2}\rho_{3,4}}{\rho_{3,2}\rho_{1,4}}\Big)
	+\II_{4,1}\Big(\frac{\rho_{1,2}\rho_{3,4}}
	{\rho_{3,2}\rho_{1,4}},
	\frac{\rho_1}{\rho_{1,4}}\Big)
	+2\II_{4,1}\Big(\frac{\rho_{1,2}}{\rho_1},
	\frac{\rho_{3,2}}{\rho_{3,4}}\Big)
 \\
 	& + \frac{10}{3} \II_{4,1}(\CR(136|2459), \R_3^+(12|345,678))
 	+ \frac{5}{3} \II_{4,1}(\CR(346|1279), \R_3^-(12|345,678))
	 \Big] \,
\end{split}
\end{equation}
should be reducible to purely \( \Lic_5 \) terms.

In order to state this more precisely, we will decompose the above combination into pieces expressed in terms of \( \II_{4,1}^+(x,y) \),
where
    \begin{equation}\label{eq:i41plusdef}
    \II_{4,1}^+(x,y) \coloneqq
    \frac{1}{2} \big( \II_{4,1}(x,y) + \II_{4,1}(x,y^{-1}) \big) \,.
    \end{equation}
The conjectural behaviour of $\II_{4,1}^{+}$ under \( \Lic_2 \)
functional equations in~$x$, and under~$\Lic_3$ functional equations
in~$y$ should play a key part in any explicit reduction of~$\GR_5$ to~$\Lic_5$.
By a conjecture of Goncharov about the structure of the motivic Lie coalgebra
in weight~$5$ (\cite[\S4.5,~\S4.8]{Go-ams}), we
expect \( \II_{4,1}^+(x,y) \) to reduce to \( \Lic_5 \) under
dilogarithm functional equations in \( x \), and
also under trilogarithm functional equations in \( y \)
(that is, $\II_{4,1}^+(x,y)$ modulo $\Lic_5$ should satisfy the same relations as
$\Lic_2(x)\otimes\Lic_3(y)$).

We can express \( \II_{4,1}(x,y) \) in terms of \( \II_{4,1}^+ \) as follows
\begin{equation}\label{eqn:i41toi41p}
	\II_{4,1}(x,y) \modsh \II_{4,1}^+(x, y) + \II_{4,1}^+(y, x) + \frac{1}{2} \Lic_5(x) + \frac{1}{2} \Lic_5(y) + 2 \Lic_5(x y) \,.
\end{equation}
Let us give some partial results supporting the conjectural behavior of~$\II_{4,1}^{+}$. By analogy to Theorem~\ref{thm:i31syms} for \( \II_{3,1} \)
we have the following.
\begin{theorem} Modulo products and explicit \( \Lic_5 \) terms, the
function \( I_{4,1}^+(x,y) \) satisfies the dilogarithm $6$-fold symmetries
in \( x \), and the \( \Lic_3 \) inversion and three-term relation in \( y \):
	\begin{enumerate}
		\item \( \II_{4,1}^+(x,y) + \II_{4,1}^+(x^{-1},y) = 0 \Mod{$\Lic_5$, products} \),
		\item \( \II_{4,1}^+(x,y) + \II_{4,1}^+(1-x,y) = 0 \Mod{$\Lic_5$, products} \),
		\item \( \II_{4,1}^+(x,y) - \II_{4,1}^+(x,y^{-1}) = 0 \Mod{$\Lic_5$, products} \), and
		\item \( \II_{4,1}^+(x,y) + \II_{4,1}^+(x,1-y)+ \II_{4,1}^+(x,1-y^{-1})- \II_{4,1}^+(x,1) = 0 \Mod{$\Lic_5$, products} \).
	\end{enumerate}
\end{theorem}
Explicit expressions for these identities are given in Appendix~\ref{app:i41}.
Additionally, in~\cite{CGR} we have established the~$5$-term relation
for $\II_{4,1}^{+}(x,1)=\II_{4,1}(x,1) $ in $x$ (modulo $\Lic_5$ and products), or rather the Nielsen polylogarithm \( S_{3,2}(x) \modsh I_{4,1}(x,1) + 4 \Lic_5(x) \),
although we will not use this.
Analogously to the weight~$4$ case it is convenient to introduce the
following symmetrization of~$\II_{4,1}^{+}$.
\begin{definition}
    Let us define
    \[\wI_{4,1}^+(x,y) \coloneqq
    \frac{1}{6}\sum_{\sigma\in\mathfrak{S}_3}
    \sgn(\sigma)\II_{4,1}^+(x^{\sigma},y)
    \,,\]
    where~$x^{\sigma}$ denotes elements of
    $\{x,1/x,1-x,1-1/x,1/(1-x),x/(x-1)\}$ (the anharmonic group)
    corresponding to $\sigma\in\mathfrak{S}_3$ (under some isomorphism).
\end{definition}

\begin{proposition} \label{prop:i41tilde}
	The combination
	\[\II_{4,1}^+(x,y)-\wI_{4,1}^+(x,y) \]
	is equal modulo products to an explicit sum of \( \Lic_5 \) terms.  Moreover, the function $\wI_{4,1}^+(x,y)$ satisfies
	\begin{align*}
	\wI_{4,1}^+(x^{\sigma},y^{\pm 1}) \;&\modsh\;
	\sgn(\sigma)\wI_{4,1}^+(x,y) \,.
	\end{align*}
\end{proposition}

We now claim that, after converting \( \II_{4,1} \) to \( \II_{4,1}^+ \) via \eqref{eqn:i41toi41p}, the combination in \eqref{eqn:gr5pluscob} breaks up into 3 non-trivial subsums which are combinations of \( \Lic_2 \) and \( \Lic_3 \) functional equations, and a piece that constitutes a trivial coboundary.
\begin{proposition} \label{prop:gr5cobdecomposition}
	The expression~\eqref{eq:gr5del} breaks up as follows
    	\begin{equation} \label{eq:propgr5}
    	\eqref{eq:gr5del} \;\modsh\;
        \Alt_{10}[A+B+C+D]+\text{\normalfont explicit $\Lic_5$'s}\,,
    	\end{equation}
    where
            \begin{align*}
        A ={} &\II_{4,1}^+\Big(\frac{\rho _{2,3}}{\rho _{2,1}},\frac{\rho _{4,3}}{\rho_{4,5}}\Big)
        -\II_{4,1}^+\Big(\frac{\rho _{4,3}}{\rho _{4,5}},\frac{\rho _{2,3}}{\rho _{2,1}}\Big)\,, \\[1ex]
        B ={} &2\II_{4,1}^+\Big(\frac{\rho _{1,2} \rho _{3,4}}{\rho _{1,4} \rho _{3,2}},\frac{\rho _{4,5}}{\rho_{4,1}}\Big)\,,	\\[1ex]
        C ={} &
        2 \II_{4,1}^+\Big(\frac{\rho _{4,3}}{\rho _{4,5}},\frac{\rho _{2,3}}{\rho_{2,1}}\Big)
        +2 \II_{4,1}^+\Big(\frac{\rho _{4,5}}{\rho _{4,1}},\frac{\rho _{1,2} \rho _{3,4}}{\rho	_{1,4} \rho _{3,2}}\Big)
        +\II_{4,1}^+\Big(\frac{\rho _1}{\rho _{1,4}},\frac{\rho _{1,2} \rho _{3,4}}{\rho _{1,4}\rho _{3,2}}\Big)
        +2\II_{4,1}^+\Big(\frac{\rho _{1,2}}{\rho _1},\frac{\rho _{3,2}}{\rho_{3,4}}\Big)
        \\&
        -\frac{4}{3} \II_{4,1}^+\Big(\frac{\rho _{3,2}}{\rho _{3,4}},\frac{\rho_{1,2}}{\rho _1}\Big)
        + \frac{5}{3} \II_{4,1}^+(\CR(346|1279),\R_3^-(12|345,678)) \,, \\[1ex]
        D ={} & \II_{4,1}^+\Big(\frac{\rho _{1,2} \rho _{3,4}}{\rho _{1,4}\rho _{3,2}},\frac{\rho _1}{\rho _{1,4}}\Big)
        +\frac{10}{3}\II_{4,1}^+\Big(\frac{\rho _{3,2}}{\rho_{3,4}},\frac{\rho _{1,2}}{\rho _1}\Big)
        + \frac{10}{3} \II_{4,1}(\CR(136|2459),\R_3^+(12|345,678))
        \\&  + \frac{5}{3} \II_{4,1}^+(\R_3^-(12|345,678),\CR(346|1279))
        \end{align*}
    have the following properties. The expressions $A$, $B$, and $C$ can be explicitly written as \( I_{4,1}^+ \) of \( \Lic_2 \) functional equations in the first argument, and \( \Lic_3 \) functional equations in the second argument.  The expression~$D$ is a coboundary term.
\end{proposition}
We prove this Proposition in Section~\ref{sec:proofpropgr5}.
We expect the expression~$D$ to also have a decomposition in terms of $\Lic_2$
and $\Lic_3$ functional equations, but since it is a coboundary, we can simply
add it to the last~$4$ terms (which are also coboundaries)
in~\eqref{eqn:gr5pluscob}.

\medskip
\section{Proof of Theorem~\protect\ref{thm:i211}}
\label{sec:proof_i211}
To prove the result we will compute the symbol of the right-hand side
of~\eqref{eq:i211} and show that it is equal to a constant multiple
of~\eqref{grsymbol}. To compute the symbol we repeatedly apply the $(n-1,1)$-part of
the coproduct to the expression on the right of~\eqref{eq:i211}.
After the first application we obtain
    \begin{align*}
    \Alt_{2m}\bigg[
    \II(0;\rho_1,\dots;\rho_m)
    \otimes\rho_1
    + \II(0;0,\rho_2,\dots;\rho_m)
    \otimes\Big(1-\frac{\rho_2}{\rho_1}\Big)\\
    + \sum_{j=2}^{m-1}
    \II(0;0,\rho_1,\dots,\widehat{\rho_j},\dots;\rho_m)
    \otimes\Big(\frac{\rho_{j}-\rho_{j+1}}{\rho_{j}-\rho_{j-1}}\Big)
    \bigg]\,.
    \end{align*}
Let us denote
    \begin{equation*}
    \begin{array}{ll}
    A_j = \detv{j,\dots,j+m-2,j+m-1}\,,&
    B_j = \detv{j,\dots,j+m-2,2m-1}\,,\\
    C_j = \detv{j+1,\dots,j+m-2,j+m-1,2m}\,,&
    D_j = \detv{j+1,\dots,j+m-2,2m-1,2m}\,
    \end{array}
    \end{equation*}
(note the offset in the definition of $C_j$ and $D_j$).
Then for $2\le j\le m-1$ we have the following factorization
(that follows from two applications of projected Pl\"ucker relations)
    \[      \frac{\rho_j-\rho_{j+1}}{\rho_j-\rho_{j-1}}
    =       \frac{C_{j-2}D_{j}}{C_{j}A_{j-1}}
      \cdot \frac{A_{j}}{D_{j-1}}\,, \]
where the first factor is fixed by the transposition $(j,j-1)$, and the
second one is fixed (up to sign) by $(j+m-2,j+m-1)$. Since
$\II(0;0,\rho_1,\dots,\widehat{\rho_j},\dots;\rho_m)$ is trivially
invariant under both of these transpositions, we obtain that
    \[\Alt_{2m}\bigg[\II(0;0,\rho_1,\dots,\widehat{\rho_j},\dots;\rho_m)
    \otimes\Big(\frac{\rho_{j}-\rho_{j+1}}{\rho_{j}-\rho_{j-1}}\Big)\bigg]
    = 0\,,\quad 2\le j \le m-1\,.\]
This leaves us with
    \begin{align*}
    \Alt_{2m}\bigg[\II(0;\rho_1,\dots,\rho_{m-1};\rho_m)\otimes \rho_1
    +\II(0;0,\rho_2,\dots,\rho_{m-1};\rho_m)\otimes
    \Big(1-\frac{\rho_2}{\rho_1}\Big)
    \bigg]\,.
    \end{align*}
Next, after the second application of the coproduct we get
    \begin{align*}
    \Alt_{2m}\bigg[
    \II(0;\rho_2,\dots,\rho_{m-1};\rho_m)\otimes \big(1-\tfrac{\rho_2}{\rho_1}\big)
    \otimes \rho_1
    +\sum_{j=2}^{m-1}\II(0;\rho_1,\dots,\widehat{\rho_j},\dots;\rho_m)
    \otimes\big(\tfrac{\rho_{j}-\rho_{j+1}}{\rho_{j}-\rho_{j-1}}\big)
    \otimes\rho_1\\
    + \II(0;\rho_2,\dots,\rho_{m-1};\rho_m)
    \otimes \rho_2\otimes
    \big(1-\tfrac{\rho_2}{\rho_1}\big)
    + \II(0;0,\rho_3,\dots,\rho_{m-1};\rho_m)
    \otimes \big(1-\tfrac{\rho_3}{\rho_2}\big)
    \otimes \big(1-\tfrac{\rho_2}{\rho_1}\big)\\
    +\sum_{j=3}^{m-1}\II(0;0,\rho_2,\dots,\widehat{\rho_j},\dots;\rho_m)
    \otimes\big(\tfrac{\rho_{j}-\rho_{j+1}}{\rho_{j}-\rho_{j-1}}\big)
    \otimes\big(1-\tfrac{\rho_2}{\rho_1}\big)\bigg]\,.
    \end{align*}
By the same reasoning as before, we see that after skew-symmetrization
each of the two sums above cancels out termwise, so we are left with
    \begin{align*}
    \Alt_{2m}\left[
    \II(0;\rho_2,\dots,\rho_{m-1};\rho_m)
    \otimes \big(1-\tfrac{\rho_2}{\rho_1}\big) \otimes \rho_1
    +\II(0;\rho_2,\dots,\rho_{m-1};\rho_m)
    \otimes \rho_2\otimes \big(1-\tfrac{\rho_2}{\rho_1}\big) \right.\\
    \left.+\II(0;0,\rho_3,\dots,\rho_{m-1};\rho_m)
    \otimes \big(1-\tfrac{\rho_3}{\rho_2}\big)
    \otimes \big(1-\tfrac{\rho_2}{\rho_1}\big)
    \right]\,.
    \end{align*}
Continuing in this fashion we arrive at the following expression for the
symbol of the RHS of~\eqref{eq:i211}
    \begin{equation} \label{eq:i211symb}
    \Alt_{2m}\bigg[
    \sum_{j=1}^{m}\Big(1-\frac{\rho_m}{\rho_{m-1}}\Big) \otimes
    \Big(1-\frac{\rho_{m-1}}{\rho_{m-2}}\Big) \otimes\dots
    \otimes \Big(1-\frac{\rho_{j+1}}{\rho_{j}}\Big)
    \otimes \rho_j
    \otimes \Big(1-\frac{\rho_j}{\rho_{j-1}}\Big) \otimes\dots
    \otimes \Big(1-\frac{\rho_2}{\rho_1}\Big) \bigg]\,,
    \end{equation}
where $\rho_j$ is inserted in the $j$-th position from the right.
For $m=2$ we directly compute
    \[\Alt_{4}\bigg[
    \Big(1-\frac{\rho_2}{\rho_{1}}\Big) \otimes \rho_1
    +\rho_2 \otimes \Big(1-\frac{\rho_2}{\rho_{1}}\Big)
    \bigg]
    = -12\Alt_{4}\Big[\detv{12}\otimes \detv{23}\Big]\,,\]
which proves~\eqref{eq:i211} in this case.
From now on we assume that $m\ge3$. Noting the following factorizations
    \[\rho_j = \frac{B_{j}}{C_{j-1}}\,,\quad\quad\quad
    1-\frac{\rho_{j+1}}{\rho_j} = \frac{A_{j}D_{j}}{B_{j}C_{j}}\,, \]
let us look at the $k$-th term of the sum in~\eqref{eq:i211symb}:
    \begin{align*}
    &\frac{A_{m-1}D_{m-1}}{B_{m-1}C_{m-1}}
    \otimes \dots
    \otimes \frac{A_{k}D_{k}}{B_{k}C_{k}}
    \otimes \frac{B_k}{C_{k-1}}
    \otimes \frac{A_{k-1}D_{k-1}}{B_{k-1}C_{k-1}}
    \otimes \dots
    \otimes \frac{A_{1}D_{1}}{B_{1}C_{1}}\,\\
    =\, &\frac{A_{m-1}D_{m-1}}{B_{m-1}C_{m-1}}
    \otimes \dots
    \otimes \frac{A_{k}D_{k}}{B_{k}C_{k}}
    \otimes B_k
    \otimes \frac{A_{k-1}D_{k-1}}{B_{k-1}C_{k-1}}
    \otimes \dots
    \otimes \frac{A_{1}D_{1}}{B_{1}C_{1}}\,\\
    -\, &\frac{A_{m-1}D_{m-1}}{B_{m-1}C_{m-1}}
    \otimes \dots
    \otimes \frac{A_{k}D_{k}}{B_{k}C_{k}}
    \otimes C_{k-1}
    \otimes \frac{A_{k-1}D_{k-1}}{B_{k-1}C_{k-1}}
    \otimes \dots
    \otimes \frac{A_{1}D_{1}}{B_{1}C_{1}}\,
    \end{align*}
(for $k=1$ and $k=m$ we omit the terms $\frac{A_{m}D_{m}}{B_{m}C_{m}}$
and $\frac{A_{0}D_{0}}{B_{0}C_{0}}$, respectively). In the above expression
each term of the type $\dots\otimes C_k\otimes B_k\otimes \dots$
will appear in the $k$-th and $(k-1)$-st terms with opposite signs
and hence these terms cancel out. Therefore, we need to compute the sum
$\sum_{k=1}^{m}(T_k^{B}-T_k^{C})$, where the~$T_k^{*}$ are given by
    \begin{equation} \label{eq:TkBC}
    \begin{split}
    T_k^{B} &= \frac{A_{m-1}D_{m-1}}{B_{m-1}C_{m-1}}
    \otimes \dots
    \otimes \frac{A_{k}D_{k}}{B_{k}}
    \otimes B_k
    \otimes \frac{A_{k-1}D_{k-1}}{B_{k-1}C_{k-1}}
    \otimes \dots
    \otimes \frac{A_{1}D_{1}}{B_{1}C_{1}}\,,\\
    T_k^{C} &= \frac{A_{m-1}D_{m-1}}{B_{m-1}C_{m-1}}
    \otimes \dots
    \otimes \frac{A_{k}D_{k}}{B_{k}C_{k}}
    \otimes C_{k-1}
    \otimes \frac{A_{k-1}D_{k-1}}{C_{k-1}}
    \otimes \dots
    \otimes \frac{A_{1}D_{1}}{B_{1}C_{1}}\,.
    \end{split}
    \end{equation}
For $\pi\in\mathfrak{S}_{2m}$ we denote by $\sigma_{\pi}$ the
automorphism of the ring of regular functions on $\Conf{2m}{m}$
induced by $\pi$. It is easy to check that the involution
    \begin{equation}\label{eq:thm3inv}
    a_1\otimes\dots\otimes a_m\mapsto
    \sigma_{\pi}(a_m)\otimes\dots\otimes \sigma_\pi(a_1)\,,
    \end{equation}
where $\pi=(2m-1,2m)(1,2m-2)(2,2m-3)\dots(m-1,m)$
interchanges $T_k^{B}$ and $T_{m+1-k}^{C}$ (for now we ignore the sign).
Indeed, the permutation~$(2m-1,2m)$ interchanges $B_j$ and $C_{j-1}$,
the permutation $(1,2m-2)(2,2m-3)\dots(m-1,m)$
maps $X_j\leftrightarrow X_{m-j}$ for $X=A,D$,
$B_j\leftrightarrow B_{m+1-j}$, and $C_j\leftrightarrow C_{m-j-1}$,
so their composition maps
$B_j\leftrightarrow C_{m-j}$ and $A_{j}\leftrightarrow A_{m-j}$,
$D_{j}\leftrightarrow D_{m-j}$, and after reversing the order
of the tensors we get the involution $T_k^B\leftrightarrow T_{m+1-k}^C$.
Thus, it is enough to calculate the terms $T_{k}^{B}$
modulo $\Alt_{2m}$ for $k=1,\dots,m$. We will expand $T_k^B$
into a sum of ``elementary tensors'' and describe all
such tensors that are not fixed by any transposition
in~$\mathfrak{S}_{2m}$ (since any such term will vanish after
skew-symmetrization). We have the following lemma.
\begin{lemma} \label{lem:ABCDs}
    Let $w=X_{m-1}\otimes\dots\otimes X_{k}\otimes B_k\otimes
    X_{k-1}\otimes\dots\otimes X_1$, where $X_i\in \{A_i,B_i,C_i,D_i\}$,
    be any term in the expansion of $T_{k}^B$, $1\le k\le m$, that is not
    fixed by any transposition, and let~$2\le j\le m-1$ be such that $j\ne k$.
    Then
    \begin{itemize}
    \item[(i)] if $X_{j-1}\in\{A_{j-1},B_{j-1}\}$,
    then $X_{j}\in\{A_{j},B_{j}\}$;
    \item[(ii)] if $X_{j-1}\in \{B_{j-1},D_{j-1}\}$,
    then $X_{j}\in\{B_{j},D_{j}\}$;
    \item[(iii)] if $X_{j}\in \{C_{j},D_{j}\}$,
    then $X_{j-1}\in\{C_{j-1},D_{j-1}\}$;
    \item[(iv)] if $X_{j}\in \{A_{j},C_{j}\}$,
    then $X_{j-1}\in\{A_{j-1},C_{j-1}\}$;
    \item[(v)] if $k<m$ and $X_{1}\in \{A_{1},C_{1}\}$,
    then $X_{m-1}\in\{C_{m-1},D_{m-1}\}$;
    \item[(vi)] if $k<m$ and $X_{1}\in \{C_{1},D_{1}\}$,
    then $X_{m-1}\in\{A_{m-1},C_{m-1}\}$.
    \end{itemize}
\end{lemma}
(Note that for $k=m$ we have $w=B_m\otimes X_{m-1}\otimes\dots\otimes
X_1$.)
\begin{proof}
For $2\le j\le m-1$ the only elements
of $S=\{B_m\}\cup\bigcup_{i=1}^{m-1}\{A_i,B_i,C_i,D_i\}$
that are not fixed (up to sign) by the transposition $(j-1,j)$
are $A_j,B_j,C_{j-1},D_{j-1}$, and hence
if $X_{j-1}\not\in\{C_{j-1},D_{j-1}\}$, then one of $A_j,B_j$
has to appear among $X_i$, thus $X_j\in\{A_j,B_j\}$, proving~(i).
Similarly, the only elements of $S$ that are not fixed
by the transposition $(j+m-2,j+m-1)$ are $A_{j-1},C_{j-1},B_{j},D_{j}$,
and from this we obtain~(ii). Parts~(iii) and~(iv) are simply
the contrapositives of~(i) and~(ii). To prove~(v) and~(vi) we look
at the transpositions $(m-1,m)$ and $(1,2m-2)$, whose
only non-fixed elements are $B_1,D_1,C_{m-1},D_{m-1},B_m$
and $A_1,B_1,A_{m-1},C_{m-1},B_m$ respectively.
\end{proof}
Next, we consider two cases.

\smallskip
\textbf{Case $k<m$.}
We claim that in the expansion of $T_k^B$ every term is fixed by some
transposition except for the following $2k$ terms:
    \begin{equation} \label{eq:thm3case1}
    \begin{split}
    &v_{j,k}^{A}=(-1)^{k-j}A_{m-1}\otimes\dots\otimes A_{k}\otimes
    B_k\otimes B_{k-1}\otimes\dots\otimes B_{j}\otimes D_{j-1}
    \otimes \dots \otimes D_{1}\,,\\
    &v_{j,k}^{D}=(-1)^{k-j}D_{m-1}\otimes\dots\otimes D_{k}\otimes
    B_k\otimes B_{k-1}\otimes\dots\otimes B_{j}\otimes A_{j-1}
    \otimes \dots \otimes A_{1}\,,
    \end{split}
    \end{equation}
for $1\le j \le k$.
Let $w=X_{m-1}\otimes\dots\otimes X_{k}\otimes B_k\otimes
X_{k-1}\otimes\dots\otimes X_1$ be any term in the expansion
of $T_k^B$ that is not fixed by any transposition
(we ignore the coefficient with which $w$ appears in $T_k^{B}$).
First, note that if any $X_i=C_i$, then so must be $X_{k}=C_k$ by
repeatedly using parts (iii)-(vi) of the lemma, but looking
at~\eqref{eq:TkBC} we see that $X_k\ne C_k$.
Thus there are no $C_i$ among $X_i$'s. Next,
if $X_1=D_1$, then $X_{m-1}=A_{m-1}$ by~(vi), and by parts (i) and (ii)
we must have $X_i=A_i$, $i\ge k$, and $X_i=D_i$ for $i<j$ for some
$k\ge j > 1$, i.e., $v_{j,k}^A$ for some $j$.
Similarly, if $X_1=A_1$, then we get $v_{j,k}^D$ for some $j$.
Finally, let $X_1=B_1$.
Then by~(i) and~(ii) we have
$w=X_{m-1}\otimes\dots\otimes X_{k}\otimes B_k\otimes
B_{k-1}\otimes\dots\otimes B_1$. If $X_k=B_k$, then $X_i=B_i$
for all~$i$, but then~$w$ would be fixed by the
transposition $(m-1,2m-1)$. If $X_k=A_k$, then by~(i) and~(ii)
we have $X_i=A_i$ for $i=k,\dots,l$, and $X_i=B_i$ for $l<i\le m-1$.
In this case if $l<m-1$, then $(2m-2,2m)$ fixes~$w$. So the only
term that is not fixed by transpositions in this case is~$v_{1,k}^{A}$.
Similarly, in the case $X_k=D_k$ we get~$v_{1,k}^{D}$.

\smallskip
\textbf{Case $k=m$.} In this case parts~(i) and~(ii) of the
above lemma immediately imply that in the expansion of~$T_m^B$
only the following~$m^2$ terms are not fixed by any transposition:
    \begin{equation} \label{eq:thm3case2}
    \begin{split}
    &w_{i,j}^{A}=(-1)^{m-j+i-1}B_{m}\otimes\dots\otimes B_{j}\otimes
    A_{j-1}\otimes A_{j-2}\otimes\dots\otimes A_{i}\otimes C_{i-1}
    \otimes \dots \otimes C_{1}\,, \\
    &w_{i,j}^{D}=(-1)^{m-j+i-1}B_{m}\otimes\dots\otimes B_{j}\otimes
    D_{j-1}\otimes D_{j-2}\otimes\dots\otimes D_{i}\otimes C_{i-1}
    \otimes \dots \otimes C_{1}\,,
    \end{split}
    \end{equation}
for $1\le i \le j \le m$  (here $w_{i,i}^{A}=w_{i,i}^{D}$,
so we only count them once).

\smallskip
What is left to show is that all of the terms in~\eqref{eq:thm3case1}
and~\eqref{eq:thm3case2} are $\Alt_{2m}$-equivalent to
    \begin{equation}\label{eq:RRR}
    R = (-1)^{m-1}\detv{1,\dots,m} \otimes \detv{2,\dots,m+1}
    \otimes \dots \otimes \detv{m,\dots,2m-1}\,.
    \end{equation}
First, we show that all terms are pairwise $\Alt_{2m}$-equivalent.
The permutation $(j+m-2,2m)$ sends $v_{j,k}^{A}$ to $v_{j-1,k}^{A}$,
and the permutation $(j-1,2m-1)$ sends $v_{j,k}^{D}$ to $v_{j-1,k}^{D}$,
thus each $v_{j,k}^{*}$ is equivalent to $v_{1,k}^{*}$.
Similarly, the permutation $(i-1,2m)$ sends $w_{i,j}^{A}$
to $w_{i-1,j}^{A}$ and the permutation $(i+m-2,2m-1)$
sends $w_{i,j}^{D}$ to $w_{i-1,j}^{D}$, thus each $w_{i,j}^{*}$
is equivalent to $w_{1,j}^{*}$. Next, applying $(k,2m-1)$
to $v_{1,k}^{A}$ we get $v_{1,k+1}^{A}$, and applying $(k+m-1,2m)$
to $v_{1,k}^{D}$ we get $v_{1,k+1}^{D}$, where we set
both $v_{1,m}^{A}$ and $v_{1,m}^D$ to be equal to $w_{1,1}^{A}$.
Similarly, applying $(j+m-2,2m-1)$ to $w_{1,j}^{A}$ we get $w_{1,j-1}^{A}$,
and applying $(j-1,2m)$ to $w_{1,j}^{D}$ we get $w_{1,j-1}^{D}$.
Thus all terms are equivalent to $w_{1,1}^{A}$, from which
we get~$R$ after first applying $(2m-1,2m-2,\dots,m+1,m)$ and
then $(1,2m-1)(2,2m-2)\dots (m-1,m+1)$
(both permutations have sign $(-1)^{m-1}$).

This shows that $\Alt_{2m}(\sum_{k=1}^{m}T_j^{B})$ is equal to
    \[m(2m-1)(-1)^{m-1}\Alt_{2m}
    \Big[\detv{1,\dots,m} \otimes \detv{2,\dots,m+1}
    \otimes \dots \otimes \detv{m,\dots,2m-1}\Big]\,.\]
Since the transformation~\eqref{eq:thm3inv} maps each term on
the right hand side of the above equation to its negation
(modulo $\Alt_{2m}$), and it maps $T_k^{B}$ to $T_{m+1-k}^{C}$,
we get that $\Alt_{2m}(\sum_{k=1}^{m}T_j^{C})$ equals
    \[m(2m-1)(-1)^{m}
    \mathrm{Alt}_{2m}\Big[\detv{1,\dots,m} \otimes \detv{2,\dots,m+1}
    \otimes \dots \otimes \detv{m,\dots,2m-1}\Big]\,.\]
Combining these two identities together with~\eqref{grsymbol}
we obtain the claim of the theorem.

\begin{remark}
    The proof of~\eqref{eq:i211inf} is analogous,
    except that~\eqref{eq:i211symb} becomes
    \begin{equation*}
    \Alt_{2m}\bigg[
    \sum_{j=1}^{m}(\rho_{m}-\rho_{m-1}) \otimes\dots
    \otimes (\rho_{j+1}-\rho_{j})
    \otimes \rho_j
    \otimes\Big(1-\frac{\rho_j}{\rho_{j-1}}\Big) \otimes\dots
    \otimes \Big(1-\frac{\rho_2}{\rho_1}\Big) \bigg]\,,
    \end{equation*}
    and the rest of the combinatorial analysis needs to be changed accordingly.
    We then again use
    \begin{align*}
    &\frac{A_{m-1}D_{m-1}}{C_{m-1}C_{m-2}}
    \otimes \dots
    \otimes \frac{A_{k}D_{k}}{C_{k}C_{k-1}}
    \otimes \frac{B_k}{C_{k-1}}
    \otimes \frac{A_{k-1}D_{k-1}}{B_{k-1}C_{k-1}}
    \otimes \dots
    \otimes \frac{A_{1}D_{1}}{B_{1}C_{1}}\,\\
    =\, &\frac{A_{m-1}D_{m-1}}{C_{m-1}C_{m-2}}
    \otimes \dots
    \otimes \frac{A_{k}D_{k}}{C_{k}C_{k-1}}
    \otimes B_k
    \otimes \frac{A_{k-1}D_{k-1}}{B_{k-1}C_{k-1}}
    \otimes \dots
    \otimes \frac{A_{1}D_{1}}{B_{1}C_{1}}\,\\
    -\, &\frac{A_{m-1}D_{m-1}}{C_{m-1}C_{m-2}}
    \otimes \dots
    \otimes \frac{A_{k}D_{k}}{C_{k}C_{k-1}}
    \otimes C_{k-1}
    \otimes \frac{A_{k-1}D_{k-1}}{B_{k-1}C_{k-1}}
    \otimes \dots
    \otimes \frac{A_{1}D_{1}}{B_{1}C_{1}}\,,
    \end{align*}
    where again the terms of the form $\dots\otimes C_k\otimes B_k\otimes \cdots$
    cancel between consecutive values of~$k$.
    All other terms in the expansion are either invariant
    under a transposition and thus are $\Alt_{2m}$-equivalent to $0$
    or are $\Alt_{2m}$-equivalent to~$R$, where $R$ is defined
    in~\eqref{eq:RRR}.
\end{remark}

\medskip
\section{Proof of Theorem~\protect\ref{thm:aomoto}}
\label{sec:aomoto}
To simplify notation, in this section we will write $\rho_i$ for $\rho_{i}^{(1,2m)}$,
contrary to the convention throughout the rest of the paper.
The proof follows the same strategy as the proof of Theorem~\ref{thm:i211} above:
to prove~\eqref{eq:aomoto} we calculate the symbol of
	\begin{equation} \label{eq:aomotocand}
	\Alt_{m,m}\Big[\II(0;\rho_{m+1},\rho_{m},\dots,\rho_{3};\rho_{2})\Big]
	\end{equation}
by repeatedly applying the $(n-1,1)$-part of the coproduct to it. After the first
application we obtain
	\[\Alt_{m,m}
	\Big[\II(0;\rho_{m},\dots, \rho_{3};\rho_{2})\otimes
	\frac{\rho_{m+1}-\rho_{m}}{\rho_{m+1}}
	+\sum_{j=3}^{m}\II(0;\rho_{m+1},\dots,\widehat{\rho_j},\dots\rho_{3};\rho_{2})
	\otimes	\frac{\rho_{j}-\rho_{j-1}}{\rho_{j}-\rho_{j+1}}\Big]\,.\]
As in the proof of Theorem~\ref{thm:i211}, we note that the
$\frac{\rho_{j}-\rho_{j-1}}{\rho_{j}-\rho_{j+1}}$
decomposes into a product of two terms, one fixed by the involution $(j-1,j)$,
and the other by $(j+m-2,j+m-1)$. Since both of these involutions fix
$\II(0;\rho_{m+1},\dots,\widehat{\rho_j},\dots,\rho_{3};\rho_{2})$ each term
in the sum $\sum_{j=3}^{m}$ vanishes under $\Alt_{m,m}$. Applying the
$(n-1,1)$-part of the coproduct again to the remaining term we get
	\[\Alt_{m,m}
	\Big[\Big(\II(0;\rho_{m-1},\dots, \rho_{3};\rho_{2})\otimes
	\frac{\rho_{m}-\rho_{m-1}}{\rho_{m}}
	+\sum_{j=3}^{m-1}\II(0;\rho_{m},\dots,\widehat{\rho_j},\dots\rho_{3};\rho_{2})
	\otimes	\frac{\rho_{j}-\rho_{j-1}}{\rho_{j}-\rho_{j+1}}\Big)\otimes
		\frac{\rho_{m+1}-\rho_{m}}{\rho_{m+1}}\Big]\,.\]
By the same argument as above, all terms in the sum $\sum_{j=3}^{m-1}$ vanish under
$\Alt_{m,m}$ (since for $3\le j\le m-1$ both $(j-1,j)$ and $(j+m-2,j+m-1)$
leave $\frac{\rho_{m+1}-\rho_{m}}{\rho_{m+1}}$ invariant).
After repeating this argument $(m-2)$ times we get that the symbol
	of~\eqref{eq:aomotocand} is equal to
	\begin{equation} \label{eq:aomoto_interm1}
	\Alt_{m,m}
	\Big[\frac{\rho_{3}-\rho_{2}}{\rho_{3}}\otimes \dots \otimes
	\frac{\rho_{m}-\rho_{m-1}}{\rho_{m}}\otimes
	\frac{\rho_{m+1}-\rho_{m}}{\rho_{m+1}}\Big]\,.
	\end{equation}

Let us denote
    \begin{equation*}
    \begin{array}{ll}
    A_j = \detv{j+1,\dots,j+m-2,j+m-1,1}\,,&
    C_j = \detv{j+1,\dots,j+m-2,2m,1}\,,\\
    B_j = \detv{j+1,\dots,j+m-2,j+m-1,j}\,,&
    D_j = \detv{j+1,\dots,j+m-2,2m,j}\,.
    \end{array}
    \end{equation*}
Then for $3\le j\le m$ we have the following corollary of the (projected) Pl\"ucker 
relations
    \[\frac{\rho_{j+1}-\rho_{j}}{\rho_{j+1}}
    = -\frac{B_{j}C_{j}}{A_{j}D_{j}}\,. \]
We expand the expression in~\eqref{eq:aomoto_interm1}, and let
$T=X_2\otimes X_3\otimes \dots\otimes X_m$ be any tensor in it with
$X_j\in\{A_j,B_j,C_j,D_j\}$.
By analogy with the proof of Lemma~\ref{lem:ABCDs} if for some $j$ we have
$(X_j,X_{j+1})\in \{(B_j,A_{j+1}),(B_j,C_{j+1}), (C_j,A_{j+1}), (C_j,B_{j+1})
,(D_j,A_{j+1}),(D_j,B_{j+1}),(D_j,C_{j+1})\}$,
then $\Alt_{m,m}T=0$, since $T$ will be fixed by the transposition $(j,j+1)$
or by $(j+m-1,j+m)$. Thus any non-vanishing $T$ must follow one of the patterns 
$A^rB^sD^t$ (${m+1\choose 2}$ such with $r,s,t\geq 0$) or $A^rC^sD^t$, 
and a simple counting argument shows that there are $m^2(={m+1\choose 
2}+{m+1\choose 2}-m)$
such terms. Let
	\[w_{k,l} = (-1)^{m-l+k}A_2\otimes\dots \otimes A_{k-1}\otimes B_k\otimes \dots\otimes
	B_l\otimes
	D_{l+1}\otimes \dots\otimes D_{m}.\]
%%is $\Alt_{m,m}$-equivalent to $B_2\otimes\dots\otimes B_{m}$. 
As in the proof
of Theorem~\ref{thm:i211}, applying the transpositions $(1,k-1)$ (resp. $(l+m,2m)$) to
$w_{k,l}$ gives $w_{k-1,l}$ (resp. $w_{k,l+1}$). Thus all $w_{k,l}$ are
$\Alt_{m,m}$-equivalent to $w_{2,m}=B_2\otimes\dots\otimes B_m$.
Similarly, all terms of the form
	\[(-1)^{m-l+k}A_2\otimes\dots \otimes A_{k-1}\otimes C_k\otimes \dots\otimes
		C_l\otimes D_{l+1}\otimes \dots\otimes D_{m}\]
are also $\Alt_{m,m}$-equivalent to $B_2\otimes\dots\otimes B_m$.
Thus,~\eqref{eq:aomoto_interm1} equals
	\[m^2\Alt_{m,m}[B_{2}\otimes \dots\otimes B_{m}]\,,\]
which by Proposition~\ref{prop:aomoto_symbol} equals $(-1)^{m-1}/m^2$ times the 
symbol of the
Aomoto polylogarithm, as claimed.

\medskip
\section{Proof of Theorem~\protect\ref{thm:gr4toi31_2}}
\label{sec:pf_i31}
Proving Theorem~\ref{thm:gr4toi31_2} amounts to computing $\Sl^\shuffle$
of both sides and checking that they are equal. The easiest way to do this
verification is to use a computer algebra system, and in fact this is the way in
which the identity of Theorem~\ref{thm:gr4toi31_2} was discovered. In this section we
will give a direct proof that does not require any computer calculations.

First, note the following identities
    \begin{align}\label{eq:i31cobracketa}
    \begin{split}
    \Sl^\shuffle\II_{3,1}(x,y) =
    \Sl^\shuffle\II_{2,1}(x,y)\otimes\frac{y}{x}
    - \Sl^\shuffle&(\Lic_{3}(x)-\Lic_{3}(x/y))\otimes(1-y^{-1}) \\
    + \Sl^\shuffle&(\Lic_{3}(y)-\Lic_{3}(y/x))\otimes(1-x^{-1})\ ,
    \end{split}
    \end{align}
    \begin{align}\label{eq:i21cobracketa}
    \begin{split}
    \Sl^\shuffle\II_{2,1}(x,y) =
    \Sl^\shuffle\II_{1,1}(x,y)\otimes\frac{y}{x}
    + \Sl^\shuffle&(\Lic_{2}(x)-\Lic_{2}(x/y))\otimes(1-y^{-1}) \\
    + \Sl^\shuffle&(\Lic_{2}(y)-\Lic_{2}(y/x))\otimes(1-x^{-1})\ .
    \end{split}
    \end{align}
Recall also
that $\Sl^{\shuffle}\Lic_2(x)=x \wedge (1-x)$
and that for $k\ge 3$ we have
$\Sl^{\shuffle}\Lic_k(x)=\Sl^{\shuffle}\Lic_{k-1}(x)\otimes
x$.
In this section we will simply write $\II_{k,1}(x,y)$, $\Lic_m(x)$ instead
of $\Sl^\shuffle\II_{k,1}(x,y)$, $\Sl^\shuffle\Lic_m(x)$.

\smallskip
Throughout the proof we work modulo the kernel of
the skew-symmetrization	operator $\Alt_8$;
in particular, any term that is invariant under an odd permutation
is annihilated by $\Alt_8$. We denote such identities by
``$\modalt$''. In the proof we will repeatedly use the fact that the
permutation $\cyc{1,6}\cyc{2,5}\cyc{3,4}$ permutes $\rho_1,\dots,\rho_4$ as
$\rho_4\leftrightarrow\rho_1$, $\rho_3\leftrightarrow\rho_2$. Note also that
the transposition $\cyc{7,8}$ maps each $\rho_i$ to its inverse $\rho_i^{-1}$.
All tensor products are written with respect to multiplication and we work modulo
torsion, so that $a\otimes bc=a\otimes b+a\otimes c$ and $a\otimes (-b)=a\otimes b$.
In particular, recalling the notation $\rho_{i,j}=\rho_i-\rho_j$, the above implies
that $\rho_{i,j}\otimes x = \rho_{j,i}\otimes x$ which we will use freely.
Finally, we denote by $a\wedge b$ the difference $a\otimes b - b\otimes a$
and by $a\odot b$ the sum $a\otimes b + b\otimes a$.

First, we give a different expression for the mod-products symbol of $\GR_4$.
\begin{lemma} \label{lem:gr4li2}
    We have the following identity
    \begin{equation} \label{eq:gr4li2}
    \frac{7}{144}\Sl^\shuffle(\GR_4) =
    8\Alt_8\Big[\Sl^\shuffle
    \Lic_2\Big(\frac{\rho_{2,4}}{\rho_{3,4}}\Big)
    \otimes\rho_{1,2}\otimes\frac{\rho_1}{\rho_4}\Big]
    \,.
    \end{equation}
\end{lemma}
\begin{proof}
Note that we have
    \begin{equation} \label{eq:r32r34}
    \Alt_8 \Big[\frac{\rho_{3,2}}{\rho_{3,4}}
    \otimes f(\rho_1,\rho_2,\rho_4)\Big] = 0\,,
    \end{equation}
    where $f(\rho_1,\rho_2,\rho_4)$ is any expression that
depends only on $\rho_1,\rho_2,\rho_4$. This follows from the
fact that $\rho_1$, $\rho_2$, and $\rho_4$ are fixed by both
$\cyc{2,3}$ and $\cyc{5,6}$,
while $\frac{\rho_{3,2}}{\rho_{3,4}}=\frac{\detv{2345}\detv{4568}}
{\detv{2348}\detv{4578}}\cdot\frac{\detv{3478}}{\detv{3456}}$
is a product of two terms, one invariant under $\cyc{2,3}$ and the
other
under $\cyc{5,6}$. Using the fact that
$\Sl^\shuffle\Lic_2(x)=x\wedge (1-x)$,
we calculate right-hand side of~\eqref{eq:gr4li2} to be equal to
    \begin{align*}
    8&\Alt_8\Big[
    \frac{\rho_{2,4}}{\rho_{3,4}}\wedge
    \frac{\rho_{3,2}}{\rho_{3,4}}
    \otimes\rho_{1,2}\otimes \frac{\rho_1}{\rho_4}
    \Big] =
    8\Alt_8\Big[
    -\rho_{4,3}\wedge\rho_{3,2}
    \otimes\rho_{2,1}\otimes \frac{\rho_1}{\rho_4}
    \Big] \\
    =8&\Alt_8\Big[
    -\rho_{4,3}\wedge\rho_{3,2}
    \otimes\rho_{2,1}\otimes \rho_1
    -\rho_{1,2}\wedge\rho_{2,3}
    \otimes\rho_{3,4}\otimes \rho_1
    \Big] \\
    =8&\Alt_8\Big[-\Pid_3(\rho_{4,3}\otimes\rho_{3,2}
    \otimes\rho_{2,1})\otimes \rho_1
    \Big] \,,
    \end{align*}
where on the second line we have applied
$\cyc{1,6}\cyc{2,5}\cyc{3,4}$ to one of the terms,
and $\Pid_3$ is the operator that annihilates shuffle
products
(see Section~\ref{sec:modprod}). Since
    \[\Pid_4(a\otimes b\otimes c\otimes d) =
    \Pid_3(a\otimes b \otimes c)\otimes d
    -\Pid_3(d\otimes c \otimes b)\otimes a\,,\]
and the odd permutation $\cyc{1,7}\cyc{2,6}\cyc{3,5}$ maps
$\detv{1234}\otimes \detv{2345} \otimes \detv{3456} \otimes \detv{4567}$
to its reversal, it is enough to prove
    \begin{equation} \label{eq:rho4321}
    \Alt_8\Big[\rho_{4,3}\otimes\rho_{3,2}
    \otimes\rho_{2,1}\otimes \rho_1\Big]
    = -14\Alt_8\Big[\detv{1234}\otimes \detv{2345}
    \otimes \detv{3456} \otimes \detv{4567}\Big]\,.
    \end{equation}
This identity can be seen by either doing the same
combinatorial analysis as in the proof of Theorem~\ref{thm:i211}
or by directly expanding the $128$ terms on the left and noting
that only $14$ of them are not fixed by any transposition
and that the rest is $\Alt_8$-equivalent
to $-\detv{1234}\otimes\detv{2345}\otimes\detv{3456}\otimes\detv{4567}$
(also compare with~\eqref{eq:rho4321b} below).
\end{proof}

Next, we compute the right-hand side of~\eqref{eq:gr4i31_2}.
After applying~\eqref{eq:i31cobracketa} to compute the
$\Sl^{\shuffle}$ of the right hand side of~\eqref{eq:gr4i31_2},
we obtain the skew-symmetrization of
    \begin{align*}
    \Big[&\II_{2,1}\Big(\frac{\rho_{1,2}\rho_{3,4}}
    {\rho_{3,2}\rho_{1,4}},
    \frac{\rho_1}{\rho_{1,4}}\Big)
    +2\II_{2,1}\Big(\frac{\rho_{1,2}}{\rho_1},
    \frac{\rho_{3,2}}{\rho_{3,4}}\Big)
    +6\Lic_{3}(U)\Big]
    \otimes U\\
    +&\Big[-\Lic_{3}\Big(\frac{\rho_{1,2}\rho_{3,4}}
    {\rho_{3,2}\rho_{1,4}}\Big)
    +\Lic_{3}(U)\Big] \otimes \frac{\rho_4}{\rho_1}
    +\Big[\Lic_{3}\Big(\frac{\rho_1}{\rho_{1,4}}\Big)
    -\Lic_{3}(U)\Big] \otimes \frac{\rho_{1,3}\rho_{2,4}}
    {\rho_{1,2}\rho_{3,4}}\\
    +&2\Big[-\Lic_{3}\Big(\frac{\rho_{1,2}}{\rho_1}\Big)
    +\Lic_{3}(U)\Big] \otimes \frac{\rho_{4,2}}{\rho_{3,2}}
    +2\Big[\Lic_{3}\Big(\frac{\rho_{3,2}}{\rho_{3,4}}\Big)
    -\Lic_{3}(U)\Big] \otimes\frac{\rho_2}{\rho_{2,1}}  \,,
    \end{align*}
where as before we abbreviate $\rho_{i,j}=\rho_i-\rho_j$ and
denote $U=\frac{\rho_1\rho_{3,2}}{\rho_{1,2}\rho_{3,4}}$.
Collecting the terms with~$\Lic_{3}(U)$ from the last two lines we get
    \begin{align*}
    &\Lic_{3}(U)
    \otimes
    \Big[\left(\frac{\rho_4}{\rho_1}\right)
    -\left(\frac{\rho_{1,3}\rho_{2,4}}
    {\rho_{1,2}\rho_{3,4}}\right)
    +2\left(\frac{\rho_{4,2}}{\rho_{3,2}}\right)
    -2\left(\frac{\rho_2}{\rho_{2,1}}\right)\Big] \\
    &\modalt \Lic_{3}(U)
    \otimes\Big[-2\left(\frac{\rho_1\rho_{3,2}}{\rho_{1,2}\rho_{3,4}}\right)
    +\left(\frac{\rho_1\rho_4\rho_{1,2}\rho_{2,4}}{\rho_2^2\rho_{1,3}\rho_{3,4}}\right)\Big]
     \modalt-2\Lic_{3}(U) \otimes U
    \,,
    \end{align*}
where we have used the fact that $\cyc{1,6}\cyc{2,5}\cyc{3,4}\cyc{7,8}$
leaves $U$ fixed and sends
the other parenthesized term to its inverse (we have added extra
parentheses to emphasize that we work multiplicatively). This leaves
us with $\Alt_8$ of
    \begin{align*}
    \Big[&\II_{2,1}\Big(\frac{\rho_{1,2}\rho_{3,4}}
    {\rho_{3,2}\rho_{1,4}},
    \frac{\rho_1}{\rho_{1,4}}\Big)
    +2\II_{2,1}\Big(\frac{\rho_{1,2}}{\rho_1},
    \frac{\rho_{3,2}}{\rho_{3,4}}\Big)
    +4\Lic_{3}(U)\Big]
    \otimes U\\
    -&\Lic_{3}\Big(\frac{\rho_{1,2}\rho_{3,4}}
    {\rho_{3,2}\rho_{1,4}}\Big)
    \otimes\frac{\rho_4}{\rho_1}
    +\cancel{\Lic_{3}\Big(\frac{\rho_1}{\rho_{1,4}}\Big)
        \otimes\frac{\rho_{1,3}\rho_{2,4}}{\rho_{1,2}\rho_{3,4}}}
    -2\Lic_{3}\Big(\frac{\rho_{1,2}}{\rho_1}\Big)
    \otimes\frac{\cancel{\rho_{4,2}}}{\rho_{3,2}}
    +2\Lic_{3}\Big(\frac{\rho_{3,2}}{\rho_{3,4}}\Big)
    \otimes\frac{\rho_2}{\rho_{2,1}}  \,,
    \end{align*}
where the crossed-out terms cancel since they expand out to a combination
of terms that are fixed either by $\cyc{1,2}$ or by $\cyc{2,3}$
and hence vanish. Next, we apply~\eqref{eq:i21cobracketa} to the
expression in the square brackets. Rearranging the terms with
$\Lic_2(U)$ as above we get (skew-symmetrization of)
    \begin{align} \label{eq:a1}
    \begin{split}
    \Big[&\II_{1,1}\Big(\frac{\rho_{1,2}\rho_{3,4}}
    {\rho_{3,2}\rho_{1,4}},
    \frac{\rho_1}{\rho_{1,4}}\Big)
    +2\II_{1,1}\Big(\frac{\rho_{1,2}}{\rho_1},
    \frac{\rho_{3,2}}{\rho_{3,4}}\Big)
    +2\Lic_{2}(U)\Big] \otimes U \otimes U \\
    +\Big[&\Lic_{2}\Big(\frac{\rho_{1,2}\rho_{3,4}}
    {\rho_{3,2}\rho_{1,4}}\Big)
    \otimes\frac{\rho_4}{\rho_1}
    +\Lic_{2}\Big(\frac{\rho_1}{\rho_{1,4}}\Big)
    \otimes\frac{\rho_{1,3}\rho_{2,4}}{\rho_{1,2}\rho_{3,4}}
    +2\Lic_{2}\Big(\frac{\rho_{1,2}}{\rho_1}\Big)
    \otimes\frac{\rho_{4,2}}{\rho_{3,2}}
    +2\Lic_{2}\Big(\frac{\rho_{3,2}}{\rho_{3,4}}\Big)
    \otimes\frac{\rho_2}{\rho_{2,1}}\Big] \otimes U \\
    -&\Lic_{3}\Big(\frac{\rho_{1,2}\rho_{3,4}}
    {\rho_{3,2}\rho_{1,4}}\Big)
    \otimes\frac{\rho_4}{\rho_1}
    +2\Lic_{3}\Big(\frac{\rho_{1,2}}{\rho_1}\Big)
    \otimes \rho_{3,2}
    +2\Lic_{3}\Big(\frac{\rho_{3,2}}{\rho_{3,4}}\Big)
    \otimes\frac{\rho_2}{\rho_{2,1}}  \,.
    \end{split}
    \end{align}
We use the following identity that is easy to verify directly
(it holds without any symmetrization)
    \begin{align*}
    \II_{1,1}\Big(\frac{\rho_{1,2}\rho_{3,4}}{\rho_{3,2}\rho_{1,4}},
    \frac{\rho_1}{\rho_{1,4}}\Big)
    +\II_{1,1}\Big(\frac{\rho_{1,2}}{\rho_1},
    \frac{\rho_{3,2}}{\rho_{3,4}}\Big)
    &+\II_{1,1}\Big(\frac{\rho_{3,4}}{\rho_3},
    \frac{\rho_1\rho_{3,2}}{\rho_3\rho_{1,2}}\Big)
    +2\Lic_{2}(U) \\
    &=
    \Lic_{2}\Big(\frac{\rho_1}{\rho_2}\Big) +
    \Lic_{2}\Big(\frac{\rho_2}{\rho_4}\Big) +
    \Lic_{2}\Big(\frac{\rho_{2,3}}{\rho_{1,3}}\Big) +
    \Lic_{2}\Big(\frac{\rho_{2,4}}{\rho_{3,4}}\Big) \,.
    \end{align*}
Since the second and the third $\II_{1,1}$ terms above
get interchanged by $\cyc{1,6}\cyc{2,5}\cyc{3,4}\cyc{7,8}$,
applying this identity to the expression in the first
square brackets of~\eqref{eq:a1} we get that~\eqref{eq:a1}
is equal to
    \begin{align}
    \label{eq:a2}
    \begin{split}
    \Big[&\Lic_{2}\Big(\frac{\rho_1}{\rho_2}\Big)
    + \cancel{\Lic_{2}\Big(\frac{\rho_2}{\rho_4}\Big)}
    + \Lic_{2}\Big(\frac{\rho_{2,3}}{\rho_{1,3}}\Big)
    + \Lic_{2}\Big(\frac{\rho_{2,4}}{\rho_{3,4}}\Big)\Big]
    \otimes U
    \otimes U \\
    +\Big[&\Lic_{2}\Big(\frac{\rho_{1,2}\rho_{3,4}}
    {\rho_{3,2}\rho_{1,4}}\Big)
    \otimes\frac{\rho_4}{\rho_1}
    +\cancel{\Lic_{2}\Big(\frac{\rho_1}{\rho_{1,4}}\Big)
        \otimes\frac{\rho_{1,3}\rho_{2,4}}{\rho_{1,2}\rho_{3,4}}}
    +2\Lic_{2}\Big(\frac{\rho_{1,2}}{\rho_1}\Big)
    \otimes\frac{\cancel{\rho_{4,2}}}{\rho_{3,2}}
    +2\Lic_{2}\Big(\frac{\rho_{3,2}}{\rho_{3,4}}\Big)
    \otimes\frac{\rho_2}{\rho_{2,1}}\Big] \otimes U \\
    -&\Lic_{2}\Big(\frac{\rho_{1,2}\rho_{3,4}}
    {\rho_{3,2}\rho_{1,4}}\Big)
    \otimes \frac{\rho_{1,2}\rho_{3,4}}
    {\rho_{3,2}\rho_{1,4}} \otimes\frac{\rho_4}{\rho_1}
    +2\Lic_{2}\Big(\frac{\rho_{1,2}}{\rho_1}\Big)
    \otimes \frac{\rho_{1,2}}{\rho_1} \otimes \rho_{3,2}
    +2\Lic_{2}\Big(\frac{\rho_{3,2}}{\rho_{3,4}}\Big)
    \otimes \frac{\rho_{3,2}}{\rho_{3,4}} \otimes \frac{\rho_2}{\rho_{2,1}}  \,.
    \end{split}
    \end{align}
We claim that the three crossed-out terms vanish under $\Alt_8$.
Indeed, the first term is
    \begin{align*}
    \Lic_{2}\Big(\frac{\rho_2}{\rho_4}\Big)\otimes U \otimes U
    = \Lic_{2}\Big(\frac{\rho_2}{\rho_4}\Big)\otimes
    \Big[
    \frac{\rho_1}{\rho_{1,2}}\otimes \frac{\rho_1}{\rho_{1,2}}
    +\frac{\rho_1}{\rho_{1,2}}\otimes \frac{\rho_{3,2}}{\rho_{3,4}}
    +\frac{\rho_{3,2}}{\rho_{3,4}}\otimes \frac{\rho_1}{\rho_{1,2}}
    +\frac{\rho_{3,2}}{\rho_{3,4}}\otimes \frac{\rho_{3,2}}{\rho_{3,4}}
    \Big] \,.
    \end{align*}
Here the term $\Lic_{2}(\frac{\rho_2}{\rho_4})\otimes
\frac{\rho_{3,2}}{\rho_{3,4}}
\otimes \frac{\rho_{3,2}}{\rho_{3,4}}$ vanishes after $\Alt_8$ since the
even permutation $\cyc{2,6}\cyc{3,5}$ changes its sign, and the other three
summands vanish by~\eqref{eq:r32r34}.
More generally,~\eqref{eq:r32r34} shows that
$\Alt_8 f(\rho_1,\rho_2,\rho_4)\otimes U = 0$,
and by applying the symmetry $\cyc{1,6}\cyc{2,5}\cyc{3,4}\cyc{7,8}$
we see also that $\Alt_8 f(\rho_1,\rho_3,\rho_4)\otimes U = 0$.
This immediately gives us the vanishing of the third crossed-out
term $\Lic_{2}(\frac{\rho_{1,2}}{\rho_1})\otimes \rho_{4,2}\otimes U$,
and for the second crossed-out term we have
    \begin{align*}
    \Lic_{2}\Big(\frac{\rho_1}{\rho_{1,4}}\Big)
    \otimes \frac{\rho_{1,3}\rho_{2,4}}{\rho_{1,2}\rho_{3,4}}
    \otimes U \modalt
    \Lic_{2}\Big(\frac{\rho_1}{\rho_{1,4}}\Big)
    \otimes \frac{\rho_{2,4}}{\rho_{1,2}}
    \otimes U
    +\Lic_{2}\Big(\frac{\rho_1}{\rho_{1,4}}\Big)
    \otimes \frac{\rho_{1,3}}{\rho_{3,4}}
    \otimes U \modalt 0\,.
    \end{align*}
Next, we apply the five-term relation
    \begin{equation}  \label{eq:fiveterm1}
    \Lic_2\Big(\frac{\rho_{1,2}\rho_{3,4}}{\rho_{3,2}\rho_{1,4}}\Big)
    =\Lic_2\Big(\frac{\rho_{2,3}}{\rho_{1,3}}\Big)
    +\Lic_2\Big(\frac{\rho_{2,4}}{\rho_{3,4}}\Big)
    +\Lic_2\Big(\frac{\rho_{1,4}}{\rho_{1,3}}\Big)
    +\Lic_2\Big(\frac{\rho_{1,4}}{\rho_{2,4}}\Big)
    \end{equation}
to the two $\Lic_2(\frac{\rho_{1,2}\rho_{3,4}}{\rho_{3,2}\rho_{1,4}})$
terms in~\eqref{eq:a2} and reorder the resulting expression by
the $\Lic_2$ terms:
    \begin{align*}
    \Lic_{2}\Big(\frac{\rho_1}{\rho_2}\Big)&\otimes\Big[
    U\otimes U
    -2\rho_{3,2}\otimes U
    + 2\frac{\rho_{1,2}}{\rho_{1}}\otimes \rho_{3,2}
    \Big] \\
    +\Lic_{2}\Big(\frac{\rho_{2,4}}{\rho_{3,4}}\Big)&\otimes\Big[
    U\otimes U
    +\frac{\rho_4}{\rho_1}\otimes U
    -2\frac{\rho_{2}}{\rho_{2,1}}\otimes U
    -\frac{\rho_{1,2}\rho_{3,4}}{\rho_{3,2}\rho_{1,4}}
    \otimes\frac{\rho_4}{\rho_1}
    -2\frac{\rho_{3,2}}{\rho_{3,4}}\otimes \frac{\rho_2}{\rho_{2,1}}
    \Big]\\
    +\Lic_{2}\Big(\frac{\rho_{1,4}}{\rho_{2,4}}\Big)&\otimes\Big[
    \frac{\rho_{4}}{\rho_{1}}\otimes U
    -\frac{\rho_{1,2}\rho_{3,4}}{\rho_{3,2}\rho_{1,4}}
    \otimes\frac{\rho_4}{\rho_1}
    \Big]
    +\Lic_{2}\Big(\frac{\rho_{2,3}}{\rho_{1,3}}\Big)\otimes\Big[
    U\otimes U
    +\frac{\rho_{4}}{\rho_{1}}\otimes U
    -\frac{\rho_{1,2}\rho_{3,4}}{\rho_{3,2}\rho_{1,4}}
    \otimes\frac{\rho_4}{\rho_1}\Big]\\
    +\Lic_{2}\Big(\frac{\rho_{1,4}}{\rho_{1,3}}\Big)&\otimes\Big[
    \frac{\rho_{4}}{\rho_{1}}\otimes U
    -\frac{\rho_{1,2}\rho_{3,4}}{\rho_{3,2}\rho_{1,4}}
    \otimes\frac{\rho_4}{\rho_1}
    \Big]\,.
    \end{align*}
As before, the terms with $\Lic_2(\frac{\rho_{1,4}}{\rho_{1,3}})$
and $\Lic_2(\frac{\rho_{1,4}}{\rho_{2,4}})$ cancel out under
$\Alt_8$ since they expand into a sum of terms fixed by transpositions.
Moreover, the odd permutation $\cyc{1,6}\cyc{2,5}\cyc{3,4}$ maps
$\Lic_{2}(\frac{\rho_{2,3}}{\rho_{1,3}})$ to
$\Lic_{2}(\frac{\rho_{2,4}}{\rho_{3,4}})$ and
$U$ to $\frac{\rho_4}{\rho_1}U$, so we get
    \begin{align*}
    \Lic_{2}\Big(\frac{\rho_1}{\rho_2}\Big)&\otimes\Big[
    U\otimes U
    -2\rho_{3,2}\otimes U
    + 2\frac{\rho_{1,2}}{\rho_{1}}\otimes \rho_{3,2}
    \Big]
    +\Lic_{2}\Big(\frac{\rho_{2,4}}{\rho_{3,4}}\Big)\otimes\Big[
    U\otimes U
    +\frac{\rho_4}{\rho_1}\otimes U\\
    -2\frac{\rho_{2}}{\rho_{2,1}}&\otimes U
    -\frac{\rho_{1,2}\rho_{3,4}}{\rho_{3,2}\rho_{1,4}}
    \otimes\frac{\rho_4}{\rho_1}
    -2\frac{\rho_{3,2}}{\rho_{3,4}}\otimes \frac{\rho_2}{\rho_{2,1}}
    -\frac{\rho_4}{\rho_1}U\otimes \frac{\rho_4}{\rho_1}U
    +\frac{\rho_{4}}{\rho_{1}}\otimes \frac{\rho_4}{\rho_1}U
    -\frac{\rho_{1,2}\rho_{3,4}}{\rho_{3,2}\rho_{1,4}}
    \otimes\frac{\rho_4}{\rho_1}
    \Big]\\
    \modalt
    \Lic_{2}\Big(\frac{\rho_1}{\rho_2}\Big)&\otimes\Big[
    \frac{\rho_{3,2}}{\rho_{3,4}}\otimes \frac{\rho_{3,2}}{\rho_{3,4}}
    -2\rho_{3,2}\otimes \frac{\rho_1\rho_{3,2}}{\rho_{1,2}\rho_{3,4}}
    + 2\frac{\rho_{1,2}}{\rho_{1}}\otimes \rho_{3,2}
    \Big] \\
    +\Lic_{2}\Big(\frac{\rho_{2,4}}{\rho_{3,4}}\Big)&\otimes\Big[
    \frac{\rho_4}{\rho_1}\otimes U
    -\frac{\rho_{1,2}\rho_{3,4}}{\rho_{3,2}\cancel{\rho_{1,4}}}
    \otimes\frac{\rho_4}{\rho_1}
    -2\frac{\rho_{2}}{\rho_{2,1}}\otimes U
    -2\frac{\rho_{3,2}}{\rho_{3,4}}\otimes \frac{\rho_2}{\rho_{2,1}}
    -\frac{\rho_1}{\cancel{\rho_{1,4}}}
    \otimes\frac{\rho_4}{\rho_1}
    \Big]
    \,,
    \end{align*}
where we again have used~\eqref{eq:r32r34}.
Here $\Lic_{2}(\frac{\rho_{2,4}}{\rho_{3,4}})
\otimes \rho_{1,4} \otimes \frac{\rho_4}{\rho_1} \modalt 0$, since
after expanding out $\Lic_2(\frac{\rho_{2,4}}{\rho_{3,4}})$ all
the terms will cancel by transpositions except for the term
    \[\detv{2348}\wedge \detv{3458}
    \otimes \rho_{1,4} \otimes \frac{\rho_4}{\rho_1}\,,\]
which vanishes under $\Alt_8$ since it is fixed by the odd
permutation $\cyc{1,6}\cyc{2,5}\cyc{3,4}$. Thus, in view of
Lemma~\ref{lem:gr4li2} it is enough to prove the following identity
    \begin{align}\label{eq:penul}
    \begin{split}
    &\Lic_{2}\Big(\frac{\rho_{2,4}}{\rho_{3,4}}\Big)
    \otimes\Big[
    \frac{\rho_4}{\rho_1}\otimes \frac{\rho_1\rho_{3,2}}{\rho_{1,2}\rho_{3,4}}
    -\frac{\rho_1\rho_{1,2}\rho_{3,4}}{\rho_{3,2}}
    \otimes\frac{\rho_4}{\rho_1}
    -2\frac{\rho_{2}}{\rho_{2,1}}\otimes
    \frac{\rho_1\rho_{3,2}}{\rho_{1,2}\rho_{3,4}}
    -2\frac{\rho_{3,2}}{\rho_{3,4}}\otimes \frac{\rho_2}{\rho_{2,1}}\\
    &+8\rho_{1,2}\otimes\frac{\rho_4}{\rho_1}
    \Big]
    +\Lic_{2}\Big(\frac{\rho_1}{\rho_2}\Big)\otimes\Big[
    \frac{\rho_{3,2}}{\rho_{3,4}}\otimes \frac{\rho_{3,2}}{\rho_{3,4}}
    -2\rho_{3,2}\otimes \frac{\rho_1\rho_{3,2}}{\rho_{1,2}\rho_{3,4}}
    + 2\frac{\rho_{1,2}}{\rho_{1}}\otimes \rho_{3,2}
    \Big] \modalt  0\,.
    \end{split}
    \end{align}
To prove~\eqref{eq:penul} we decompose it into parts that are
symmetric and skew-symmetric in the last two tensor positions.

\smallskip
\noindent\textbf{Skew-symmetric part.} For the skew-symmetric part
of~\eqref{eq:penul} we need to show that
    \begin{align*}
    \Lic_{2}\Big(\frac{\rho_{2,4}}{\rho_{3,4}}\Big)
    \otimes\Big[
    -\frac{\rho_{2}}{\rho_{2,1}}\wedge
    \frac{\rho_1}{\rho_{1,2}}
    +4\rho_{1,2}\wedge\frac{\rho_4}{\rho_1}
    -\cancel{\rho_1\wedge\rho_4}
    \Big] +\Lic_{2}\Big(\frac{\rho_1}{\rho_2}\Big)
    \otimes (\rho_{3,2}\wedge \rho_{3,4}) \modalt 0\,.
    \end{align*}
Here the term with $\rho_1\wedge\rho_4$ cancels for the same reason
as the term $\Lic_{2}(\frac{\rho_{2,4}}{\rho_{3,4}})
\otimes \rho_{1,4} \otimes \frac{\rho_4}{\rho_1}$ above.
Using~\eqref{eq:r32r34} and the mod-products symbol for $\Lic_2$
we get
    \begin{align*}
    \Alt_8\Big[
    (\rho_{3,2}\wedge \rho_{3,4})\otimes
    \Big[
    -\Lic_{2}\Big(\frac{\rho_1}{\rho_{2}}\Big)
    +4\rho_{1,2}\wedge\frac{\rho_4}{\rho_1}
    \Big]
    +\Lic_{2}\Big(\frac{\rho_1}{\rho_2}\Big)
    \otimes (\rho_{3,2}\wedge \rho_{3,4})
    \Big]\,.
    \end{align*}
Since the original expression that we are computing (i.e.,
$\Sl^\shuffle$ applied to the RHS of~\eqref{eq:gr4i31_2})
lies in the image of the projector $\Pid_4$, and since
symmetrizing $\Pid_4(a\otimes b\otimes c\otimes d)$ in
the last two tensor positions results in $(a\wedge b)\wedge (c\wedge d)$,
we see that the above expression is skew-symmetric under
interchanging the first two and the last two tensors.
Thus we need to show
    \begin{align*}
    (\rho_{3,2}\wedge \rho_{3,4})
    \otimes\Big[
    -2\Lic_{2}\Big(\frac{\rho_1}{\rho_{2}}\Big)
    +4\rho_{1,2}\wedge\frac{\cancel{\rho_4}}{\rho_1}
    \Big] \modalt 0\,.
    \end{align*}
The crossed-out term vanishes by
    \begin{align*}
    (\rho_{3,2}\wedge \rho_{3,4})\otimes
    (\rho_{1,2}\wedge \rho_4)
    &\modalt
    -(\rho_{3,2}\wedge \rho_{2,1})\otimes
    (\rho_{4,3}\wedge \rho_1) \\
    &\modalt 14 (\detv{2345}\wedge\detv{3456}) \otimes
    (\detv{1234}\wedge\detv{4567}) \modalt 0\,,
    \end{align*}
where in the first equality we have applied the
permutation $\cyc{1,6}\cyc{2,5}\cyc{3,4}$, in the second
we have used~\eqref{eq:rho4321}, and in the third we have
used the fact that the term is fixed by an odd permutation
$\cyc{1,7}\cyc{2,6}\cyc{3,5}$. Thus we only need to prove
    \begin{align} \label{eq:claim_skew}
    (\rho_{3,2}\wedge \rho_{3,4})
    \otimes\Big[
    \rho_{1,2}\wedge\rho_1
    +\rho_{1,2}\wedge\rho_2
    -\rho_1\wedge\rho_2
    \Big] \modalt 0\,.
    \end{align}
We claim that
    \begin{align}
    \label{eq:rho4321b}
    \begin{split}
    \Alt_8\Big[\rho_{4,3}\otimes\rho_{3,2}
    \otimes\frac{\rho_{1,2}}{\rho_1}\otimes \rho_2\Big]
    = \Alt_8\Big[&\detv{1234}\otimes \detv{2345}
    \otimes \detv{3456} \otimes \detv{4567} \\
    -13&\detv{1234}\otimes \detv{2345}
    \otimes \detv{4567} \otimes \detv{3456}\Big]\,.
    \end{split}
    \end{align}
Taking the skew-symmetrization of~\eqref{eq:rho4321}
and~\eqref{eq:rho4321b}
in the last two tensor positions we obtain~\eqref{eq:claim_skew}.
We can prove~\eqref{eq:rho4321b} by direct expansion using the notation
from the proof of Theorem~\ref{thm:i211}:
    \[
    \rho_{4,3}
    \otimes\rho_{3,2}
    \otimes\frac{\rho_{1,2}}{\rho_1}
    \otimes \rho_2
    =
    \frac{A_3D_3}{C_3C_2}
    \otimes
    \frac{A_2D_2}{C_2C_1}
    \otimes
    \frac{A_1D_1}{B_1C_1}
    \otimes
    \frac{B_2}{C_1}
    \,.\]
Omitting any terms that are invariant under odd permutations we get
    \begin{align*}
     &A_3\otimes A_2\otimes \tfrac{D_1}{B_1}\otimes B_2
     +D_3\otimes D_2\otimes \tfrac{A_1}{B_1}\otimes B_2
     +C_2\otimes A_2D_2\otimes B_1\otimes B_2
     -C_2\otimes C_1\otimes B_1\otimes B_2\\
     +&A_3\otimes \tfrac{C_2}{A_2}\otimes D_1\otimes C_1
      +D_3\otimes \tfrac{C_2}{D_2}\otimes A_1\otimes C_1
      -C_3\otimes C_2\otimes A_1D_1\otimes C_1
      -C_3\otimes C_2\otimes C_1\otimes B_2\,,
    \end{align*}
where $C_3\otimes C_2\otimes C_1 \otimes B_2$ is equivalent to
$\detv{1234}\otimes\detv{2345}\otimes\detv{3456}\otimes\detv{4567}$,
while the other $13$ terms are equivalent to
$\detv{1234}\otimes\detv{2345}\otimes\detv{4567}\otimes\detv{3456}$.

\smallskip
\noindent\textbf{Symmetric part.} For the symmetric part we need to prove
    \begin{align} \label{eq:symid0}
    \begin{split}
    &\Lic_{2}\Big(\frac{\rho_{2,4}}{\rho_{3,4}}\Big)
    \otimes\Big[
    \frac{\rho_4}{\rho_1}\odot \frac{\rho_{3,2}}{\rho_{3,4}}
    -2\frac{\rho_{3,2}}{\rho_{3,4}}\odot \frac{\rho_2}{\rho_{2,1}}
    -\frac{\rho_{2}}{\rho_{2,1}}\odot \frac{\rho_1}{\rho_{1,2}}
    +3\rho_{1,2}\odot\frac{\rho_4}{\rho_1}
    \Big]\\
    +&\Lic_{2}\Big(\frac{\rho_1}{\rho_2}\Big)\otimes\Big[
    \frac{1}{2}\rho_{3,4}\odot \rho_{3,4}
    -\frac{1}{2}\rho_{3,2}\odot \rho_{3,2}
    -2\frac{\rho_{1}}{\rho_{1,2}}\odot \rho_{3,2}
    \Big] \modalt  0\,,
    \end{split}
    \end{align}
where we denote $a\odot b = a\otimes b + b\otimes a$.
\begin{lemma}
    The following identities hold
    \begin{equation}\label{eq:symid1}
    \Lic_{2}\Big(\frac{\rho_{2,4}}{\rho_{3,4}}\Big)
    \otimes\Big[
    \frac{\rho_4}{\rho_1}\odot \frac{\rho_{3,2}}{\rho_{3,4}}
    -2\frac{\rho_{3,2}}{\rho_{3,4}}\odot \frac{\rho_2}{\rho_{2,1}}
    +3\rho_{1,2}\odot\frac{\rho_4}{\rho_1}
    \Big] \modalt
    2 \Lic_{2}\Big(\frac{\rho_1}{\rho_2}\Big)\otimes\Big[
    \frac{\rho_{1}}{\rho_{1,2}}\odot \rho_{3,2}
    \Big]\,,
    \end{equation}
    and
    \begin{equation}\label{eq:symid2}
    \Lic_{2}\Big(\frac{\rho_{2,4}}{\rho_{3,4}}\Big)
    \otimes\Big[
    -\frac{\rho_{2}}{\rho_{2,1}}\odot \frac{\rho_1}{\rho_{1,2}}
    \Big]
    +\Lic_{2}\Big(\frac{\rho_1}{\rho_2}\Big)\otimes\Big[
    \frac{1}{2}\rho_{3,4}\odot \rho_{3,4}
    -\frac{1}{2}\rho_{3,2}\odot \rho_{3,2}
    \Big] \modalt  0\,.
    \end{equation}
\end{lemma}
\begin{proof}
First, we prove~\eqref{eq:symid1}. By rearranging the terms
in the first set of square brackets we get an equivalent identity
    \begin{equation} \label{eq:symid1a}
    \Lic_{2}\Big(\frac{\rho_{2,4}}{\rho_{3,4}}\Big)
    \otimes\Big[
    \cancel{\frac{\rho_4}{\rho_2}\odot \frac{\rho_{3,2}}{\rho_{3,4}}}
    -\frac{\rho_2\rho_1}{\rho_{2,1}^2}\odot\frac{\rho_{3,2}}{\rho_{3,4}}
    +3\rho_{1,2}\odot\frac{\rho_4}{\rho_1}
    \Big]
    -2\Lic_{2}\Big(\frac{\rho_1}{\rho_2}\Big)\otimes\Big[
    \frac{\rho_{1}}{\rho_{1,2}}\odot \rho_{3,2}
    \Big]    \modalt 0   \,,
    \end{equation}
where the crossed-out term vanishes since it is anti-invariant
under $\cyc{2,6}\cyc{3,5}$. We rewrite
    \begin{equation*}
    -\frac{\rho_2\rho_1}{\rho_{2,1}^2}
    \odot\frac{\rho_{3,2}}{\rho_{3,4}}
    = -\frac{1}{2}
    \frac{\rho_2\rho_1}{\rho_{2,1}^2}
    \odot\frac{\rho_{3,2}^2\rho_4}{\rho_{3,4}^2\rho_2}
    +\frac{1}{2}
    \frac{\rho_2\rho_1}{\rho_{2,1}^2}
    \odot\frac{\rho_{4}}{\rho_{2}}\,,
    \end{equation*}
where the first term on the right is fixed by $\cyc{7,8}$.
Since $\cyc{7,8}$ maps $\Lic_2(\frac{\rho_{2,4}}{\rho_{3,4}})$
to $\Lic_2(\frac{\rho_{3}\rho_{2,4}}{\rho_{2}\rho_{3,4}})$,
we have
    \begin{equation} \label{eq:li2split}
    \Lic_2\Big(\frac{\rho_{2,4}}{\rho_{3,4}}\Big)
    \otimes\Big[\frac{\rho_2\rho_1}{\rho_{2,1}^2}
    \odot\frac{\rho_{3,2}^2\rho_4}{\rho_{3,4}^2\rho_2}\Big]
    \modalt
    \frac{1}{2}\Big[
    \Lic_2\Big(\frac{\rho_{2,4}}{\rho_{3,4}}\Big)
    -\Lic_2\Big(\frac{\rho_{3}\rho_{2,4}}{\rho_{2}\rho_{3,4}}\Big)
    \Big]
    \otimes\Big[\frac{\rho_2\rho_1}{\rho_{2,1}^2}
    \odot\frac{\rho_{3,2}^2\rho_4}{\rho_{3,4}^2\rho_2}\Big]
    \,,
    \end{equation}
and hence we can use the five-term relation
    \begin{equation} \label{eq:fivetermr24r34}
    \Lic_2\Big(\frac{\rho_{4}}{\rho_{3}}\Big)
    +\Lic_2\Big(\frac{\rho_{3}}{\rho_{2}}\Big)
    =\Lic_2\Big(\frac{\rho_{3}\rho_{2,4}}{\rho_{2}\rho_{3,4}}\Big)
    -\Lic_2\Big(\frac{\rho_{2,4}}{\rho_{3,4}}\Big)
    +\Lic_2\Big(\frac{\rho_{4}}{\rho_{2}}\Big)
    \end{equation}
to rewrite the LHS of~\eqref{eq:symid1a} as
    \begin{align} \label{eq:symid1b}
    \begin{split}
    &\frac{1}{4}\Big[
    \Lic_{2}\Big(\frac{\rho_{3}}{\rho_{2}}\Big)
    +\Lic_{2}\Big(\frac{\rho_{4}}{\rho_{3}}\Big)
    -\cancel{\Lic_{2}\Big(\frac{\rho_{4}}{\rho_{2}}\Big)}
    \Big]\otimes
    \Big[
    \frac{\cancel{\rho_2\rho_1}}{\rho_{2,1}^2}
    \odot\frac{\rho_{3,2}^2\rho_{4}}{\rho_{3,4}^2\rho_{2}}
    \Big]\\
    &+\Lic_{2}\Big(\frac{\rho_{2,4}}{\rho_{3,4}}\Big)
    \otimes\Big[
    \frac{1}{2}\frac{\rho_2\rho_1}{\rho_{2,1}^2}
    \odot\frac{\rho_{4}}{\rho_{2}}
    +3\rho_{1,2}\odot\frac{\rho_4}{\rho_1}
    \Big] - 2\Lic_{2}\Big(\frac{\rho_1}{\rho_2}\Big)\otimes\Big[
    \frac{\rho_{1}}{\rho_{1,2}}\odot \rho_{3,2}
    \Big]    \,.
    \end{split}
    \end{align}
Here the first crossed-out term cancels by~\eqref{eq:r32r34}
and the second crossed-out term cancels since it is invariant
under $\cyc{7,8}$. We claim that~\eqref{eq:symid1b} vanishes
as a corollary of the following identities:
    \begin{align}
    &\Lic_{2}\Big(\frac{\rho_{2,4}}{\rho_{3,4}}\Big)
    \otimes\Big[\rho_{2}\odot \rho_4 \Big] \modalt 0\,,
    \label{eq:eqb1}\\
    &\Lic_{2}\Big(\frac{\rho_{4}}{\rho_{3}}\Big)
    \otimes\Big[\rho_{1,2}\odot \rho_{3,4}\Big]
    -\Lic_2\Big(\frac{\rho_{3}}{\rho_{2}}\Big)
    \otimes\Big[\rho_{3,2}\odot\rho_{2,1}\Big]
    +2\Lic_{2}\Big(\frac{\rho_{1}}{\rho_{2}}\Big)
    \otimes\Big[\rho_{3,2}\odot\rho_{2,1}\Big]
    \modalt 0\,,
    \label{eq:eqb2}\\
    &\Lic_{2}\Big(\frac{\rho_{4}}{\rho_{3}}\Big)
    \otimes\Big[\rho_{1,2}\odot \rho_{2}\Big]
    +\Lic_2\Big(\frac{\rho_{3}}{\rho_{2}}\Big)
    \otimes\Big[\rho_{1,2}\odot \rho_{2}\Big]
    +\Lic_{2}\Big(\frac{\rho_{2,4}}{\rho_{3,4}}\Big)
    \otimes\Big[
    \rho_2\odot \rho_1\rho_2
    \Big] \modalt 0\,,
    \label{eq:eqb5}\\
    &\Lic_{2}\Big(\frac{\rho_{4}}{\rho_{3}}\Big)
    \otimes\Big[
    \rho_{1,2}\odot \rho_{4}
    \Big]
    +\Lic_2\Big(\frac{\rho_{3}}{\rho_{2}}\Big)
    \otimes\Big[
    \rho_{1,2}\odot \rho_{4}
    \Big]
    +\Lic_{2}\Big(\frac{\rho_{2,4}}{\rho_{3,4}}\Big)
    \otimes\Big[
    \rho_4\odot \rho_1\rho_2
    \Big] \modalt 0\,,
    \label{eq:eqb3}\\
    &\Lic_{2}\Big(\frac{\rho_{4}}{\rho_{3}}\Big)
    \otimes\Big[
    \rho_{1,2}\odot\rho_{2,3}
    \Big]
    +2\Lic_{2}\Big(\frac{\rho_{2,4}}{\rho_{3,4}}\Big)
    \otimes\Big[
    \rho_{1,2}\odot\rho_1
    \Big] \modalt 0\,,
    \label{eq:eqb6}\\
    &\Lic_{2}\Big(\frac{\rho_{2,4}}{\rho_{3,4}}\Big)
    \otimes\Big[
    \rho_{2}\odot\frac{\rho_{1,2}}{\rho_1\rho_2}
    -\rho_{1,2}\odot\rho_1
    \Big] \modalt 0\,,
    \label{eq:eqb7}\\
    &\Lic_2\Big(\frac{\rho_{3}}{\rho_{2}}\Big)
    \otimes\Big[
    \rho_{1,2}\odot\rho_{3,4}
    \Big]
    -2\Lic_{2}\Big(\frac{\rho_{1}}{\rho_{2}}\Big)
    \otimes\Big[\rho_{3,2}\odot\rho_1
    \Big]
    +\Lic_{2}\Big(\frac{\rho_{2,4}}{\rho_{3,4}}\Big)
    \otimes\Big[
    \rho_{1}\rho_{1,2}^2\odot \rho_4
    \Big] \modalt 0\,.
    \label{eq:eqb4}
    \end{align}
More precisely, one gets~\eqref{eq:symid1b} by summing up the above $7$
identities with coefficients given by $(1,1,1/2,-1/2,-1,1,1)$. We prove these
identities as follows. Relation~\eqref{eq:eqb1} follows from
anti-invariance under $\cyc{2,6}\cyc{3,5}$. Equation~\eqref{eq:eqb2}
after rearranging terms (using cyclic shifts
and $\cyc{1,6}\cyc{2,5}\cyc{3,4}$) becomes
$\Lic_2(\frac{\rho_1}{\rho_2})\otimes\frac{\rho_{3,2}}{\rho_{3,4}}\odot
\rho_{2,1}\modalt 0$, which is true by~\eqref{eq:r32r34}.
Equation~\eqref{eq:eqb5}, after applying $\cyc{7,8}$ and splitting
$\Lic_{2}(\frac{\rho_{2,4}}{\rho_{3,4}})$ analogously
to~\eqref{eq:li2split} becomes
    \[\frac{1}{2}\Big[\Lic_{2}\Big(\frac{\rho_{4}}{\rho_{3}}\Big)
    +\Lic_{2}\Big(\frac{\rho_{3}}{\rho_{2}}\Big)+
    \Lic_{2}\Big(\frac{\rho_{2,4}}{\rho_{3,4}}\Big)
    -\Lic_{2}\Big(\frac{\rho_{3}\rho_{2,4}}{\rho_{2}\rho_{3,4}}\Big)\Big]
    \otimes \Big[\rho_{2}\odot\rho_{1}\rho_{2}\Big]
    \modalt 0\,, \]
which clearly follows from~\eqref{eq:fivetermr24r34} and~\eqref{eq:r32r34}.
Identity~\eqref{eq:eqb3} is proved completely analogously to~\eqref{eq:eqb5}.
Equation~\eqref{eq:eqb6} follows from~\eqref{eq:rho4321}
and \eqref{eq:rho4321b} after using the symmetry $\cyc{1,6}\cyc{2,5}\cyc{3,4}$.
Identity~\eqref{eq:eqb7} also follows
from~\eqref{eq:rho4321},~\eqref{eq:rho4321b} and the following identity
    \begin{align} \label{eq:rho4321c}
    \begin{split}
    \Alt_8\Big[\rho_{4,3}\otimes\rho_{3,2}
    \otimes\rho_2\otimes\rho_2\Big]
    = \Alt_8\Big[&\detv{1234}\otimes \detv{2345}
    \otimes \detv{3456} \otimes \detv{4567} \\
    +&\detv{1234}\otimes \detv{2345}
    \otimes \detv{4567} \otimes \detv{3456}\Big]\,,
    \end{split}
    \end{align}
that is easily proved by expanding in $\detv{ijkl}$. Finally,
for equation~\eqref{eq:eqb4} we use cyclic shift in the second
term to rewrite the first two terms of the LHS as
    \begin{align*}
    &\Lic_2\Big(\frac{\rho_{3}}{\rho_{2}}\Big)
    \otimes\Big[\rho_{1,2}\odot\rho_{3,4}\Big]
    -2\Lic_{2}\Big(\frac{\rho_{1}}{\rho_{2}}\Big)
    \otimes\Big[\rho_{3,2}\odot\rho_1
    \Big] \modalt \Lic_2\Big(\frac{\rho_{3}}{\rho_{2}}\Big)
    \otimes\Big[
    \frac{\rho_{1,2}}{\rho_2^2}\odot\rho_{3,4}
    \Big] \\
    \modalt &\Lic_2\Big(\frac{\rho_{3}}{\rho_{2}}\Big)
    \otimes\Big[\frac{\rho_{1,2}}{\rho_2}\odot\frac{\rho_{3,4}}{\rho_3}
    -\cancel{\rho_2\odot\rho_3}\Big]
    \modalt \Lic_2\Big(\frac{\rho_{3}}{\rho_{2}}\Big)
    \otimes\Big[\frac{\rho_{1,2}}{\rho_1}\odot\frac{\rho_{3,4}}{\rho_4}\Big]
     \,,
    \end{align*}
where we used the permutation~$\cyc{1,6}\cyc{2,5}\cyc{3,4}$
and~\eqref{eq:r32r34} in the second equality, and~$\cyc{7,8}$ in the
third. Using the five-term relation~\eqref{eq:fivetermr24r34} we rewrite
    \begin{equation*}
    \Lic_{2}\Big(\frac{\rho_{2,4}}{\rho_{3,4}}\Big)
    \otimes\Big[\rho_{1}\odot \rho_4\Big]
    \modalt
    -\frac{1}{2}\Lic_{2}\Big(\frac{\rho_{3}}{\rho_{2}}\Big)
    \otimes\Big[\rho_{1}\odot \rho_4\Big]\,,
    \end{equation*}
and thus, in view of~\eqref{eq:rho4321}, equation~\eqref{eq:eqb4}
follows from the following identity
    \begin{equation}\label{eq:rho4321d}
    \Lic_2\Big(\frac{\rho_{3}}{\rho_{2}}\Big)
    \otimes\Big[\frac{\rho_{1,2}}{\rho_1}\otimes\frac{\rho_{3,4}}{\rho_4}
    -\frac{1}{2}\rho_1\otimes\rho_4\Big]
    \modalt 28\,(\detv{3456}\wedge \detv{2345})\otimes
    \detv{1234} \otimes \detv{4567}\,,
    \end{equation}
that we verify by direct expansion with the following simplifications.
Since $\frac{\rho_3}{\rho_2}=\CR(34|2578)^{-1}$,
we can expand $\Lic_2(\frac{\rho_3}{\rho_2})$ as
    \[\Lic_2\Big(\frac{\rho_{3}}{\rho_{2}}\Big) =
    \frac{1}{2}\Alt_{\mathfrak{S}_{\{2,5,7,8\}}}\detv{3425}\wedge
    \detv{3427}\]
Since the above expression is fixed by $\cyc{3,4}$ and $\cyc{1,6}$,
we expand $[\frac{\rho_{1,2}}{\rho_1}\otimes\frac{\rho_{3,4}}{\rho_4}-\frac{1}{2}\rho_1\otimes\rho_4]$ as
    \begin{align*}
    &\detv{2348}\otimes\detv{4567}
    -\detv{1237}\otimes\detv{3456}
    +\detv{1237}\otimes\detv{3458}
    -\detv{1237}\otimes\detv{4578}\\
    +&\detv{1234}\otimes\detv{4578}
    -\detv{1234}\otimes\detv{4567}
    -\detv{2378}\otimes\detv{4567}
    +\detv{2378}\otimes\detv{3456} \\
    +\tfrac{1}{2}&(\detv{1237}\otimes\detv{4567}
      -\detv{1238}\otimes\detv{4568})
    +\tfrac{1}{2}\detv{1237}\otimes\detv{4568}
    +\tfrac{1}{2}\detv{1238}\otimes\detv{4567}\,.
    \end{align*}
Here we did not include the terms that are invariant under $\cyc{3,4}$
or $\cyc{1,6}$. Moreover, the three terms on the last line
vanish under skew-symmetrization after multiplying
by~$\Lic_{2}(\frac{\rho_3}{\rho_2})$. After this we simply expand
the remaining expression and collect the terms modulo $\Alt_8$. This
proves~\eqref{eq:rho4321d} and thus~\eqref{eq:eqb4}
and~\eqref{eq:symid1}.

Next, we prove~\eqref{eq:symid2}, the second claim of the lemma.
By applying the permutation $\cyc{1,6}\cyc{2,5}\cyc{3,4}$ and cyclic
shifts to the last two terms in~\eqref{eq:symid2} we see
that~\eqref{eq:symid2} equals (mod~$\Alt_8$)
    \begin{equation} \label{eq:symid2a}
    \Lic_{2}\Big(\frac{\rho_{2,4}}{\rho_{3,4}}\Big)
    \otimes\Big[
    -\frac{\rho_{2}}{\rho_{2,1}}\odot \frac{\rho_{1}}{\rho_{1,2}}
    \Big]
    -\Lic_{2}\Big(\frac{\rho_4}{\rho_3}\Big)
    \otimes\Big[\frac{1}{2}\rho_{1,2}\odot \rho_{1,2}\Big]
    -\Lic_{2}\Big(\frac{\rho_3}{\rho_2}\Big)
    \otimes\Big[\frac{1}{2}\rho_{1,2}\odot \rho_{1,2}
    \Big]\,.
    \end{equation}
Using the five-term relation~\eqref{eq:fivetermr24r34} and
noting that $f(\rho_1,\rho_2,\rho_4)\modalt 0$ and that
$\Lic_2(\frac{\rho_3\rho_{2,4}}{\rho_2\rho_{3,4}})$
is equivalent to $\Lic_2(\frac{\rho_{2,4}}{\rho_{3,4}})$
under $\cyc{7,8}$ we get that~\eqref{eq:symid2a} is equal to
    \begin{equation*}
    \Lic_{2}\Big(\frac{\rho_{2,4}}{\rho_{3,4}}\Big)
    \otimes\Big[
    -\rho_{2}\odot \frac{\rho_{1}}{\rho_{1,2}}
    +\rho_{2,1}\odot \rho_{1}
    -\frac{1}{2}\rho_{1,2}\odot \rho_{1,2}
    +\frac{1}{2}\frac{\rho_{1,2}}{\rho_1\rho_2}\odot
    \frac{\rho_{1,2}}{\rho_1\rho_2}
    \Big] =
    \Lic_{2}\Big(\frac{\rho_{2,4}}{\rho_{3,4}}\Big)
    \otimes\Big[\rho_{1}\otimes\rho_{1}+\rho_{2}\otimes\rho_{2}\Big]\,.
    \end{equation*}
Again using~\eqref{eq:r32r34} we see that it is enough to check that
    \[
    (\rho_{4,3}\wedge\rho_{3,2})\otimes\Big[
    \rho_{1}\otimes\rho_{1}
    +\rho_{2}\otimes\rho_{2}\Big]
    \modalt 0\,.
    \]
This, in turn, follows from
    \[(\rho_{4,3}\otimes\rho_{3,2}-\rho_{3,2}\otimes\rho_{2,1})
    \otimes \rho_{1}\otimes\rho_{1} \modalt 0\,, \]
which we now check by doing expansion. In what follows we
cancel terms by transpositions:
    \begin{align*}
    \Big[\frac{\detv{3456}\detv{4578}}{\detv{3458}\detv{4568}}
    \otimes\frac{\detv{2345}\cancel{\detv{3478}}}
    {\detv{2348}\cancel{\detv{3458}}}
    -      \frac{\cancel{\detv{2345}}\detv{3478}}
    {\cancel{\detv{2348}}\detv{3458}}
    \otimes\frac{\detv{1234}\detv{2378}}{\detv{1238}\detv{2348}}
    \Big]
    \otimes\Big(\frac{\detv{1237}}{\detv{1238}}\Big)^{\otimes 2}
    \end{align*}
by $\cyc{1,2}$ and $\cyc{2,3}$, then
    \begin{align*}
    \Big[\frac{\detv{3456}\cancel{\detv{4578}}}
    {\detv{3458}\cancel{\detv{4568}}}
    \otimes\frac{\detv{2345}}{\detv{2348}}
    -\frac{\detv{3478}}{\detv{3458}}
    \otimes\frac{\cancel{\detv{1234}}\detv{2378}}
    {\cancel{\detv{1238}}\detv{2348}}
    \Big]
    \otimes\Big(\frac{\detv{1237}}{\detv{1238}}\Big)^{\otimes 2}
    \end{align*}
by $\cyc{2,3}$ and $\cyc{1,2}$, then
    \begin{align*}
    \Big[\frac{\detv{3456}}{\detv{3458}}
    \otimes\frac{\cancel{\detv{2345}}}{\detv{2348}}
    -\frac{\cancel{\detv{3478}}}{\detv{3458}}
    \otimes\frac{\detv{2378}}{\detv{2348}}
    \Big]
    \otimes\Big(\frac{\detv{1237}}{\detv{1238}}\Big)^{\otimes 2}
    \end{align*}
by $\cyc{4,5}$ and $\cyc{5,6}$, and finally
(note that the two terms $\detv{3458}\otimes\detv{2348}$ cancel out)
    \begin{align*}
    \Big[-\cancel{\detv{3456}}\otimes\detv{2348}
    +\cancel{\detv{3458}}\otimes\detv{2378}\Big]
    \otimes\Big(\frac{\detv{1237}}{\detv{1238}}\Big)^{\otimes 2}
    \end{align*}
by $\cyc{5,6}$ and $\cyc{4,5}$.
\end{proof}
Combining~\eqref{eq:symid1} and~\eqref{eq:symid2} gives~\eqref{eq:symid0},
and hence concludes the proof of Theorem~\ref{thm:gr4toi31_2}.

\medskip
\section{Proof of Theorem~\protect\ref{thm:gr4toli4}}
\label{sec:pf_gr4toli4}
We will use the identity from Theorem~\ref{thm:gr4toi31_2}:
\begin{equation*}
\begin{split}
\frac{7}{144}\,\GR_4 \;\modsh\; \Alt_8 \Big[
\II_{3,1}\Big(\frac{\rho_{1,2}\rho_{3,4}}
{\rho_{3,2}\rho_{1,4}},
\frac{\rho_1}{\rho_{1,4}}\Big)
+2&\II_{3,1}\Big(\frac{\rho_{1,2}}
{\rho_1},
\frac{\rho_{3,2}}{\rho_{3,4}}\Big)
+6\Lic_{4}\Big(\frac{\rho_1\rho_{3,2}}
{\rho_{1,2}\rho_{3,4}}\Big)\Big]\,.
\end{split}
\end{equation*}
Observe that the $\Lic_4$ term in the above expression corresponds
exactly to the single $\Lic_4$ term in~\eqref{eq:gr4i31_2}, and
also that the two $\symsix$ terms correspond to passing
from $\II_{3,1}$ to $\wI_{3,1}$ using Proposition~\ref{prop:i31tilde}.
Thus to prove the theorem it is enough
to establish the following two lemmas.
\begin{lemma}\label{lem:gr4li4_1}
    We have
    \begin{align*}
    &\Alt_{8}\Big[\wI_{3,1}\Big(\frac{\rho_{1,2}\rho_{3,4}}
    {\rho_{3,2}\rho_{1,4}},
    \frac{\rho_1}{\rho_{1,4}}\Big)\Big]\\
    = &\Alt_{8}\fiveterm\big(\tfrac{\rho_4}{\rho_1};\,
    -[\tfrac{\rho _{4,2}}{\rho_{4,1}};\tfrac{\rho_{4,1}}{\rho_{4,3}}]
    +             [34|2685;48|7653]
    -\tfrac{1}{4} [43|1256;43|1268]
    +\tfrac{1}{12}[43|1256;42|1365]\big)\,.
    \end{align*}
\end{lemma}
\begin{lemma} \label{lem:gr4li4_2}
    We have
    \begin{align*}
    &\Alt_{8}\Big[\wI_{3,1}\Big(\frac{\rho_{1,2}}{\rho_1},
    \frac{\rho_{3,2}}{\rho_{3,4}}\Big)
    +\wI_{3,1}(\CR(34| 2567), \CR(67| 1345))\Big] \\
    = &\Alt_{8}\fiveterm(\tfrac{\rho_2}{\rho_1};\,
    -\tfrac{1}{2} [34|2685;48|7653]
    +\tfrac{1}{2}[48|7235;48|7263]
    +\tfrac{1}{4}[46|5238;43|2568])\,.
    \end{align*}
\end{lemma}
\begin{proof}[Proof of Lemma~\ref{lem:gr4li4_1}]
    First, using a five-term relation equivalent to~\eqref{eq:fiveterm1} we get
    \begin{equation} \label{eq:gr4li4_1a}
    \begin{split}
    &\quad\quad\wI_{3,1}\Big(\frac{\rho_{1,2}\rho_{3,4}}
    {\rho_{3,2}\rho_{1,4}},
    \frac{\rho_1}{\rho_{1,4}}\Big)
    = \wI_{3,1}\Big(\frac{\rho_1}{\rho_4},
    \frac{\rho_{1,2}\rho_{3,4}}
    {\rho_{3,2}\rho_{1,4}}\Big)\\
    &=
    \fiveterm\Big(\frac{\rho_1}{\rho_4};\,
    \frac{\rho_{4,2}}{\rho_{4,1}},
    \frac{\rho_{4,1}}{\rho_{4,3}}\Big)
    - \wI_{3,1}\Big(\frac{\rho_1}{\rho_4},
    \frac{\rho_{3,2}}{\rho_{3,4}}\Big)
    - \wI_{3,1}\Big(\frac{\rho_1}{\rho_4},
    \frac{\rho_{3,1}}{\rho_{3,2}}\Big)
    - \wI_{3,1}\Big(\frac{\rho_1}{\rho_4},
    \frac{\rho_{4,2}}{\rho_{4,1}}\Big)
    - \wI_{3,1}\Big(\frac{\rho_1}{\rho_4},
    \frac{\rho_{4,1}}{\rho_{4,3}}\Big)\\
    &\modalt
    -\fiveterm\Big(\frac{\rho_4}{\rho_1};\,
    \frac{\rho_{4,2}}{\rho_{4,1}},
    \frac{\rho_{4,1}}{\rho_{4,3}}\Big)
    -2\wI_{3,1}\Big(\frac{\rho_1}{\rho_4},
    \frac{\rho_{3,2}}{\rho_{3,4}}\Big)\,.
    \end{split}
    \end{equation}
    Here the last two terms on the second line vanish under $\Alt_8$, since
    they are fixed by $\cyc{2,3}$ and $\cyc{1,2}$ respectively, and using
    the involution $\cyc{1,6}\cyc{2,5}\cyc{3,4}$ and the $6$-fold
    symmetry of $\wI_{3,1}$ we see that the term
    $\wI_{3,1}(\frac{\rho_1}{\rho_4},\frac{\rho_{3,1}}{\rho_{3,2}})$
    is $\Alt_8$-equivalent to
    $\wI_{3,1}(\frac{\rho_1}{\rho_4},\frac{\rho_{3,2}}{\rho_{3,4}})$.
    Note that
    \begin{equation} \label{eq:gonchtocr}
    \frac{\rho_{3,4}}{\rho_{3,2}} =
    \frac{\CR(34|2685)}{\CR(48|7635)}\,.
    \end{equation}
    Thus, if we denote $r_1=\CR(34|2685)$, $r_2=\CR(48|7635)$, then
    \begin{equation}\label{eq:lemtetra2a}
    \begin{split}
    \fiveterm\Big(\frac{\rho_4}{\rho_1};\,r_1,r_2^{-1}\Big)
    &=
    \wI_{3,1}\Big(\frac{\rho_1}{\rho_4},-[r_1]+[r_2]-[(1-\sigma_{(56)}(r_0))^{-1}]-[1-r_0]
    -[(1-r_3)^{-1}]\Big)\\
    &=
    \wI_{3,1}\Big(\frac{\rho_1}{\rho_4},
    [r_0]-[\sigma_{(56)}(r_0)]-[r_1]+[r_2]-[r_3]\Big)\,,
    \end{split}
    \end{equation}
    where we denote $r_0 = \frac{\rho_{3,4}}{\rho_{3,2}}$,
    $r_3=\frac{\detv{2346}\detv{4578}}{\detv{2345}\detv{4678}}$, and as
    before~$\sigma_{\pi}$ denotes the action of $\pi\in\mathfrak{S}_8$.
    Since $\sigma_{(23)}(r_3)=r_3$ and $\sigma_{(12)}(r_2)=r_2$
    (and these two involutions fix $\rho_1/\rho_4$)
    the last two terms in~\eqref{eq:lemtetra2a}
    cancel out after skew-symmetrization, and hence we obtain
    \begin{equation}\label{eq:gr4li4_1b}
    -2\,\wI_{3,1}\Big(\frac{\rho_1}{\rho_4},\frac{\rho_{3,2}}{\rho_{3,4}}\Big)
    \modalt \fiveterm\Big(\frac{\rho_4}{\rho_1};\,r_1,r_2^{-1}\Big)
    -\wI_{3,1}\Big(\frac{\rho_4}{\rho_1},r_1\Big)\,.
    \end{equation}
    For $\wI_{3,1}(\frac{\rho_4}{\rho_1},\CR(34| 2685))$ we use
    the following five-term identities:
    \begin{align*}
    \fiveterm\big(\tfrac{\rho_4}{\rho_1};
    \,&\CR(43| 1256),\CR(42| 1365)\big)\\
    &= \wI_{3,1}\big(\tfrac{\rho_4}{\rho_1},
    [43|1256] + [42|1365] + [45|1326] + [41|3526] + [46|1235]\big)\\
    &\modalt \wI_{3,1}\big(\tfrac{\rho_4}{\rho_1},
    3[43|1256]\big)\,,\\
    \fiveterm\big(\tfrac{\rho_4}{\rho_1};
    \,&\CR(43| 1256),\CR(43| 1268)\big)\\
    &= \wI_{3,1}\big(\tfrac{\rho_4}{\rho_1},
    [43|1256] + [43|1268] + [43|1586] + [43|1528] + [43|2856]\big)\\
    &\modalt \wI_{3,1}\big(\tfrac{\rho_4}{\rho_1},
    [43|1256] + 2[43|1268] + 2[43|2685]\big)
    \modalt \wI_{3,1}\big(\tfrac{\rho_4}{\rho_1},
    [43|1256] + 4[34|2685]\big)\,,
    \end{align*}
    where we have used the fact that $\frac{\rho_4}{\rho_1}$ is fixed by
    arbitrary permutations of $\{1,2,3\}$ and $\{4,5,6\}$, and on the
    last line we have used the fact that it is mapped to its inverse
    under the involution $\cyc{1,6}\cyc{2,5}\cyc{3,4}$. From these we obtain
    \begin{equation*}
    \wI_{3,1}\big(\tfrac{\rho_4}{\rho_1},r_1\big)
    \modalt \fiveterm\big(\tfrac{\rho_4}{\rho_1};\,
    \tfrac{1}{4} [43|1256;43|1268]
    -\tfrac{1}{12}[43|1256;42|1365]
    \big)\,,
    \end{equation*}
    which together with~\eqref{eq:gr4li4_1a} and~\eqref{eq:gr4li4_1b}
    proves the claim.
\end{proof}

\begin{proof}[Proof of Lemma~\ref{lem:gr4li4_2}]
    We start by rewriting
    $\wI_{3,1}(\frac{\rho_{1,2}}{\rho_1},\frac{\rho_{3,2}}{\rho_{3,4}}) = \wI_{3,1}(\frac{\rho_{2}}{\rho_1},\frac{\rho_{3,4}}{\rho_{3,2}})$.
    Using~\eqref{eq:gonchtocr} and the same five-term
    relation as in~\eqref{eq:lemtetra2a} we get
    \begin{equation*}
    \begin{split}
    \fiveterm\Big(\frac{\rho_1}{\rho_2};\,r_1,r_2^{-1}\Big)
    =\wI_{3,1}\Big(\frac{\rho_2}{\rho_1},
    [r_0]-[\sigma_{(56)}(r_0)]-[r_1]+[r_2]-[r_3]\Big)\,,
    \end{split}
    \end{equation*}
    where as before $r_0 = \frac{\rho_{3,4}}{\rho_{3,2}}$,
    $r_1=\CR(34|2685)$, $r_2=\CR(48|7635)$, and
    $r_3=\frac{\detv{2346}\detv{4578}}{\detv{2345}\detv{4678}}$.
    Since $\sigma_{(23)}(r_3)=r_3$ and $\sigma_{(23)}(\rho_i)=\rho_i$ for
    $i=1,2$, the term with $r_3$ vanishes after skew-symmetrization, and we obtain
    \begin{equation} \label{eq:gr4li4_2a}
    \wI_{3,1}\Big(\frac{\rho_{2}}{\rho_{1}},
    \frac{\rho_{3,2}}{\rho_{3,4}}\Big)
    \modalt
    \frac{1}{2}\fiveterm\Big(\frac{\rho_1}{\rho_2};\,r_1,r_2^{-1}\Big)
    +\frac{1}{2}\wI_{3,1}\Big(\frac{\rho_{2}}{\rho_{1}},\CR(34|2685)\Big)
    -\frac{1}{2}\wI_{3,1}\Big(\frac{\rho_{2}}{\rho_{1}},\CR(48|7635)\Big)\,.
    \end{equation}
    Since $\frac{\rho_2}{\rho_1}=\CR(23| 1487)$, we can rewrite
    the remaining combination of $\wI_{3,1}$'s (modulo $\Alt_8$)
    together with the term~$\wI_{3,1}(\CR(34| 2567), \CR(67| 1345))$ as
    \begin{equation*}
    \wI_{3,1}\big(\tfrac{\rho_{2}}{\rho_{1}},
    \tfrac{1}{2}[34|2685]-\tfrac{1}{2}[48|7635]+[48|7235]\big)\,.
    \end{equation*}
    To express it in terms of $\fiveterm$ we will use two five-term relations.
    The first is
    \begin{align*}
    \fiveterm\big(\tfrac{\rho_2}{\rho_1};
    \,&\CR(48| 7235),\CR(48| 7263)\big)\\
    &= \wI_{3,1}\big(\tfrac{\rho_2}{\rho_1},
    [48|7235] + [48|7263] + [48|2563] + [48|2576] + [48|3567]\big)\\
    &\modalt \wI_{3,1}\big(\tfrac{\rho_2}{\rho_1},
    2[48|7235] - 2[48|7635] + [48|2356]\big)
    \modalt \wI_{3,1}\big(\tfrac{\rho_2}{\rho_1},
    2[48|7235] - 2[48|7635] \big)\,,
    \end{align*}
    where we have used the fact that the permutations~$\cyc{2,3}$
    and~$\cyc{5,6}$ both leave $\frac{\rho_2}{\rho_1}$ fixed, and the
    term $\wI_{3,1}([23|1487],[48|2356])$ vanishes after skew-symmetrization,
    since the even permutation $\cyc{1,5}\cyc{2,4}\cyc{3,8}\cyc{6,7}$
    maps it to $\wI_{3,1}([48|5236],[23|4817])$ that is equal
    to $-\wI_{3,1}([23|1487],[48|2356])$ by Proposition~\ref{prop:i31tilde}.
    For the second five-term relation we use the symmetry
    $\wI_{3,1}(x,y)=-\wI_{3,1}(y,x)$ together with $36$-fold symmetry
    for $\wI_{3,1}$ to rewrite
    \begin{equation} \label{eq:gr4li4_2b}
    \wI_{3,1}([23|1487],[48|7635]) =
    -\wI_{3,1}([48|7635],[23|1487]) \modalt
    \wI_{3,1}([23|1487],[46|5238])\,,
    \end{equation}
    where in the second equality we have used
    the cyclic permutation~$\cyc{1,5,7,8,3,4,2,6}$.
    Then
    \begin{align*}
    \fiveterm\big(\tfrac{\rho_2}{\rho_1};
    \,&\CR(46| 5238),\CR(43| 2568)\big)\\
    &= \wI_{3,1}\big(\tfrac{\rho_2}{\rho_1},
    [46|5238] + [43|2568] + [42|3586] + [48|5326] + [45|2638]\big)\\
    &\modalt \wI_{3,1}\big(\tfrac{\rho_2}{\rho_1},
    2[46|5238] + 2[34|2685] - [48|2356]\big)
    \modalt \wI_{3,1}\big(\tfrac{\rho_2}{\rho_1},
    2[46|5238] + 2[34|2685] \big)\,,
    \end{align*}
    where we have used again that~$\wI_{3,1}([23|1487],[48|2356])$
    vanishes under skew-symmetrization.
    Combining the two five-term relations together
    with~\eqref{eq:gr4li4_2b} we get
    \begin{equation*}
    \wI_{3,1}\big(\tfrac{\rho_{2}}{\rho_{1}},
    \tfrac{1}{2}[34|2685]-\tfrac{1}{2}[48|7635]+[48|7235]\big)
    \modalt \fiveterm\big(\tfrac{\rho_2}{\rho_1};
    \tfrac{1}{2}[48|7235;48|7263]
    +\tfrac{1}{4}[46|5238;43|2568])\,,
    \end{equation*}
    which together with~\eqref{eq:gr4li4_2a} proves the claim.
\end{proof}

\medskip
\section{Proof of Proposition~\protect\ref{prop:gr5cobdecomposition}}
\label{sec:proofpropgr5}

The decomposition of~\eqref{eq:gr5del} in Equation~\eqref{eq:propgr5}
into the \( \II_{4,1}^+ \) subsums is a direct rewriting of the \( \II_{4,1} \) expression.  It is clear that the last subsum is a trivial coboundary piece: every summand depends only on $9$ of the $10$ points.  The claim will follow from the following lemmas which express the first three subsums as combinations of \( \Lic_2 \) and \( \Lic_3 \) functional equations.

We will write
\[
V(x,y) = [x]+[y]+\Big[\frac{1-x}{1-xy}\Big]+[1-xy]
+\Big[\frac{1-y}{1-xy}\Big]
\]
for the dilogarithm five-term relation in its 5-cyclic form.

Without loss of generality, we can replace \( \II_{4,1} \) with \( \wI_{4,1} \), since they are equal up to explicit \( \Lic_5 \) terms.

\begin{lemma}\label{lem:gr5sum1}
    The combination
    \[
    \Alt_{10} \Big[ \II_{4,1}^+\Big(\frac{\rho _{2,3}}{\rho _{2,1}},\frac{\rho _{4,3}}{\rho_{4,5}}\Big)
    -\II_{4,1}^+\Big(\frac{\rho _{4,3}}{\rho _{4,5}},\frac{\rho _{2,3}}{\rho _{2,1}}\Big) \Big]
    \]
    vanishes identically.
\end{lemma}

\begin{lemma}\label{lem:gr5sum2}
    We have
    \begin{alignat*}{2}
    & \mathrlap{
        \Alt_{10} \Big[ 2 \wI_{4,1}^+\Big(\frac{\rho _{1,2} \rho _{3,4}}{\rho _{1,4} \rho _{3,2}},\frac{\rho _{4,5}}{\rho_{4,1}}\Big) \Big]
    }
    \\ & = \Alt_{10} \Big[
    \wI_{4,1}^+\Big(
    && -2 V\Big(\frac{\rho_{2,1}}{\rho_{2,3}}, \frac{\rho_{3,4}}{\rho_{1,4}} \Big)
    - V([340|9625], [234 | 1650])
    + V([450|9736], [345 | 2760])
    \\ & && + \frac{1}{3} V([340|9625], [340|9657])
    + \frac{1}{3} V([345|2706], [347|2560]),
    \frac{\rho _{4,5}}{\rho_{4,1}}\Big)
    \Big] \,.
    \end{alignat*}
\end{lemma}

We first note a useful \( \Lic_3 \) functional equation which will enter as part of the reduction.

\begin{lemma}\label{lem:gr5li3}
    Let
    \begin{align*}
    \mathcal{T} \coloneqq {} & 3 \left[\frac{\rho _{1,2} \rho _{3,4}}{\rho _{1,4} \rho _{3,2}}\right]
    -3 [\R_3(14|259,370)]
    -3 [\R_3(54|179,260)]
    \\&
    +[\R_3(14|257,369)]
    +[\R_3(14|259,376)]
    -[\R_3(94|150,276)]
    \\&
    -[\R_3(94|157,260)]
    -[\CR(154|2769)]
    +[\CR(194|2570)] \,.
    \end{align*}
    Then following is a \( \Lic_3 \) functional equation
    \begin{align*}
    \Alt_{\mathfrak{S}_{\{1,2,3\}} \times \mathfrak{S}_{\{5,6,7\}}} \Lic_3(\mathcal{T}) \modsh 0 \,.
    \end{align*}
\end{lemma}

Currently, we do not reduce the third orbit directly to the 22-term, or the 840-term \( \Lic_3 \) functional equation.  Instead we invoke more general functional equations to simplify the reduction for the moment.

\begin{lemma}\label{lem:gr5sum3}
    The combination
    \begin{align*}
    \Alt_{10} \Big[ &  2 \wI_{4,1}^+\Big(\frac{\rho _{4,3}}{\rho _{4,5}},\frac{\rho _{2,3}}{\rho_{2,1}}\Big)
    +2 \wI_{4,1}^+\Big(\frac{\rho _{4,5}}{\rho _{4,1}},\frac{\rho _{1,2} \rho _{3,4}}{\rho	_{1,4} \rho _{3,2}}\Big)
    +\wI_{4,1}^+\Big(\frac{\rho _1}{\rho _{1,4}},\frac{\rho _{1,2} \rho _{3,4}}{\rho _{1,4}\rho _{3,2}}\Big)
    +2\wI_{4,1}^+\Big(\frac{\rho _{1,2}}{\rho _1},\frac{\rho _{3,2}}{\rho_{3,4}}\Big)
    \\&
    -\frac{4}{3} \wI_{4,1}^+\Big(\frac{\rho _{3,2}}{\rho _{3,4}},\frac{\rho_{1,2}}{\rho _1}\Big)
    + \frac{5}{3} \wI_{4,1}^+(\CR(346|1279),\R_3^-(12|345,678))  \Big]
    \end{align*}

    can be decomposed into a sum of the form
    \[
    \Alt_{10} \Big[ \sum_i \xi_i \wI_{4,1}^+(\Xi_i, y_i) + \sum_j \lambda_j \wI_{4,1}^+(x_j, \Lambda_j) \Big]
    \]
    where \( \Xi_i \) are \( \Lic_2 \) functional equations, and \( \Lambda_j \) are \( \Lic_3 \) functional equations.
\end{lemma}

\begin{proof}[Proof of Lemma \ref{lem:gr5sum1}]
    The involution \( \cyc{1,8}\cyc{2,7}\cyc{3,6}\cyc{4,5} \) of signature \( +1 \), induces the map \( \rho_i \mapsto \rho_{6-i} \), \( i = 1, \ldots, 5 \).  Under this, the second summand maps exactly to the first, and the combination equals
    \begin{align*}
       & \Alt_{10} \Big[ \II_{4,1}^+\Big(\frac{\rho _{2,3}}{\rho _{2,1}},\frac{\rho _{4,3}}{\rho_{4,5}}\Big)
        -\II_{4,1}^+\Big(\frac{\rho _{4,3}}{\rho _{4,5}},\frac{\rho _{2,3}}{\rho _{2,1}}\Big) \Big] \\
    & =  \Alt_{10} \Big[ (1 - \sigma_{\cyc{1,8}\cyc{2,7}\cyc{3,6}\cyc{4,5}}) \II_{4,1}^+\Big(\frac{\rho _{2,3}}{\rho _{2,1}},\frac{\rho _{4,3}}{\rho_{4,5}}\Big)
     \Big]  \\
     & = 0 \,. \qedhere
    \end{align*}
\end{proof}

\begin{proof}[Proof of Lemma \ref{lem:gr5sum2}]
    Choosing \( x = \frac{\rho_{2,1}}{\rho_{2,3}} \), \( y = \frac{\rho_{3,4}}{\rho_{1,4}} \), we obtain
    \begin{align*}
    \wI_{4,1}^+\Big(V\Big(\frac{\rho_{2,1}}{\rho_{2,3}}, \frac{\rho_{3,4}}{\rho_{1,4}} \Big), \frac{\rho _{4,5}}{\rho_{4,1}} \Big) ={}
    &\wI_{4,1}^+\Big(\frac{\rho _{1,3} \rho _{2,4}}{\rho _{1,4} \rho _{2,3}},\frac{\rho _{4,5}}{\rho_{4,1}}\Big)
    +\wI_{4,1}^+\Big(\frac{\rho _{2,1}}{\rho _{2,3}},\frac{\rho _{4,5}}{\rho_{4,1}}\Big)
    +\wI_{4,1}^+\Big(\frac{\rho _{2,3}}{\rho _{2,4}},\frac{\rho _{4,5}}{\rho_{4,1}}\Big)
    \\
    &+\wI_{4,1}^+\Big(\frac{\rho _{1,4}}{\rho _{2,4}},\frac{\rho _{4,5}}{\rho_{4,1}}\Big)
    +\wI_{4,1}^+\Big(\frac{\rho _{3,4}}{\rho _{1,4}},\frac{\rho _{4,5}}{\rho_{4,1}}\Big)\,.
    \end{align*}

    Observe that, under the six-fold symmetry in the first argument, we have
    \[
    \wI_{4,1}^+\Big(\frac{\rho _{1,3} \rho _{2,4}}{\rho _{1,4} \rho _{2,3}},\frac{\rho _{4,5}}{\rho_{4,1}}\Big) = -\wI_{4,1}^+\Big(\frac{\rho _{1,2} \rho _{3,4}}{\rho _{1,4} \rho _{3,2}},\frac{\rho _{4,5}}{\rho_{4,1}}\Big) \,.
    \]
    Moreover, note that the two \( \wI_{4,1}^+ \) terms on the second line are invariant under the transposition \( \cyc{2,3} \) and \( \cyc{1,2} \), respectively.  Hence under \( \Alt_{10} \) they vanish identically.  Overall
    \begin{equation}\label{eqn:singlesub0}
    2\wI_{4,1}^+\Big(\frac{\rho _{1,2} \rho _{3,4}}{\rho _{1,4} \rho _{3,2}},\frac{\rho _{4,5}}{\rho_{4,1}}\Big)
    \modaltt
    -2\wI_{4,1}^+\Big(V\Big(\frac{\rho_{2,1}}{\rho_{2,3}}, \frac{\rho_{3,4}}{\rho_{1,4}} \Big), \frac{\rho _{4,5}}{\rho_{4,1}} \Big)
    +2\wI_{4,1}^+\Big(\frac{\rho _{2,1}}{\rho _{2,3}},\frac{\rho _{4,5}}{\rho_{4,1}}\Big)
    +2\wI_{4,1}^+\Big(\frac{\rho _{2,3}}{\rho_{2,4}},\frac{\rho_{4,5}}{\rho_{4,1}}\Big)
    \end{equation}

    Note that
    \[
    \frac{\rho_{2,1}}{\rho_{2,3}} = \frac{\CR(340 | 9625)}{\CR(342 | 1605)} \,.
    \]
    If we write \( r_1 = \CR(340|9625) \), \( r_2 = \CR(234 | 1605) \), then
    \begin{equation}\label{eqn:singlesub1}
    \begin{split}
    \wI_{4,1}^+\Big(V(r_1, r_2^{-1}), \frac{\rho _{4,5}}{\rho_{4,1}}\Big) &= \wI_{4,1}^+\Big([r_1] - [r_2] + [1-r_0] + [(1-\sigma_{\cyc{5,6}}r_0)^{-1}] + [(1 - r_3)^{-1}], \frac{\rho _{4,5}}{\rho_{4,1}}\Big) \\
    &= \wI_{4,1}^+\Big( - [r_0] + [\sigma_{\cyc{5,6}}r_0] + [r_1] - [r_2] + [r_3], \frac{\rho _{4,5}}{\rho_{4,1}}\Big)
    \end{split}
    \end{equation}
    where \( r_0 = \frac{\rho_{2,1}}{\rho_{2,3}} \) and \( r_3 = \frac{\detv{12345}\detv{34690}}{\detv{12346}\detv{34590}} \).  Notice that \( \sigma_{\cyc{2,3}} r_2 = r_2 \) and \( \sigma_{\cyc{1,2}} r_3 = r_3 \) and both of these involutions fix \( \frac{\rho_{4,5}}{\rho_{4,1}} \).  Hence the last two terms in \eqref{eqn:singlesub1} vanish after skew-symmetrization.  Since \( \cyc{5,6} \) also fixes the second argument, we obtain
    \[
    -2\wI_{4,1}^+\Big(r_0, \frac{\rho _{4,5}}{\rho_{4,1}}\Big) \modaltt \wI_{4,1}^+\Big(V(r_1, r_2^{-1}), \frac{\rho _{4,5}}{\rho_{4,1}}\Big) - \wI_{4,1}^+\Big(r_1,\frac{\rho _{4,5}}{\rho_{4,1}}\Big)
    \]

    Next, notice that up to six-fold symmetries
    \[
    \wI_{4,1}^+\Big(\frac{\rho _{2,3}}{\rho_{2,4}}, w\Big) = -\wI_{4,1}^+\Big( \frac{\rho_{3,2}}{\rho_{3,4}}, w\Big)
    \] and
    \[
    \frac{\rho_{3,2}}{\rho_{3,4}} = \sigma_{\cyc{1,2,3,4,5,6,7}} \frac{\rho_{2,1}}{\rho_{2,3}} \,.
    \]
    Unfortunately this permutation does not fix the second argument of \( \wI_{4,1}^+ \).  Nevertheless by applying it the generators of the five-term relation, we immediately obtain
    \[
    \wI_{4,1}^+\Big(V(r'_1, (r'_2)^{-1}), \frac{\rho _{4,5}}{\rho_{4,1}}\Big) = \wI_{4,1}^+\Big( - [r_0'] + [\sigma_{\cyc{6,7}}r_0'] + [r_1'] - [r_2'] + [r_3'], \frac{\rho _{4,5}}{\rho_{4,1}}\Big) \,,
    \]
    where
    \begin{alignat*}{2}
    r_0' &= \frac{\rho_{3,2}}{\rho_{3,4}} \,, & \quad r_1' &= \CR(450|9736) \,, \\
    r_2' &= \CR(345 | 2706) \,, & \quad r_3' &= \frac{\detv{23456}\detv{54790}}{\detv{23457}\detv{45690}} \,.
    \end{alignat*}
    This time \( \sigma_{\cyc{1,2}}r_1' = r_1' \), \( \sigma_{\cyc{2,3}} r_3' = r_3' \) and each of the involutions \( \cyc{1,2} \), \( \cyc{2,3} \) and \( \cyc{6,7} \) fixes \( \frac{\rho _{4,5}}{\rho_{4,1}} \).  Hence after skew-symmetrization we find
    \begin{equation}\label{eqn:singlesub2}
    -2\wI_{4,1}^+\Big(r_0', \frac{\rho _{4,5}}{\rho_{4,1}}\Big) \modaltt
    \wI_{4,1}^+\Big(V(r'_1, (r'_2)^{-1}), \frac{\rho _{4,5}}{\rho_{4,1}}\Big) +	\wI_{4,1}^+\Big(r_2', \frac{\rho _{4,5}}{\rho_{4,1}}\Big) \,.
    \end{equation}

    From \eqref{eqn:singlesub0}, \eqref{eqn:singlesub1} and \eqref{eqn:singlesub2} we obtain
    \begin{equation}\label{eqn:singlesubpart}
    \begin{split}
    & 2 \II_{4,1}^+\Big(\frac{\rho _{1,2} \rho _{3,4}}{\rho _{1,4} \rho _{3,2}},\frac{\rho _{4,5}}{\rho_{4,1}}\Big) \modaltt {} \\
    & \wI_{4,1}^+\Big([340 | 9625],\frac{\rho _{4,5}}{\rho_{4,1}}\Big)
    +\wI_{4,1}^+\Big([345 | 2706], \frac{\rho _{4,5}}{\rho_{4,1}}\Big)
    \\ &
    + \wI_{4,1}^+\Big(
    -2 V\Big(\frac{\rho_{2,1}}{\rho_{2,3}}, \frac{\rho_{3,4}}{\rho_{1,4}} \Big)
    - V([340|9625], [234 | 1650])
    + V([450|9736], [345 | 2760]),
    \frac{\rho _{4,5}}{\rho_{4,1}}\Big) \,.
    \end{split}
    \end{equation}
    To reduce \( \wI_{4,1}^+\big([340 | 9625] + [345 | 2706],\frac{\rho _{4,5}}{\rho_{4,1}}\big) \), consider the following five-term relations.  Firstly
    \begin{equation}\label{eqn:singlesubV1}
    \begin{split}
    &\wI_{4,1}^+\Big(V([340|9625], [340|9657]),  \frac{\rho _{4,5}}{\rho_{4,1}}\Big) \\
    & = \wI_{4,1}^+\Big([340|9625] + [340|9657] + [340|7652] + [340|9267] + [340|9275],  \frac{\rho _{4,5}}{\rho_{4,1}}\Big) \\
    &\modaltt \wI_{4,1}^+\Big( 3[340|9625] + [340|2567], \frac{\rho _{4,5}}{\rho_{4,1}}\Big)
    \end{split}
    \end{equation}
    using the six-fold anharmonic symmetry, and the invariance of the second argument under arbitrary permutations of \( \{1,2,3\} \) and \( \{5,6,7\} \).  The term containing cross-ratio \( [340|9657] \) vanishes because it is invariant under \( \cyc{1,2} \).  Then
    \begin{equation}\label{eqn:singlesubV3}
    \begin{split}
    &\wI_{4,1}^+\Big(V([345|2706], [347|2560]),  \frac{\rho _{4,5}}{\rho_{4,1}}\Big) \\
    & = \wI_{4,1}^+\Big([345|2706]+ [347|2560]+ [342|5076]+ [340|2576]+ [346|2750],  \frac{\rho _{4,5}}{\rho_{4,1}}\Big) \\
    &\modaltt \wI_{4,1}^+\Big( 3[345|2706] - [340|2567], \frac{\rho _{4,5}}{\rho_{4,1}}\Big)
    \end{split}
    \end{equation}
    again by the six-fold symmetry, and by the invariance of the second argument under permutations of \( \{5,6,7\} \).  The third term vanishes because it is invariant under \( \cyc{2,3} \).

    From \eqref{eqn:singlesubV1} and \eqref{eqn:singlesubV3} we conclude
    \begin{align*}
    & \wI_{4,1}^+\big([340 | 9625] + [345 | 2706],\frac{\rho _{4,5}}{\rho_{4,1}}\big)
    \\
    & {}  \modaltt \frac{1}{3} \wI_{4,1}^+\Big(V([340|9625], [340|9657]) + V([345|2706], [347|2560]),  \frac{\rho _{4,5}}{\rho_{4,1}}\Big) \,.
    \end{align*}
    Together with \eqref{eqn:singlesubpart} this establishes the claim.
\end{proof}

\begin{proof}[Proof of Lemma \ref{lem:gr5sum3}]
	Let \( \Omega \) denote the combination
	\begin{align*}
	\Omega = {} & 2 \wI_{4,1}^+\Big(\frac{\rho _{4,3}}{\rho _{4,5}},\frac{\rho _{2,3}}{\rho_{2,1}}\Big)
	+2 \wI_{4,1}^+\Big(\frac{\rho _{4,5}}{\rho _{4,1}},\frac{\rho _{1,2} \rho _{3,4}}{\rho	_{1,4} \rho _{3,2}}\Big)
	+\wI_{4,1}^+\Big(\frac{\rho _1}{\rho _{1,4}},\frac{\rho _{1,2} \rho _{3,4}}{\rho _{1,4}\rho _{3,2}}\Big)
	+2\wI_{4,1}^+\Big(\frac{\rho _{1,2}}{\rho _1},\frac{\rho _{3,2}}{\rho_{3,4}}\Big)
	\\&
	-\frac{4}{3} \wI_{4,1}^+\Big(\frac{\rho _{3,2}}{\rho _{3,4}},\frac{\rho_{1,2}}{\rho _1}\Big)
	+ \frac{5}{3} \wI_{4,1}^+(\CR(346|1279),\R_3^-(12|345,678))  \,.
	\end{align*}

	First note the following five-term relation (here 0 is not a vector index, but the number)
	\begin{align*}
	& \wI_{4,1}^+\Big(V\big( \CR(\rho_4\infty\rho_10), \CR(\rho_4\infty0\rho_5) \big), \frac{\rho _{1,2} \rho _{3,4}}{\rho	_{1,4} \rho _{3,2}} \Big) \\
	&= \wI_{4,1}^+\Big( [\CR(\rho_4\infty\rho_10)] + [\CR(\rho_4\infty0\rho_5)] + [\CR(\rho_5\infty0\rho_1)] + [\CR(\rho_1\rho_1\infty\rho_5)] + [\CR(\rho_4\rho_1\rho_50)], \frac{\rho _{1,2} \rho _{3,4}}{\rho	_{1,4} \rho _{3,2}}) \,.
	\end{align*}
	Up to six-fold symmetries this is equal to
	\begin{align*}
	& \wI_{4,1}^+\Big(
	\Big[\frac{\rho_1}{\rho_{1,4}}\Big]
	- \Big[\frac{\rho_4}{\rho_5}\Big]
	+ \Big[\frac{\rho_1}{\rho_5}\Big]
	+ \Big[\frac{\rho_{4,5}}{\rho_{4,1}}\Big]
	- \Big[\sigma_{\cyc{9,10}}\frac{\rho_{4,5}}{\rho_{4,1}}\Big], \frac{\rho _{1,2} \rho _{3,4}}{\rho	_{1,4} \rho _{3,2}} \Big)  \\
	& \modaltt
	\wI_{4,1}^+\Big(
	\Big[\frac{\rho_1}{\rho_{1,4}}\Big]
	+ 2\Big[\frac{\rho_{4,5}}{\rho_{4,1}}\Big]
	- \Big[\frac{\rho_4}{\rho_5}\Big]
	+ \Big[\frac{\rho_1}{\rho_5}\Big], \frac{\rho _{1,2} \rho _{3,4}}{\rho	_{1,4} \rho _{3,2}} \Big) \,,
	\end{align*}
	the equality after skew-symmetrization following since \( \frac{\rho _{1,2} \rho _{3,4}}{\rho_{1,4} \rho _{3,2}} \) is invariant under the transposition \( \cyc{9,10} \).

	We claim now that modulo \( \Alt_{10} \) the orbit \( \wI_{4,1}^+\Big( \frac{\rho_1}{\rho_5}, \frac{\rho _{1,2} \rho _{3,4}}{\rho	_{1,4} \rho _{3,2}} \Big) \) is a combination of \( \Lic_3 \) functional equations in the second argument.  Indeed, since the first argument is invariant under \( \mathfrak{S}_{\{1,2,3,4\}} \times \mathfrak{S}_{\{5,6,7,8\}} \), we get
	\begin{align*}
	\wI_{4,1}^+\Big( \frac{\rho_1}{\rho_5}, \frac{\rho _{1,2} \rho _{3,4}}{\rho	_{1,4} \rho _{3,2}}\Big)
	& = \frac{1}{3} \wI_{4,1}^+\Big(\frac{\rho_1}{\rho_5}, \mathcal{T}
	+3 [\R_3(14|259,370)]
	+3 [\R_3(54|179,260)]
	\\
	&
	\quad\quad\quad\quad\quad -[\R_3(14|257,369)]
	-[\R_3(14|259,376)]
	+[\R_3(94|150,276)]
	\\&
	\quad\quad\quad\quad\quad +[\R_3(94|157,260)]
	+[\CR(154|2769)]
	-[\CR(194|2570)] \Big) \\
	& \modaltt \frac{1}{3} \wI_{4,1}^+\Big(\frac{\rho_1}{\rho_5}, \mathcal{T}
	+3 [\R_3(54|179,260)]
	+[\R_3(94|150,276)]
	+ [\R_3(94|157,260)] \Big) \,,
	\end{align*}
	since the removed terms are also invariant under permutations of \( \{1,2,3,4\} \).  Notice now that
	\[
	\wI_{4,1}^+\Big(\frac{\rho_{1}}{\rho_{5}}, x\Big) \modaltt \frac{1}{4} \wI_{4,1}^+\Big(\frac{\rho_{1}}{\rho_{5}}, (1 + \sigma_{\cyc{9,0}})(1 - \sigma_{\cyc{1,5}\cyc{2,6}\cyc{3,7}\cyc{4,8}}) x \Big)
	\]
	using the 6-fold symmetry, since the indicated permutations fix or invert the first argument.

	We check the functional equation \( \Lic_3(\Lambda_i) \modsh 0 \) holds, for the following combinations
	\begin{align*}
	\Lambda_1 &= \Alt_{\mathfrak{S}_{\{1,2,3\}} \times \mathfrak{S}_{\{5,6,7\}}} \mathcal{T} \,, \\
	\Lambda_2 &= \Alt_{\mathfrak{S}_{\{1,2,3,4\}} \times \mathfrak{S}_{\{5,6,7,8\}}} (1 + \sigma_{\cyc{9,0}})(1 - \sigma_{\cyc{1,5}\cyc{2,6}\cyc{3,7}\cyc{4,8}}) [\R_3(54|179,260)] \,, \\
	\Lambda_3 &= \Alt_{\mathfrak{S}_{\{1,2,3,4\}} \times \mathfrak{S}_{\{5,6,7,8\}}} (1 + \sigma_{\cyc{9,0}})(1 - \sigma_{\cyc{1,5}\cyc{2,6}\cyc{3,7}\cyc{4,8}}) \big( [\R_3(94|150,276)]
	+ [\R_3(94|157,260)] \big) \,.
	\end{align*}
	Hence
	\[
	\Alt_{10}\Big[\wI_{4,1}^+\Big( \frac{\rho_1}{\rho_5}, \frac{\rho _{1,2} \rho _{3,4}}{\rho	_{1,4} \rho _{3,2}}\Big) \Big]
	= \Alt_{10} \Big[ \wI_{4,1}^+\Big(\frac{\rho_1}{\rho_5},
	\frac{1}{3 \cdot 3!^2} \Lambda_1 + \frac{1}{4 \cdot 4!^2} (3\Lambda_2 + \Lambda_3)
	\Big) \Big] \,,
	\]
	and so
	\begin{equation}\label{eqn:sum1comb1}
	\begin{split}
	& \Alt_{10} \Big[ \wI_{4,1}^+\Big(
	\Big[\frac{\rho_1}{\rho_{1,4}}\Big]	+ 2\Big[\frac{\rho_{4,5}}{\rho_{4,1}}\Big], \frac{\rho _{1,2} \rho _{3,4}}{\rho	_{1,4} \rho _{3,2}} \Big) \Big] =
	\Alt_{10} \Big[ \wI_{4,1}^+\Big(\frac{\rho_4}{\rho_5}, \frac{\rho _{1,2} \rho _{3,4}}{\rho	_{1,4} \rho _{3,2}} \Big)  \\
	& \quad\quad + \wI_{4,1}^+\Big(V\big( \CR(\rho_4\infty\rho_10), \CR(\rho_4\infty0\rho_5) \big), \frac{\rho _{1,2} \rho _{3,4}}{\rho	_{1,4} \rho _{3,2}} \Big)
	- \wI_{4,1}^+\Big(\frac{\rho_1}{\rho_5},
	\frac{1}{3 \cdot 3!^2} \Lambda_1 + \frac{1}{4 \cdot 4!^2} (3\Lambda_2 + \Lambda_3)
	\Big) \Big] \,.
	\end{split}
	\end{equation}

	Note now the following five-term relations
	\begin{align*}
	& \wI_{4,1}^+(V(\CR(564 | 0387), \CR(560|4987)^{-1}), \frac{\rho_{3,2}}{\rho_{1,2}})
	= \wI_{4,1}^+\Big([\CR(564 | 0387)] - [\CR(560|4987)] \\ & \quad\quad + \Big[1-\sigma_{\cyc{7,8}}\frac{\rho_{4,3}}{\rho_{4,5}}\Big] + \Big[\Big(1- \frac{\rho_{4,3}}{\rho_{4,5}}\Big)^{-1}\Big] + \Big[\Big(1-\frac{\detv{34568}\detv{56790}}{\detv{34567}\detv{56890}}\Big)^{-1}\Big], \frac{\rho_{3,2}}{\rho_{1,2}}\Big) \,,
	\end{align*}
	\begin{align*}
	& \wI_{4,1}^+(V(\CR(453 | 0276), \CR(450|3976)^{-1}), \frac{\rho_{1,2}}{\rho_{1}})
	= \wI_{4,1}^+\Big([\CR(453 | 0276)] - [\CR(450|3976)] \\ & \quad\quad + \Big[1-\sigma_{\cyc{6,7}}\frac{\rho_{3,2}}{\rho_{3,4}}\Big] + \Big[\Big(1- \frac{\rho_{3,2}}{\rho_{3,4}}\Big)^{-1}\Big] + \Big[\Big(1-\frac{\detv{23457}\detv{45690}}{\detv{23456}\detv{45790}}\Big)^{-1}\Big], \frac{\rho_{1,2}}{\rho_{1}}\Big) \,.
	\end{align*}
	In the former the last term is invariant under \( \cyc{3,4} \) so vanishes after skew-symmetrization, in the latter it is invariant under \( \cyc{2,3} \) and so vanishes also. In the latter, the term \( [\CR(453|0276)] \) also vanishes due to invariance under \( \cyc{3,4} \).  From this we obtain
	\begin{equation}\label{eqn:sum1comb2}
	\begin{split}
	& \Alt_{10} \Big[ -\frac{4}{3} \wI_{4,1}^+\Big(\frac{\rho _{3,2}}{\rho _{3,4}},\frac{\rho_{1,2}}{\rho _1}\Big) +2 \wI_{4,1}^+\Big(\frac{\rho _{4,3}}{\rho _{4,5}},\frac{\rho _{2,3}}{\rho_{2,1}}\Big) \Big] \\
	&= \Alt_{10} \Big[
	\wI_{4,1}^+(V(\CR(564 | 0387), \CR(560|4987)^{-1}) , \frac{\rho_{3,2}}{\rho_{1,2}})\\
	& \quad\quad\quad - \frac{2}{3} \wI_{4,1}^+\Big(V(\CR(453 | 0276), \CR(450|3976)^{-1}), \frac{\rho_{1,2}}{\rho_{1}}\Big) \\
	& \quad\quad\quad - [\CR(564 | 0387)] + [\CR(560|4987)]
	- \tfrac{2}{3} [\CR(450|3976)]
	\Big]\,.
	\end{split}
	\end{equation}

	From \( \mathcal{T} \), we again obtain
	\begin{equation}\label{eqn:sum1comb3}
	\begin{split}
	& \Alt_{10} \Big[ \wI_{4,1}^+\Big( \frac{\rho_4}{\rho_5}, \frac{\rho _{1,2} \rho _{3,4}}{\rho_{1,4} \rho _{3,2}}\Big) \Big] \\
	& = \Alt_{10} \Big[ \frac{1}{3} \wI_{4,1}^+\Big(\frac{\rho_4}{\rho_5}, \frac{1}{3!^2} \Lambda_1
	+3 [\R_3(14|259,370)]
	+3 [\R_3(54|179,260)]
	\\
	&
	\quad\quad\quad\quad\quad +[\R_3(14|257,369)]
	-[\R_3(14|259,376)]
	+[\R_3(94|150,276)]
	\\&
	\quad\quad\quad\quad\quad +[\R_3(94|157,260)]
	+[\CR(154|2769)]
	-[\CR(194 |2570)] \Big) \Big] \,.
	\end{split}
	\end{equation}

	We can substitute \eqref{eqn:sum1comb1}, \eqref{eqn:sum1comb2} and \eqref{eqn:sum1comb3} into the original combination \( \Omega \), and rewrite the remaining arguments in terms of cross-ratios and triple-ratios using
	\begin{alignat*}{2}
	\frac{\rho_1}{\rho_2} &= \CR(234 | 1590) \,, &
	\frac{\rho_4}{\rho_5} &= \CR(567 | 4890) \,, \\
	\frac{\rho_{2,3}}{\rho_{2,1}} &= \R_3(34 | 520 , 619) \,, & \quad
	\frac{\rho_{3,2}}{\rho_{3,4}} &= \R_3(45 | 360 , 279) \,.
	\end{alignat*}
	Moreover, we can put the second argument into a canonical form, namely \( \CR(123|4567) \) or \( \R_3(12|345,678) \) respectively, by choosing the inverse of the permutation which maps \( \{1,\ldots,10\} \) to the points which appear in the second argument, and then the complementary points in order of index.   Drop, for simplicity, the functional equations appearing in \eqref{eqn:sum1comb1}, \eqref{eqn:sum1comb2} and \eqref{eqn:sum1comb3}.  Note also that
	\begin{align*}
		& \wI_{4,1}^+(\CR(346|1279), \R_3^-(12|345,678)) \\
		&{}\:\:=\:\: \wI_{4,1}^+(\CR(346|1279), \R_3(12|345,678)) - \wI_{4,1}^+(\CR(346|1279), \R_3(12|678,345)) \\
		&{}\:\:=\:\: \wI_{4,1}^+(\CR(346|1279), \R_3(12|345,678)) - \sigma_{\cyc{3,6}\cyc{5,7}\cyc{4,8}} \wI_{4,1}^+(\CR(368|1259), \R_3(12|354,687)) \\
		&{}\modaltt \wI_{4,1}^+([\CR(346|1279)] + [\CR(368|1259)], \R_3(12|345,678)) \,.
	\end{align*}
	So we find that \( \Omega \) reduces to
	\begin{equation}\label{eqn:omegacrtr1}
	\begin{split}
	\tfrac{1}{3} \wI_{4,1}^+(& [\CR(256|3970)] +2 [\CR(356|2798)] + [\CR(569|2730)],\CR(123|4567)) \\[1ex]
	{} + \wI_{4,1}^+(
	& {-}2[\CR(136|2958)]
	-[\CR(147|2058)]
	+[\CR(236|1590)]
	-[\CR(356|2809)]
	\\&+[\CR(479|2058)]
	+\tfrac{5}{3} [\CR(346|1279)]
	+\tfrac{5}{3} [\CR(368|1259)]
	-\tfrac{1}{3} [\CR(457|1820)]
	\\&-\tfrac{1}{3} [\CR(457|2980)]
	-\tfrac{1}{3} [\CR(478|1520)]
	-\tfrac{1}{3} [\CR(478|2950)],\R_3(12|345,678)) \,.
	\end{split}
	\end{equation}
	Under the automorphisms of \( g = \R_3(12|345,678) \) including inverting, and the six-fold symmetries, we note the following equalities
	\begin{align*}
	\sigma_{\cyc{3,4}\cyc{6,7}\cyc{9,0}} \wI_{4,1}^+(\CR(136|2958),g) &= \wI_{4,1}^+(\CR(147|2058), g) \,, \\
	\sigma_{\cyc{3,4,5}\cyc{6,7,8}\cyc{9,0}} \wI_{4,1}^+(\CR(346|1279),g) &= \wI_{4,1}^+(\CR(457|1280), g) = -\wI_{4,1}^+(\CR(457|1820), g) \,, \\
	\sigma_{\cyc{3,4}\cyc{6,7}} \wI_{4,1}^+(\CR(356|2809),g) &= \wI_{4,1}^+(\CR(457|2809), g) = \wI_{4,1}^+(\CR(457|2980),g) \,, \\
	\sigma_{\cyc{3,4}\cyc{6,7}\cyc{9,0}} \wI_{4,1}^+(\CR(368|1259),g) &= \wI_{4,1}^+(\CR(478|1250), g) = - \wI_{4,1}^+(\CR(478|1520), g) \,.
	\end{align*}
	We also note, under the automorphisms of \( p = \CR(123|4567) \), that
	\[
	\sigma_{\cyc{2,3}\cyc{8,0}} \wI_{4,1}^+(\CR(256|3970), p) = \wI_{4,1}^+(\CR(356|2978), p) = -\wI_{4,1}^+(\CR(356|2798), p) \,.
	\]
	So the above combination \eqref{eqn:omegacrtr1} is \( \Alt_{10} \)-equivalent to
	\begin{equation}\label{eqn:omegacrtr2}
	\begin{split}
	\tfrac{1}{3} \wI_{4,1}^+(& -[\CR(256|3970)] + [\CR(569|2730)],\CR(123|4567)) \\[1ex]
	{} + \wI_{4,1}^+(
	&\tfrac{4}{3} [\CR(346|1279)]
	+\tfrac{4}{3} [\CR(368|1259)]
	-\tfrac{4}{3} [\CR(356|2809)]
	-\tfrac{1}{3} [\CR(478|2950)] \\
	& -[\CR(136|2958)]
	+[\CR(236|1590)]
	+[\CR(479|2058)],\R_3(12|345,678)) \,.
	\end{split}
	\end{equation}

	We focus first on the (cross-ratio, cross-ratio) terms in \eqref{eqn:omegacrtr2}.  Consider the five-term
	\begin{align*}
	& \wI_{4,1}^+(V([256|7390], [567|3209]), [123|4567]) \\
	& = \wI_{4,1}^+([256|3790] + [567|3209] + [356|2970] + [569|3270] + [560|3729], [123|4567]) \,.
	\end{align*}
	The second term is invariant under \( \cyc{6,7} \) since it maps the second argument to its inverse; this term vanishes after skew-symmetrization.  Note that
	\begin{align*}
	\sigma_{\cyc{2,3}} \wI_{4,1}^+([562|3790],[123|4567]) &= \wI_{4,1}^+([563|2790],[123|4567]) = -\wI_{4,1}^+([563|2970],[123|4567]) \,, \\
	\sigma_{\cyc{0,9}} \wI_{4,1}^+([569|3270],[123|4567]) &= \wI_{4,1}^+([560|3279],[123|4567]) = -\wI_{4,1}^+([560|3729],[123|4567])  \,,
	\end{align*}
	so after skew-symmetrization the first and third, and fourth and fifth terms combine to give
	\begin{equation}\label{eqn:omegacrcr}
	\begin{split}
	& \frac{1}{6} \wI_{4,1}^+(V([256|7390], [567|3209]), [123|4567]) \\
	&= \frac{1}{3} \wI_{4,1}^+([256|3790] + [569|3270], [123|4567]) \\
	&= \frac{1}{3} \wI_{4,1}^+(-[256|3970] + [569|2730], [123|4567]) \,.
	\end{split}
	\end{equation}

	This leaves only the following (cross-ratio, triple-ratio) terms in \eqref{eqn:omegacrtr2} to reduce.  Unfortunately, the reduction here relies on finding a suitable decomposition purely with computer assistance.  Introduce the following combination
	\begin{alignat*}{7}
		\Psi = {} && 2 \, [134|2569]&&{}+{}\, [134|2590]&&{}-{}2 \, [134|2689]&&{}+{}\, [134|2890]&&{}+{}\, [134|5690]&&{}-{}\, [134|6890]
\\ && {}+{}24 \, [136|2479]&&{}+{}20 \, [136|2490]&&{}-{}4 \, [136|2790]&&{}-{}2 \, [137|2459]&&{}-{}2 \, [137|2489]&&	{}+{}2 \, [137|2569]
\\ && {}+{}2 \, [137|2590]&&{}-{}2 \, [137|2689]&&{}+{}2 \, [137|2890]&&{}-{}\, [137|4590]&&	{}-{}\, [137|4890]&&{}+{}\, [137|5690]
\\ && {}-{}\, [137|6890]&&{}-{}6 \, [139|2460]&&{}+{}6 \, [139|2670]&&{}-{}2 \, [167|2359]&&{}-{}2 \, [167|2389]&&{}+{}\, [167|2590]
\\ && {}+{}\, [167|2890]&&{}-{}\, [167|3590]&&{}-{}\, [167|3890]&&{}+{}6 \, [169|2340]&&{}+{}6 \, [169|2370]&&{}+{}5 \, [346|1259]
\\ && {}-{}18 \, [346|1279]&&{}+{}5 \, [346|1289]&&{}-{}8 \, [346|1290]&&{}+{}4 \, [346|1579]&&{}-{}3 \, [346|1590]&&{}-{}4 \, [346|1789]
\\ && {}+{}38 \, [346|1790]&&{}-{}3 \, [346|1890]&&{}-{}5 \, [349|1260]&&{}-{}18 \, [367|1249]&&{}+{}5 \, [367|1259]&&{}+{}5 \, [367|1289]
\\ && {}-{}8 \, [367|1290]&&{}-{}4 \, [367|1459]&&{}-{}4 \, [367|1489]&&{}+{}14 \, [367|1490]&&{}-{}3 \, [367|1590]&&{}-{}3 \, [367|1890]
\\ && {}-{}14 \, [369|1240]&&{}-{}14 \, [369|1270]&&{}+{}24 \, [369|1470]&&{}-{}5 \, [379|1240]&&{}-{}5 \, [379|1260]&&{}+{}5 \, [780|1259] \mathrlap{\,.}
	\end{alignat*}
	Denote the (cross-ratio, triple-ratio) terms in \eqref{eqn:omegacrtr2} by
	\begin{align*}
		\Omega' =  \wI_{4,1}^+(
		&\tfrac{4}{3} [\CR(346|1279)]
		+\tfrac{4}{3} [\CR(368|1259)]
		-\tfrac{4}{3} [\CR(356|2809)]
		-\tfrac{1}{3} [\CR(478|2950)] \\
		& -[\CR(136|2958)]
		+[\CR(236|1590)]
		+[\CR(479|2058)],\R_3(12|345,678)) \,.
	\end{align*}
	Then one can check that
	\begin{equation}\label{eqn:omegali2}
		\Omega' - \frac{1}{24} \wI_{4,1}^+(\Psi, \R_3(12|345,678))
	\end{equation}
	is a \( \Lic_2 \)-functional equation in the first arguments, under automorphisms and inversion of the triple-ratio \( \R_3(12|345,678) \).  In particular it will be expressible as a combination of five-term relations.  One can also check after permuting the points so that the first argument is \( \CR(123|4567) \), that
	\begin{equation}\label{eqn:omegali3}
		\Omega' + \frac{1}{24} \wI_{4,1}^+(\Psi, \R_3(12|345,678))
	\end{equation}
	 is a \( \Lic_3 \)-functional equation  in the second argument, under automorphisms of the cross-ratio \( \CR(123|4567) \) and the 6-fold symmetries.  From the sum of \eqref{eqn:omegali2} and \eqref{eqn:omegali3}, we conclude that \( \Omega' \) decomposes into \( I_{4,1}^+ \) combinations of purely \( \Lic_2 \) functional equations, and purely \( \Lic_3 \) functional equations, in the first and second argument, respectively.
	 This completes the decomposition of \( \Omega \) into such functional equations, and hence establishes the claim.
\end{proof}

\appendix
\medskip
\section{\texorpdfstring{An explicit expression for $\symsix(x,y)$ and $\fiveterm(z;x,y)$ in terms of $\Lic_4$}{An explicit expression for Sym\textunderscore{}36(x,y) and V(z;x,y) in terms of Li\textunderscore{}4}}
\label{app:i31}
For the sake of completeness we give explicitly the combination of \( \Lic_4 \) terms appearing on the
right-hand side of
\begin{equation}
	\II_{3,1}(x,y) - \wI_{3,1}(x,y) \modsh \sum_j \lambda_j \Lic_4(f_j(x,y)) \,,
\end{equation}
which we denoted by \( \symsix(x,y) \).  The combination can be obtained by applying Theorem \ref{thm:i31syms} to relate every \( \II_{3,1}(x^\sigma,y^\pi)) \) in \( \wI_{3,1}(x,y) \) back to \( \sgn(\sigma)\sgn(\pi) \II_{3,1}(x,y) \).  The resulting expression is as follows.
\begin{align*}
& \II_{3,1}(x,y) - \wI_{3,1}(x,y) \modsh \\
&-\frac{1}{12} \Lic_4\Big(\frac{(1-x) y^2}{x^2 (1-y)}\Big)
+\frac{1}{12} \Lic_4\Big(\frac{x^2 y}{(1-x) (1-y)^2}\Big)
+\frac{1}{12} \Lic_4\Big(\frac{x y^2}{(1-x)^2 (1-y)}\Big)
\\&
+\frac{1}{6} \Lic_4\Big(\frac{(1-x) x y^2}{y-1}\Big)
-\frac{1}{4} \Lic_4\Big(\frac{-x}{(1-x) (1-y)}\Big)
+\Lic_4\Big(\frac{1-y}{1-x}\Big)
+\frac{3}{4} \Lic_4\Big(\frac{x (1-y)}{x-1}\Big)
\\&
-\frac{1}{2} \Lic_4((1-x) y)
-\frac{1}{2} \Lic_4\Big(\frac{y}{1-x}\Big)
-\frac{3}{2} \Lic_4\Big(\frac{y}{x}\Big)
+\frac{1}{4} \Lic_4\Big(\frac{(x-1) y}{x}\Big)
-\frac{1}{2} \Lic_4(x y)
\\&
-\frac{1}{4} \Lic_4\Big(\frac{x y}{x-1}\Big)
+\frac{1}{4} \Lic_4\Big(\frac{(1-x)^2 y}{x (1-y)^2}\Big)
-\frac{1}{4} \Lic_4\Big(\frac{-y}{(1-x) (1-y)}\Big)
-\frac{5}{4} \Lic_4\Big(\frac{(1-x) y}{y-1}\Big)
\\&
+\frac{1}{4} \Lic_4\Big(\frac{y}{x (-1+y)}\Big)
-\Lic_4\Big(\frac{(1-x) y}{x (1-y)}\Big)
-\frac{3}{4} \Lic_4\Big(\frac{x y}{y-1}\Big)
-\frac{1}{2} \Lic_4\Big(\frac{x y}{(1-x) (1-y)}\Big)
\\&
+\frac{1}{12} \Lic_4\Big(\frac{(1-y) y}{(1-x) x}\Big)
-\Lic_4\Big(\frac{1}{1-x}\Big)
+\Lic_4(x)-\frac{1}{2} \Lic_4\Big(\frac{x}{x-1}\Big)
+\Lic_4\Big(\frac{1}{1-y}\Big)
+\frac{3}{2} \Lic_4\Big(\frac{y}{y-1}\Big) \,.
\end{align*}\pagebreak[1]

We also give explicitly the combination of \( \Lic_4 \) terms appearing on the
right-hand side of

\begin{equation} \label{eq:i31fiveterm2}
\wI_{3,1}\Big(z,[x]+[y]+\Big[\frac{1-x}{1-xy}\Big]+[1-xy]
+\Big[\frac{1-y}{1-xy}\Big]\Big)
\;\modsh\; \sum_{j} \nu_j\Lic_4(f_j(x,y,z)),
\end{equation}
which we denoted by \( \fiveterm(z; x, y) \).
The expression we give is only slightly different from
the one given in~\cite{Ga} in that we give a relation only
for the $36$-fold symmetrization of $I_{3,1}$. We write the
identity in the following symmetric form.
Choose $z_1,\dots,z_9\in\mathbb{P}^1(\CC)$ in such a way that
$z=\CR(z_1,z_2,z_3,z_4)$, $x=\CR(z_5,z_6,z_7,z_8)$,
$y=\CR(z_5,z_6,z_8,z_9)$, for example, we can
take $(z_1,\dots,z_9)=(\infty,0,1,z,1-x,0,1-\frac{1}{y},1,\infty)$.
Then the left-hand side of~\eqref{eq:i31fiveterm2} is skew-symmetric
under the action of $\mathfrak{S}_4\times \mathfrak{S}_5$ on the $9$ points
$z_1,\dots,z_9$. Thus we can decompose the $\Lic_4$ terms
into orbits under the action of $\mathfrak{S}_4\times \mathfrak{S}_5$.
The resulting expression is as follows.  Note that we write \( (abcd) = \CR(abcd) \) as shorthand for the individual cross-ratio in the \( \Lic_4 \) arguments, to differentiate them from the notation for formal linear combinations elsewhere.
	\begin{align*}
	\quad\frac{4}{3}\Alt_{\mathfrak{S}_4\times \mathfrak{S}_5}
	&\wI_{3,1}([1234],[5678]-[5679]+[5689]-[5789]+[6789])
	\\
	\modsh \Alt_{\mathfrak{S}_4\times \mathfrak{S}_5} \bigg[
	-&\Lic_4\left(-\frac{(1234)(5768)((7659)-(1234))}
	{(5968)^2 ((8659)-(1234))^2}\right)
	+2\Lic_4\left(\frac{(5978)((8659)-(1234))}
	{(7659)-(1234)}\right)\\
	-2&\Lic_4\left(\frac{(1324)((7659)-(1234))}
	{(5876)(5798)^2}\right)
	+2\Lic_4\left(\frac{(5987)((7659)-(1234))}
	{(1324)((8659)-(1234))}\right)\\
	+3&\Lic_4\left(\frac{(5968)(7689)((7659)-(1234))}
	{(1234)(1324)}\right)
	+4\Lic_4\left(\frac{(5986)((7659)-(1234))}
	{(8659)-(1234)}\right)\\
	+4&\Lic_4\left(\frac{(5896)}
	{(1324)(5678)}\right)
	-6\Lic_4\left(\frac{(7659)-(1234)}
	{(1324)(7689)}\right)
	+8\Lic_4\left(-\frac{(7659)-(1234)}
	{(1324)(7569)(8679)}\right)\bigg]\,.
	\end{align*}

\section{\texorpdfstring{Explicit expressions for symmetries of \( I_{4,1}^+(x, y) \) in terms of \( \Lic_5 \)}{Explicit expressions for symmetries of I\textunderscore{}\{4,1\}\textasciicircum{}+(x, y) in terms of Li\textunderscore{}5}}\label{app:i41}

Recall the function
\[
	\II_{4,1}^+(x,y) \coloneqq \frac{1}{2} \big( \II_{4,1}(x,y) + \II_{4,1}(x, y^{-1}) \big) \,.
\]
Modulo products, and \emph{explicit} \( \Lic_5 \) terms, it satisfies the \( \Lic_2 \) anharmonic symmetries in \( x \), and the \( \Lic_3 \) inversion in \( y \).  It also satisfies the \( \Lic_3 \) three-term relation (including constant term) \( \Lic_3(y) + \Lic_3(1-y) + \Lic_3(1-y^{-1}) \modsh \Lic_3(1) \) in \( y \).  Explicitly, we have the following identities.

\begin{theorem} The function \( I_{4,1}^+(x,y) \) satisfies the following symmetries and identities.

	\begin{enumerate}
\item[(i)] We have \(
\II_{4,1}^+(x,y) - \II_{4,1}^+(x,y^{-1}) = 0 \,.
\)

\item[(ii)] Modulo products the combination \( \II_{4,1}^+(x,y) + \II_{4,1}^+(x^{-1},y)  \) is equal to \begin{equation*}
-2 \Lic_5\Big(\frac{y}{x}\Big)-2 \Lic_5(x y)-\Lic_5(x)-\Lic_5(y) \,.
\end{equation*}

\item[(iii)] Modulo products the combination \( \II_{4,1}^+(x,y) + \II_{4,1}^+(1-x,y) \) is equal to  \begin{align*}
	& \frac{1}{12} \Lic_5\Big(\frac{x^2 y}{(1-x) (1-y)^2}\Big)
	+\frac{1}{12} \Lic_5\Big(\frac{(1-x)^2 y}{x (1-y)^2}\Big)
	+\frac{1}{6} \Lic_5\Big(\frac{(1-x) x y^2}{y-1}\Big)
	+\frac{1}{6} \Lic_5\Big(\frac{(1-y) y}{(1-x) x}\Big)
	\\&
	-\frac{1}{2} \Lic_5\Big(\frac{1-x}{x (y-1)}\Big)
	-\frac{1}{2} \Lic_5\Big(\frac{x y}{(1-x) (1-y)}\Big)
	-\frac{1}{2} \Lic_5\Big(\frac{(1-x) (1-y)}{-x}\Big)
	-\frac{1}{2} \Lic_5\Big(\frac{(1-x) y}{x (1-y)}\Big)
	\\&
	-\frac{7}{4} \Lic_5\Big(\frac{y}{1-x}\Big)
	-\frac{7}{4} \Lic_5((1-x) y)
	-\frac{7}{4} \Lic_5\Big(\frac{y}{x}\Big)
	-\frac{7}{4} \Lic_5(x y)
	\\&
	-\Lic_5\Big(\frac{1-y}{x}\Big)
	-\Lic_5\Big(\frac{1-x}{1-y}\Big)
	-\Lic_5\Big(\frac{(1-x) y}{y-1}\Big)
	-\Lic_5\Big(\frac{x y}{y-1}\Big)
	\\&
	+\frac{1}{2} \Lic_5(1-x)
	+\frac{1}{2} \Lic_5\Big(\frac{1}{x}\Big)
	+\Lic_5\Big(\frac{x-1}{x}\Big)
	+\Lic_5\Big(\frac{1}{1-y}\Big)
	+\Lic_5\Big(\frac{y}{y-1}\Big) \,.
\end{align*}

\item[(iv)] Modulo products the combination \(  \II_{4,1}^+(x,y) + \II_{4,1}^+(x,1-y)+ \II_{4,1}^+(x,1-y^{-1})- \II_{4,1}^+(x,1) \) is equal~to \begin{align*}
	& -\frac{1}{18} \Lic_5\Big(\frac{(1-x) y^2}{x^2 (1-y)}\Big)
	-\frac{1}{18} \Lic_5\Big(\frac{x^2 y}{(1-x) (1-y)^2}\Big)
	-\frac{1}{18} \Lic_5\Big(\frac{x^2 (1-y) y}{x-1}\Big)
	\\&
	+\frac{1}{36} \Lic_5\Big(\frac{x y^2}{(1-x)^2 (1-y)}\Big)
	+\frac{1}{36} \Lic_5\Big(\frac{(1-x)^2 (y-1) y}{x}\Big)
	+\frac{1}{36} \Lic_5\Big(\frac{(1-x)^2 y}{x (1-y)^2}\Big)
	\\&
	+\frac{1}{9} \Lic_5\Big(\frac{(x-1) x y^2}{1-y}\Big)
	+\frac{1}{9} \Lic_5\Big(\frac{(1-y) y}{(1-x) x}\Big)
	+\frac{1}{9} \Lic_5\Big(\frac{y}{(x-1) x (1-y)^2}\Big)
	\\&
	-\frac{1}{2} \Lic_5\Big(\frac{1}{(1-x) (1-y)}\Big)
	-\frac{1}{2} \Lic_5\Big(\frac{1-y}{1-x}\Big)
	-\frac{1}{2} \Lic_5\Big(\frac{y}{1-x}\Big)
	\\&
	-\frac{1}{2} \Lic_5\Big(-\frac{y}{(1-x) (1-y)}\Big)
	-\frac{1}{2} \Lic_5\Big(\frac{(1-x) y}{y-1}\Big)
	-\frac{1}{2} \Lic_5((1-x) y)
	\\&
	-\frac{5}{4} \Lic_5\Big(\frac{x}{1-y}\Big)
	-\frac{5}{4} \Lic_5(x (1-y))
	-\frac{5}{4} \Lic_5\Big(\frac{y}{x}\Big)
	-\frac{5}{4} \Lic_5(x y)
	\\&
	-\frac{5}{4} \Lic_5\Big(\frac{y}{x (y-1)}\Big)
	-\frac{5}{4} \Lic_5\Big(\frac{x y}{y-1}\Big)
	+\Lic_5\Big(\frac{1}{1-x}\Big)
	+\frac{3 \Lic_5(x)}{2} \,.
\end{align*}
\end{enumerate}
\end{theorem}

\begin{proof}  Each identity is checked directly on the level of the mod-products symbol.

	The identity in (i) is immediate from the definition of \( \II_{4,1}^+ \).  The identity in (ii) follows from the inversion property of \( \II_{a,b}(x^{-1},y^{-1}) \) given in Theorem 6.1.2 of \cite{Ch}
    (see also~\cite{Pa} for a more general version of the inversion property).

	The identity in (iii) can be obtained from the case \( a = 1, b = 0 \) of the reduction of \( I_{4,1}^+ \) under the so-called algebraic \( \Lic_2 \) functional equation \( \sum_i \Lic_2(p_i(t)) = 0 \) where \( p_i(t) \) are the roots counted with multiplicity of \( x^a(1-x)^b = t \).  This is given in Theorem 7.4.6 of \cite{Ch} for the related function \( I_{4,1}^-(x,y) \) and in Corollary 7.4.9 of \cite{Ch} for \( I_{4,1}^+(x,y) \) itself.

	The identity in (iv) can be obtained from Theorem 7.4.17 in \cite{Ch} where it is stated for the related function \( \II_{4,1}^-(x,y) \).  Note that the constant term is written using the Nielsen polylogarithm \( S_{3,2} \) instead of \( \II_{4,1}^{\pm} \) with one argument specialized to $1$, but they are related via \( S_{3,2}(x) \modsh \II_{4,1}(x,1) + 4\Lic_5(x) \).
\end{proof}

\end{document}